\documentclass[a4paper,11pt]{article}
\usepackage[utf8]{inputenc}
\usepackage[T1]{fontenc}
\usepackage{lmodern}
\usepackage{comment}
\usepackage{amsthm,amsmath,amsfonts,amssymb,bbm,mathrsfs,stmaryrd}
\usepackage{mathtools}
\usepackage{enumitem}
\usepackage{url}
\usepackage{dsfont}
\usepackage{appendix}
\usepackage{amsthm}
\usepackage[dvipsnames]{xcolor}
\usepackage{graphicx}
\usepackage{esint}
\usepackage{braket}

\usepackage[colorlinks=true, linkcolor=NavyBlue, urlcolor=black, citecolor=NavyBlue,pdfstartview=FitH]{hyperref}

\usepackage[english]{babel}

\usepackage{caption,tikz,subfigure}

\usepackage[top=2.5cm, bottom=2.5cm, left=2.2cm, right=2.2cm]{geometry}

\makeatletter

\@addtoreset{equation}{section}
\makeatother

\setlist[enumerate,1]{label=(\roman*), font = \normalfont} 

\let\originalleft\left
\let\originalright\right
\renewcommand{\left}{\mathopen{}\mathclose\bgroup\originalleft}
\renewcommand{\right}{\aftergroup\egroup\originalright}

\newlength{\bibitemsep}
\newlength{\bibparskip}
\let\oldthebibliography\thebibliography
\renewcommand\thebibliography[1]{\oldthebibliography{#1}
	\setlength{\parskip}{\bibitemsep}
	\setlength{\itemsep}{\bibparskip}}

\newcommand{\N}{\mathbb{N}}

\newcommand{\R}{\mathbb{R}}
\newcommand{\C}{\mathbb{C}}
\newcommand{\F}{\mathbb{F}}
\newcommand{\HH}{\mathbb{H}}

\renewcommand{\P}{\mathbb{P}}
\newcommand{\E}{\mathbb{E}}

\newcommand{\1}{\mathbbm{1}}

\newcommand{\cC}{\mathcal{C}}

\newcommand{\cN}{\mathcal{N}}
\newcommand{\cR}{\mathcal{R}}
\newcommand{\cT}{\mathcal{T}}

\newcommand{\Ec}[1]{\mathbb{E} \left[#1\right]}

\newcommand{\Pp}[1]{\mathbb{P} \left(#1\right)}

\newcommand{\mc}[1]{\mathcal{#1}}

\newcommand{\ii}{\mathrm{i}}
\DeclareMathOperator{\re}{Re}
\DeclareMathOperator{\im}{Im}

\newcommand{\diff}{\mathop{}\mathopen{}\mathrm{d}}
\DeclareMathOperator{\Var}{Var}
\DeclareMathOperator{\Cov}{Cov}
\DeclareMathOperator{\supp}{supp}

\newcommand{\abs}[1]{\left\lvert#1\right\rvert}
\newcommand{\abso}[1]{\lvert#1\rvert}
\newcommand{\norme}[1]{\left\lVert#1\right\rVert}
\newcommand{\normeo}[1]{\lVert#1\rVert}

\newcommand{\blambda}{\boldsymbol\lambda}

\newcommand{\tend}[1]{\underset{#1}{\longrightarrow}}

\newcommand{\mrm}[1]{{\mathrm{#1}}}
\theoremstyle{plain}
\newtheorem{theorem}{Theorem}[section]
\newtheorem{proposition}[theorem]{Proposition}
\newtheorem{lemma}[theorem]{Lemma}
\newtheorem{corollary}[theorem]{Corollary}
\newtheorem{conjecture}[theorem]{Conjecture}

\theoremstyle{definition}
\newtheorem{definition}[theorem]{Definition}

\theoremstyle{remark}
\newtheorem{remark}[theorem]{Remark}

\title{CLT for $\beta$-ensembles with Freud weights, application to the KLS conjecture in Schatten balls}
\author{Charlie Dworaczek Guera\footnote{KTH Stockholm, \newline
		\textit{email:} chadg@kth.se}
		\and Ronan Memin\footnote{Institut de mathématiques de Toulouse, France \newline \textit{email:} ronan.memin@ens.fr}
		\and Michel Pain\footnote{Institut de Mathématiques de Toulouse (UMR5219), Université de Toulouse, CNRS; UPS, F-31062 Toulouse Cedex 9, France; \textit{Email:} michel.pain@math.univ-toulouse.fr}}

\begin{document}
	
	\maketitle
	
	\begin{abstract}
		In this paper, we are interested in the $\beta$-ensembles (or 1D log-gas) with Freud weights, namely with a potential of the form $\abs{x}^{p}$ with $p \geq 2$. Since this potential is not of class $\mc{C}^{3}$ when $p \in (2,3]$, most of the literature does not apply. In this singular setting, we prove a central limit theorem for linear statistics with general test-functions. Our strategy relies on establishing an optimal local law in the spirit of \cite{BouModPai2022}.
		
		Our results allow us to give a consistency check of the KLS conjecture for the uniform distributions on $p$-Schatten balls and the functions $f(X)=\mathrm{Tr}\left(X^r\right)^q$. 
		While the case $p>3$, $q=1$, $r=2$ was proven in \cite{DadFraGueZit2023}, we address in the present paper the case $p\geq2$, $q\geq1$ and $r\geq2$ an even integer. The proofs are based on a link between the moments of norms of uniform laws on $p$-Schatten balls and the $\beta$-ensembles with Freud weights.
	\end{abstract}
	
	{
		\hypersetup{linkcolor=black}
		\tableofcontents
	}
	
	\section{Introduction}
	
	\subsection{Setting of the problem}
	
	Let $\beta>0$, $p>0$ and $V(x) \coloneqq c_p \abs{x}^p$ for a constant $c_p>0$ to be fixed later. For any $N\geq 1$, the \textit{$\beta$-ensembles with Freud weights} is the probability distribution on $\R^N$ given by:
	\begin{equation} \label{eq:beta_ensembles_density}
		\diff \P(\lambda_1,\dots,\lambda_N)
		\coloneqq \frac{1}{Z_{N}}
		\prod_{1 \leq i < j \leq N} 
		\abs{\lambda_i - \lambda_j}^\beta \cdot
		e^{- \frac{\beta N}{2} \sum_{i=1}^N V(\lambda_i)} 
		\diff \lambda_1 \cdots \diff \lambda_N,
	\end{equation}
	where the \textit{partition function} $Z_N$ is defined by
	\begin{equation}\label{eq:fonction}
		Z_{N}\coloneqq\int_{\R^N}\prod_{1 \leq i < j \leq N} 
		\abs{\lambda_i - \lambda_j}^\beta \cdot
		e^{- \frac{\beta N}{2} \sum_{i=1}^N V(\lambda_i)} 
		\diff \lambda_1 \cdots \diff \lambda_N.
	\end{equation}
	Note that we omit the dependence on $N$ in the notation $\P$.
	Results in this paper cover only the case $p \geq 2$.
	For $\beta=1,2,4$ and quadratic $V$, this law is realized by the (renormalized by $\sqrt{N}$) spectrum of the classical random matrix ensembles GOE, GUE and GSE. For general $\beta>0$, it is realized as the joint law of the eigenvalues of certain tridiagonal random matrices \cite{dumitriu2002matrix,krishnapur2016universality}. This model, for general $V$ and $\beta$, has been extensively studied for decades. A central object in the study of the Gibbs measure $\P_N$ is the empirical measure $\mu_N$ defined by
	\begin{equation}
		\label{def:mesure_empirique}
		\mu_N\coloneqq\dfrac{1}{N}\sum_{i=1}^N\delta_{\lambda_i}.
	\end{equation}
 	For any continuous potential $V$ such that $V(x) \geq (1+\varepsilon)\log \abs{x}$ for $\abs{x}$ large enough,  the empirical measure converges weakly a.s.\@ towards a deterministic limit $\mu_V$, called the \textit{equilibrium measure} which is the unique minimizer among all probability measures $\mu$ of the energy functional:
	\begin{equation}\label{eq:functional E}
		\mc{E}(\mu)\coloneqq\int_\R V(x)\diff \mu(x)-\iint_{\R^2}\log\abs{x-y}\diff \mu(x)\diff \mu(y),
	\end{equation}
	see \cite{BouPasShc1995,johansson}, as well as \cite{arous1997large,AndGuiZei10} for a large deviation approach.
	Under the Freud potential $V(x)=c_p\abs{x}^{p}$ with $p>0$, the measure $\mu_V$ is known, see \cite[Lemma 4.1]{DadFraGueZit2023} which is slightly rewritten from \cite[Theorem IV.6.1]{saff2024logarithmic}: it is supported on $[-1,1]$, provided we choose
	\begin{equation}\label{def:cp}
		c_p\coloneqq\dfrac{\Gamma(\frac{p}{2})\Gamma(\frac{1}{2})}{\Gamma(\frac{p+1}{2})},
	\end{equation}
	and it is absolutely continuous with respect to the Lebesgue measure, with density given by: 
	\begin{equation}\label{def:mu_p}
		\dfrac{\diff \mu_V}{\diff x}(x) \coloneqq \1_{\abs{x}\leq1}\dfrac{p\abs{x}^{p-1}}{\pi}\int_{\abs{x}}^1\dfrac{y^{-p}}{\sqrt{1-y^2}}\diff y.
	\end{equation}
	This measure is known as the \textit{Ullman} distribution, see \cite[IV.5.1]{saff2024logarithmic} or \cite{KabProTha2020a}.

	A natural question, following the convergence of $\mu_N$, is to investigate the nature of the fluctuations of the empirical measure around its limit $\mu_V$, specifically by establishing a \textit{central limit theorem} (CLT). The relevant object for this purpose is:
	\begin{equation}\label{def: LN}
		L_N\coloneqq N\left(\mu_N-\mu_V\right).
	\end{equation}
	It can be shown that this object is asymptotically Gaussian in the following sense: for any smooth bounded function $f$:
	\begin{equation}\label{introeq:clt}
		L_N(f) \xrightarrow[N\rightarrow\infty]{\mrm{law}} \mrm{Gaussian}.
	\end{equation}
	Such a CLT has been first established in \cite{johansson} and has been then extended to various settings (with different assumptions on $V$ and $f$, or $f,\beta$ dependent on $N$), see Subsection \ref{subsec:literature}.
	However, all known results in the literature rely on regularity assumptions on $V$. For instance, the most refined result \cite{BekLebSer2018} requires $V\in\mc{C}^6$ and $f\in\mc{C}^4$. 
	
	In our setting, \textit{i.e.\@} the Freud potential with $p\geq 2$, $V$ is in general only $\mc{C}^{2}$ and not $\mc{C}^{3}$ at 0.
	The main goal of this article is to establish such a CLT for this model, for a general test-function $f$ and for this potential that is smooth everywhere except at one point, where it presents a singularity (whose third derivative is locally integrable). This classically yields the subleading-order term of the \textit{free energy} $\log Z_N$. One of the main ingredients of the proof of the CLT is to establish an optimal \textit{local law}. 
	Our motivation is twofold. Firstly, we aim at developing a new approach to prove CLT for more singular potentials. 
	Secondly, the particular choice of Freud weights has applications in asymptotic convex geometry: indeed, proving a central limit theorem for $f(x)=x^r$ suffices to give a check on the so-called \textit{Kannan--Lovász--Simonovits conjecture} for \textit{$p$-Schatten} unit balls with $p\geq2$, as detailed below in Subsection \ref{sec:KLSconjecture}. 
	
	This CLT has already been obtained when $p>3$ and with $f(x)=x^2$ in \cite{DadFraGueZit2023}, motivated by applications in convex geometry. 
	Their proof follows closely \cite{BekLebSer2018} with some local improvements required because $V$ is only $\cC^3$ in their case: the fact that the equilibrium measure is explicit shows that it has two additional levels of regularity compared to what is proved for general $\cC^3$ potential in \cite[Lemma~3.1]{BekLebSer2018}, and the fact that the test function $f$ considered is even results in gaining one level of regularity for the inverse of $f$ via the master operator.
	This strategy would not work for the case $p \in (2,3]$ considered here, or even for $p \in (3,4]$ but $f$ not even.
	Our approach is more self-contained, importing mainly the \textit{a priori} bounds and controls of outliers stated in Section~\ref{sec:a_priori}.
	
	\subsection{Main results}
	
	\subsubsection{About \texorpdfstring{$\beta$}{beta}-ensembles}
	Our main result is a central limit theorem, in the sense of the convergence of moments, for the linear statistics of $\mathcal{C}^3$ functions having at most exponential growth.
	\begin{theorem}[CLT]
		\label{thm:CLT}
		Let $p\geq 2$, and $f\in \mathcal{C}^3(\R)$ with at most exponential growth, \emph{i.e.\@} there exists $C>0$ such that $\abs{f(x)} \leq Ce^{C\abs{x}}$ for any $x \in \R$. Then, the following convergence holds in terms of moments:
		\begin{equation}\label{eq:thmclt}
			L_N(f) \xrightarrow[N\rightarrow\infty]{} \mc{N}\left(m_V(f),\sigma^2(f)\right) 
		\end{equation}
		where the mean and the variance of the limiting Gaussian distribution are given by:
		\begin{align*}
			m_V(f)&\coloneqq\left(\dfrac{1}{2}-\dfrac{1}{\beta}\right)\int_{-1}^{1}\Xi^{-1}[f]'(x)\diff\mu_V(x)\\
			\sigma^2(f)&\coloneqq\dfrac{1}{\beta\pi^2}\iint_{[-1,1]^2}\left(\dfrac{f(x)-f(y)}{x-y}\right)^2 \cdot \dfrac{1-xy}{\sqrt{1-x^2}\sqrt{1-y^2}}\diff x\diff y
		\end{align*}
		where the operator $\Xi^{-1}$ is given explicitly in Lemma~\ref{lem:inverse_master_op} (taking $\alpha=1$ for $p>2$ and $\alpha=0$ for $p=2$).
	\end{theorem}
	
	\begin{remark}\label{rem:CLT}
		\begin{itemize}
			\item Theorem \ref{thm:CLT} is a special case of Theorem \ref{thm:CLTalpha} when specializing the interpolating parameter $\alpha=1$ for $p>2$ and $\alpha=0$ for $p=2$.
			This result is of course already known when $p=2$.
			\item The hypothesis on the growth of $f$ in Theorem \ref{thm:CLT} will be used in Lemma \ref{lem:truncation} to replace $f$ by a truncated version, while preserving the convergence of moments. However one can deduce from Theorem \ref{thm:CLT} the weak convergence of the linear statistics $L_N(f)$ for any $f\in \mathcal{C}^3(\R)$ with arbitrary growth at infinity by truncating $f$ outside of the support of the equilibrium measure.
			\item The mean $m_V(f)$ can be rewritten, by Lemma \ref{lem:variance formula}, \begin{multline}\label{eq:moyenne}
                m_V(f) = \left(\frac{1}{2}-\frac{1}{\beta}\right)\bigg(\frac{1}{\pi^2}\int_{-1}^1 \frac{r'(x)}{r(x)}\int_{-1}^1\frac{f(t)-f(x)}{t-x}\frac{\diff t}{\sqrt{1-t^2}}\sqrt{1-x^2}\diff x\\ - \frac{f(1)+f(-1)}{2}+\frac{1}{\pi}\int_{-1}^1f(x)\frac{\diff x}{\sqrt{1-x^2}}\bigg).
            \end{multline}
            where $r(x)$ is defined in \eqref{def:r_generale} with $\alpha=1$.
        \end{itemize}
	\end{remark}
	
	It is classical that such a central limit theorem is enough deduce the subleading-order term for the partition function, see e.g.\@ \cite[Theorem 4.28]{Gui2019}. We state this consequence below and, to this end, we define the entropy of $\mu$, a probability measure absolutely continuous with respect to the Lebesgue measure, by
	\begin{equation} \label{eq:def_entropie}
	    \mrm{Ent}[\mu]\coloneqq-\int_\R\log\frac{\diff\mu(x)}{\diff x}\diff\mu(x).
	\end{equation}
	However, as observed in \cite{Sonn} shortly after the first version of this paper appeared, the following expansion can be deduced directly from \cite[Corollary 1.1]{LebléSerf} for any $p\geq 3/2$.
	
	\begin{corollary}[Next-order expansion of the free energy]\label{thm:subleadingZN}
		Let $p\geq2$, then as $N\rightarrow\infty$,
		\begin{equation}
		\dfrac{1}{N^{2}\beta}\log Z_N=-\dfrac{1}{2}\left(\log 2+\dfrac{3}{2p}\right)+\dfrac{\log N}{2N}+\dfrac{F^{\{-1\}}}{N}+o\left(\dfrac{1}{N}\right) 
		\end{equation}
		with
		\begin{align*}
		F^{\{-1\}} 
		& \coloneqq F_G^{\{-1\}}+\left(\dfrac{1}{\beta}-\dfrac{1}{2}\right)\left(\mrm{Ent}[\mu_V]-\log\pi+\dfrac{1}{2}\right), \\
		F^{\{-1\}}_G&\coloneqq\left(\dfrac{1}{2}-\dfrac{1}{\beta}\right)\left(\log\dfrac{\beta}{2}+\log 2\right)-\dfrac{1}{2\beta}-\dfrac{1}{4} +\dfrac{1}{\beta}\log\dfrac{2\pi}{\Gamma\left(\beta/2\right)}.
		\end{align*}
	\end{corollary}
	
	\begin{remark}
		In \cite{BorGui2013,borot2024asymptotic}, the authors proved a full $1/N$ expansion for the free energy in the case of an analytic potential $V$. Due to our singular setting, we are only able to access the subleading order term. 
	\end{remark}
	
	To state the local law, we introduce the Stieltjes transforms:
	\begin{equation}\label{def: s_N, s_V}
		s_N(z)=\int_\R\dfrac{\diff\mu_N(\lambda)}{\lambda-z} = \dfrac{1}{N}\sum_{k=1}^N\dfrac{1}{\lambda_k-z},\hspace{1cm}s_{V}(z)=\int_\R\dfrac{\diff\mu_V(\lambda)}{\lambda-z},
	\end{equation}
	\noindent which is a well defined expression as soon as $z\in \C\setminus\R$ for $s_N(z)$ and $z\in \C\setminus[-1,1]$ for $s_V(z)$. 
	The $N$-dependent function $s_N$ is called the empirical Stieltjes transform and $s_V$ is called the equilibrium Stieltjes transform.
	\begin{theorem}[Optimal local law] \label{thm:local_law}
	    Let $p \geq 2$.
		There exist $C,\delta_0>0$ such that the following holds, with $\cR \coloneqq [-1-\delta_0,1+\delta_0]+\ii(0,\delta_0]$: 
		for any $N \geq 1$, $q\geq1$ and $z=x+\ii y\in \cR$,
		\[
		\E \left[\left|s_N(z)-s_{V}(z)\right|^q\right] \leq \frac{(Cq^{64})^{q}}{(Ny)^q},
		\]
		and, if moreover $\abs{x} \leq 1 + y$,
		\[
		\E \left[\left|s_N(z)-s_{V}(z)\right|^q\right] \leq \frac{(Cq)^{q/2}}{(Ny)^q} + \frac{(Cq^{32})^{q}}{(N \abs{z} \abs{z^2-1}^{1/2})^q}.
		\]
	\end{theorem}
	
	\begin{remark}
		\begin{itemize}
			\item As $\im{z}$ goes to zero, the imaginary part of the empirical Stieltjes transform at $z$ can be seen as a smooth approximation to the empirical measure in an interval centered at $\re(z)$ of size $\sim\im{z}$.
			\item If $p$ is even, then $V$ is analytic and this local law was already obtained in \cite{BouModPai2022}, with better bounds.
			\item 
			A local law is an important tool providing strong concentration results: via the use of the Hellfer--Sjöstrand formula one can deduce bounds on linear statistics $L_N(f)$ (see Section~\ref{sec:HS}) or rigidity estimates 
			(see e.g.\@ \cite[Lemma 3.8]{BouModPai2022}). 
			A rigidity estimate states that, with overwhelming probability, for each $1\leq k \leq N$, the $k$-th smallest particle is at distance at most $a_N N^{-2/3} [k\wedge(N-k+1)]^{-1/3}$ of the $k$-th $N$-quantile of $\mu_V$, where $a_N \to \infty$ is the precision of the rigidity estimate.
			Such a rigidity estimate was established for $V \in \cC^4(\R)$ and $a_N = N^\varepsilon$ for arbitrary $\varepsilon>0$ in \cite{BouErdYau2014}.
			Following the argument in \cite[Lemma 3.8]{BouModPai2022}, Theorem~\ref{thm:local_law} yields a rigidity estimate with $a_N = (\log N)^\ell$ for $\ell$ large enough.
			
			\item This local law is called optimal in the sense that it proves $s_N(z)-s_{V}(z) = O(1/(Ny))$, with a numerator which does not include a factor diverging as $N \to \infty$. We do not claim that the dependence of the numerator on $q$ is optimal, and do not try to optimize it in the proof. Some improvements in this direction could be made, but to the cost of additional technicalities that we choose to avoid.
		\end{itemize}
	\end{remark}
	\subsubsection{About the KLS conjecture for \texorpdfstring{$p$}{p}-Schatten balls}\label{sec:KLSconjecture}
	As briefly mentioned above, the CLT for the $\beta$-ensembles allows us to give a consistency check on the Kannan--Lovász--Simonovits (KLS) conjecture for the $p$-Schatten balls, $p\geq2$, that we describe below. 
	
	In convex geometry, and more specifically in the context of the study of high-dimensional convex bodies $K\subset\R^d$, this conjecture has been open for three decades \cite{kannan1995isoperimetric} and takes its roots from the study of high-dimensional sampling algorithm, see \cite{alonso2015approaching,lee2018kannan} for a review. The conjecture posits that:
	
	\begin{conjecture}[KLS conjecture]\label{conj:KLS}
		There exists a universal constant $\mathbf{C}_\mrm{univ}>0$, such that for all $d\geq 1$, $f$ smooth on $\R^d$ and $\mu$ an absolutely continuous with respect to the Lebesgue measure and log-concave probability measure on $\R^d$:
		\begin{equation}\label{eq:varianceconj}
			\Var_\mu f\leq \mathbf{C}_\mrm{univ}\cdot\lambda_\mu \cdot \mathbb{E}_\mu\left[ \|\nabla f\|^{2}_2\right]
		\end{equation}
		where $\lambda_\mu=\sup_{\theta\in\mathbb{S}^{d-1}}\Var_\mu[\braket{\theta,X}]=\|\Cov(\mu)\|_\mrm{op}$ denotes the largest eigenvalue of the covariance matrix of $\mu$, given by $\Cov(\mu)\coloneqq\left(\int_{\R^d}x_ix_j\diff\mu(x)-\int_{\R^d}x_i\diff\mu(x)\int_{\R^d}x_j\diff\mu(x)\right)_{1\leq i,j\leq d}$.
	\end{conjecture}
	
	The KLS conjecture implies the variance (or thin-shell) conjecture (\textit{i.e.\@} the previous conjecture in the special case $f(X)=\|X\|_{\mrm{HS}}^2$) and the hyperplane conjecture (or Bourgain's slicing problem). These two conjectures have been studied a lot and finally established in the recent papers \cite{klartag2024affirmative,klartag2025thin}.
	
	The thin-shell conjecture takes its roots in the papers \cite{anttila2003central,bobkov2003central} where the authors studied the rate of convergence of convex bodies around the normal distribution, known as the central limit problem for isotropic convex bodies. It can also be shown that a log-concave probability measure $\mu$ satisfying the variance conjecture satisfies a strong concentration inequality, namely:
	\[
	\mu\left( \left\{ x\in\R^d : \abs{\|x\|-(\E_\mu\|X\|^2)^{1/2}}>t(\E_\mu\|X\|^2)^{1/2} \right\}\right)
	\leq 2 \exp \left( -C\sqrt{t}\frac{(\E_\mu\|X\|^2)^{1/4}}{\sqrt{\lambda_\mu}} \right)
	\]
	for a constant $C>0$ and for all $t>0$.
	
	Even though, the KLS conjecture is not known in general, it has been established in certain cases, \textit{e.g.\@} for the uniform measures on the simplex and $\ell_N^p$ balls, see \cite{MR2532220} but also on the so-called Orlicz balls \cite{kolesnikov2018kls,barthe2023volume} defined as
	\[
	\left\{x\in\R^d : \sum_{i=1}^pf_i(\|x\|)\leq1\right\}
	\]
	where the $f_i$'s are called Young functions, \textit{i.e.\@} $f\colon\R_+\rightarrow\R_+$ increasing convex functions verifying $f(0)=0$. Our goal here is to study the case of $p$-Schatten balls in the spirit of \cite{DadFraGueZit2023}.
	
	Let $\mathbb{F}\coloneqq\R,\C$ or $\HH$ and $M\in\mc{M}_N(\F)$ be a $N\times N$ matrix with entries in $\F$. Denote its singular values by $(s_1(M),\dots,s_N(M))$, \textit{i.e.\@} the eigenvalues of $\sqrt{M^{*}M}$. We define its $p$-Schatten norm for $p\in[1,\infty)$, denoted $\sigma_p(M)$ by
	\begin{equation}\label{eq:schattennorm}
		\sigma_p(M)\coloneqq\left(\sum_{i=1}^Ns_i(M)^p\right)^{1/p}.
	\end{equation}
	On the space $E=\mc{M}_N(\F)_{\mrm{s.a.}}$ of self-adjoint matrices, $\sigma_p(M)$ reduces to the $\ell^p$ norm of the eigenvalues.
	The couple $S_p^N\coloneqq\left(E,\sigma_p\right)$ defines a normed-vector space and we denote by $B_E(S_p^N)$ its unit ball
	$$B_E(S_p^N)\coloneqq\left\{M\in\mc{M}_N(\F)_{\mrm{s.a.}}, \sigma_p(M)\leq1\right\}.$$
	Note that when $p=2$, $\sigma_2(M)=\|M\|_{\mrm{HS}}\coloneqq\sqrt{\mrm{tr}(M^{*}M)}$ and we recover the Euclidean structure.
	
	A link between the $\beta$-ensembles and the unit balls for the $p$-Schatten norms exists and is encoded in the following lemma.
	\begin{lemma}[Weyl's integration formula, \cite{raymond1984volume,guedon2007concentration}]\label{lem:straymond}
		For any $p\geq1$ and any $F\colon \R^N\rightarrow\R_+$ symmetric and positively homogeneous of degree $k$, we have:
		\begin{multline*}
			\int_{B_E(S_p^N)}F\left(s_1(M),\dots,s_N(M)\right) \diff M
			\\=\dfrac{c_N}{\Gamma\left(1+\frac{d_N+k}{p}\right)}\left(\frac{N\beta}{2}c_p\right)^{\frac{d_N+k}{p}}\int_{\R^N}F(x_1,\dots,x_N)\prod_{i<j}^N|x_i-x_j|^\beta \cdot \prod_{i=1}^Ne^{-\frac{N\beta}{2}c_p|x_i|^p}\diff x_i,
		\end{multline*}
		where $c_p$ is defined in \eqref{def:cp}, $c_N\coloneqq \abs{\mc{U}_N(\F)}/(N!\abs{\mc{U}_1(\F)}^N)$ with $\mc{U}_N(\F)$ the unitary group over $\F$, $\beta=1,2,4$ according to $\F=\R,\C,\HH$, and $d_N\coloneqq\frac{\beta}{2}N(N-1)+N$.
	\end{lemma}
	
	As an application with $F \equiv 1$, one gets the following relationship between the partition function $Z_N$ introduced in \eqref{eq:fonction} and the volume of $B_E(S_p^N)$:
	\begin{equation} \label{eq:ball_and_Z}
	    \abso{B_E(S_p^N)} = \dfrac{c_N}{\Gamma\left(1+\frac{d_N}{p}\right)} \left(\frac{N\beta}{2}c_p\right)^{\frac{d_N}{p}} Z_N.
	\end{equation}
	Hence, the following asymptotic expansion of $\log \abso{B_E(S_p^N)}$, whose proof is given in Section \ref{subsec:corballs}, follows directly from Corollary \ref{thm:subleadingZN}.
	The volume of Schatten balls, both in $\mc{M}_N(\F)$ and in $\mc{M}_N(\F)_{\mrm{s.a.}}$, was first studied in \cite{raymond1984volume} which identified the first two orders of the expansion below, without identifying explicitly the constant of the 2nd order (except for $p=2$). This was done more recently in \cite{KabProTha2020a,KabProTha2020b}, and recovered in \cite{DadFraGueZit2023}. 
	The following expansion, up to the 4th order, was also obtained in \cite{Sonn} for any $p \geq 3/2$ shortly after the first version of this paper appeared.
	
	
\begin{corollary}[Volume of $p$-Schatten balls]\label{cor:volume_balls}
		As $N\xrightarrow[]{}\infty$, $\beta=1,2,4$, $p\geq2$ it holds that:
		
		$$ \log\abso{B_E(S_p^N)} = aN^2\log N + bN^2 +cN\log N + dN +o(N),$$
		where
		$$ a=-\frac{\beta}{2}\left(\frac{1}{p}+\frac{1}{2}\right),\quad b=\beta\left(\frac{\log(pc_p)}{2p}+\frac{1}{4}\left(\frac{3}{2}+\log\frac{\pi}{\beta}\right)-\frac{1}{4p} \right),\quad c = \frac{\beta-2}{2}\left(\frac{1}{p}+\frac{1}{2}\right),$$
		$$d=\frac{2-\beta}{2p}\log(pc_p)+\beta F^{\{-1\}}+\log\Gamma(\beta/2) + \frac{1}{2} +\frac{\beta}{4}\left(1-\log(\pi\beta)\right)-\dfrac{1}{2}\log\frac{4\pi}{\beta},$$
		where $c_p$ is given by \eqref{def:cp} and $F^{\{-1\}}$ by Corollary \ref{thm:subleadingZN}.
	\end{corollary}

	
	Our main motivation concerning convex geometry is the following result, which provides a consistency check of the KLS conjecture in the context of $p$-Schatten balls for a specific class of functions $f$ and with an explicit constant.
	The CLT in Theorem~\ref{thm:CLT} is the main ingredient of the proof and it does not seem that this result is deducible only from the free energy expansion stated in Corollary~\ref{thm:subleadingZN} or in \cite{Sonn}.
	
	\begin{theorem}\label{thm:KLS}
		For all $p\geq2$, $r\geq 2$ an even integer and $q\in [1,\infty)$, the KLS conjecture, see Conjecture \ref{conj:KLS}, holds for $\mu$ uniform laws on $B_E(S_p^N)$ with $\F=\R, \C,\HH$, $N$ large enough (depending on $p,q,r$) and $f(X)=\mrm{Tr}(X^r)^q$ with constant $4$, \emph{i.e.\@}
	\begin{equation*}
			\Var_\mu f\leq 4\cdot\lambda_\mu \cdot \mathbb{E}_\mu \left[\|\nabla f\|^{2}_2\right].
		\end{equation*}
	\end{theorem}
	
	\subsection{Connection with the literature}\label{subsec:literature}
	
	\paragraph{Central limit theorems in $\beta$-ensembles.}
	The study of the fluctuations of the $\beta$-ensembles was initiated in the seminal paper \cite{johansson} for polynomial $V$. These results were later extended to the case of analytic $V$ in the one cut regime for $f$ analytic \cite{BorGui2013} (see also \cite{PasShc2011} for a partial result) and in the multi-cut setting for $f$ analytic \cite{borot2024asymptotic} and for for $f\in\cC^6$ \cite{Shc2}. 
	The regularity assumptions on $V$ and $f$ were relaxed in \cite{BekLebSer2018}, which covers the case $V \in \cC^6$ and $f\in\cC^4$ in the multi-cut setting, relying on the master operator method which is one of the two main ingredients in this paper.
	They also cover some critical cases for a restricted class of test functions.
	The currently best-known CLT in terms of $f$, requiring $f\in\mathsf{H}^{1/2+\varepsilon}$, has been established for Wigner matrices in \cite{landon2022almost}.
	Quantitative CLTs were obtained, via Stein's method, for $V \in \cC^8$ and $f\in\cC^9$ in \cite{LamLedWeb2019}, and for $V \in \cC^7$ and $f \in \cC^{1,\varepsilon}$ \cite{angst2024sharp}.
	In \cite{BouModPai2022}, a CLT is proved with analytic $V$ for singular test functions of the type $f(x) = \1_{x \leq x_0}$ or $f(x) = \log \abs{x-x_0}$, which require an additional renormalization compared to previously mentioned cases.
	A key tool in their paper is an optimal local law proved via loop equations: this method is our second main ingredient here.
	CLTs at mesoscopic scales have also been developed, notably in \cite{bekerman2018mesoscopic,peilen2024local}. Additionally, the high-temperature regime ($\beta=\beta_N\rightarrow0$) has been studied in \cite{benaych2015poisson,nakano2018gaussian,hardy2021clt,dworaczek2024clt,MazzucaMeminCLT}.

	
	\paragraph{Transition for bulk statistics at $p=1$.}
	For $\beta$-ensembles with Freud weights, it has been shown that the local behavior of the particles $\{\lambda_i\}_{i=1}^N$ under the distribution $\P_N$ undergoes a transition at $p=1$, at least in the case $\beta=2$ \cite{claeys2023weak}. This confirms earlier conjectures \cite{canali1995nonuniversality}, suggesting a dichotomy in the local statistics. 
	When $p>1$ called the \textit{strong confinement} regime, the equilibrium density \eqref{def:mu_p} is still bounded and the local statistics are described asymptotically by the Sine-$\beta$ process. In contrast, for $p<1$ called the \textit{weak confinement} regime, the density diverges at the origin and the local statistics are governed by a different, non-universal kernel.
	Finally, in the critical case $p=1$, the density diverges logarithmically at the origin but surprisingly the local limit is still the Sine-$\beta$ process.

	\paragraph{Multicriticality and quantum chromodynamics.}
	The study of random matrix models at criticality, \textit{i.e.\@} when the equilibrium density vanishes in the bulk, has attracted considerable interest in quantum field theories, see \cite{di19952d,verbaarschot2000random}. Owing to universality, understanding the microscopic behaviour near such critical points can lead to several predictions regarding some features of the spectrum of the Dirac operator in quantum chromodynamics in the critical universality class. It is well known that polynomial potentials $V_\mrm{pol}$ can generate  critical $\beta$-ensembles in which the equilibrium density vanishes at the origin with an even integer exponent $\abs{\lambda}^{2m}$ with $m$ an integer, see \textit{e.g.\@} \cite{bleher2003double,claeys2006universality,BekLebSer2018}. However, this construction only yields polynomial vanishing with integer powers. To obtain a richer family of critical exponents, specifically real non-integer powers, a mixture of Freud weight and polynomial potentials was proposed in \cite{akemann2002new,janik2002new}. Indeed taking $V(x)=a\abs{x}^{p}-bx^2$ for $2<p<3$ and constants $a,b$ finely tuned, the equilibrium density vanishes as $\abs{\lambda}^{p-1}$ thus exhibiting a real exponent. By further adding polynomial terms and increasing the power of the Freud weight, one can in fact generate arbitrary exponent.
	
	\subsection{Definitions and strategy}
	\label{sec:strategy}
	
	Our goal is twofold. We want to obtain the CLT and the subleading terms in the asymptotic expansion of $\log Z_N$. 
	Because of Lemma \ref{lem:interpolZN}, by working with the general potential $V_\alpha$ defined by linearly interpolating between $V$ and the quadratic potential as follows, for $p>2$,
	\begin{equation}\label{def: Valpha}
		V_\alpha(x)\coloneqq \alpha c_p\abs{x}^p+(1-\alpha)2x^2,
	\end{equation}
	and denoting 
	\begin{equation} \label{def:P_alpha}
		\diff \P_\alpha(\lambda_1,\dots,\lambda_N)
		\coloneqq \frac{1}{Z_{N,\alpha}}
		\prod_{1 \leq i < j \leq N} 
		\abs{\lambda_i - \lambda_j}^\beta \cdot
		e^{- \frac{\beta N}{2} \sum_{i=1}^N V_\alpha(\lambda_i)} 
		\diff \lambda_1 \cdots \diff \lambda_N,
	\end{equation}
	we can obtain the two results at once. Indeed, obtaining the CLT under $V_\alpha$ for general $\alpha\in[0,1]$ (in terms of the convergence of moments) implies Theorem \ref{thm:CLT} by taking $\alpha=1$; and one obtains Corollary \ref{thm:subleadingZN} by integrating over $\alpha$ the limits of $\E_\alpha\left[L_N(\partial_\alpha V_\alpha)\right]$ and using the asymptotic for the Selberg integral for $\alpha=1$, see Section \ref{sec:logZN}. Prior to explaining the strategy, we give the definitions that are going to be the main tools of our analysis.
	
	\paragraph{The equilibrium measure.}
	We denote by $\mu_{V_\alpha}$ the equilibrium measure under $V_\alpha$ (recall it is defined as the unique minimizer of \eqref{eq:functional E} with $V_\alpha$ instead of $V$).
	By \cite[Theorems~I.1.5 and I.3.3]{saff2024logarithmic}, this measure is characterized by the existence of a constant $C_\alpha \in \R$ such that the associated \textit{effective potential} is nonnegative outside of the support:
	\begin{equation}\label{eq:characterization muValpha}
		V_{\mrm{eff},\alpha}(x)\coloneqq \dfrac{V_\alpha(x)}{2}-\int_\R\log\abs{x-y}\diff\mu_{V_\alpha}(y)-C_\alpha
		\begin{cases}
			=0 & \text{on } \supp \mu_{V_\alpha}, \\ 
			\geq 0 & \text{elsewhere}.
		\end{cases}
	\end{equation}
	Recall that, for $\alpha=1$, $\mu_{V_1} = \mu_V$ is given in \eqref{def:mu_p}. On the other hand, for $\alpha=0$, the equilibrium measure is the semi-circular distribution on $[-1,1]$.
	Since they have the same support, it follows from the characterization above that $\mu_{V_\alpha}$ is the linear interpolation between these two measures: hence, it is absolutely continuous with respect to Lebesgue measure, supported on $[-1,1]$ and its density is given for all $x\in[-1,1]$ by:
	\begin{equation}\label{def: mu_Valpha}
		\dfrac{\diff\mu_{V_\alpha}}{\diff x}(x)\coloneqq \alpha\dfrac{\diff\mu_V}{\diff x}(x)+\dfrac{2(1-\alpha)}{\pi}\sigma(x),
	\end{equation}
	where we defined, for $x \in [-1,1]$,
	\begin{equation}
	    \sigma(x)\coloneqq\sqrt{1-x^2}.
	\end{equation}
	We now gather some basic facts on $\mu_{V_\alpha}$.
	Differentiating the relationship \eqref{eq:characterization muValpha} on $[-1,1]$, one gets the following formula (see e.g.\@ \cite[Theorem~11.2.1]{PasShc2011})
	\begin{equation} \label{eq:equilibrium_relation}
	    \frac{V_\alpha'(x)}{2} 
	    - \fint_{-1}^1 \frac{1}{x-y} \diff\mu_{V_\alpha}(y) = 0,
	\end{equation}
	where $\fint$ denotes the Cauchy principal value of the integral defined in \eqref{def:PV}.
    Integrating this equality w.r.t.\@ $h(x)\diff\mu_{V_\alpha}(x)$ for some Hölder continuous function $h \colon \R \to \R$ and symmetrizing the second term, one can deduce the following integral identity that we will use several times:
	\begin{equation}\label{eq:identiteintegralmuvalpha}
		\int_{-1}^1 V'_\alpha(x)h(x)\diff\mu_{V_\alpha}(x)
		-\iint_{[-1,1]^2}\dfrac{h(x)-h(y)}{x-y}\diff\mu_{V_\alpha}(y)\diff\mu_{V_\alpha}(x)=0.
	\end{equation}
	Another expression for the equilibrium measure is the following (see \cite[Theorem 11.2.4]{PasShc2011})
	\begin{equation}\label{eq:mu=rsigma}
		\dfrac{\diff\mu_{V_\alpha}}{\diff x}(x)=\dfrac{\sigma(x)r_\alpha(x)}{\pi} \1_{[-1,1]}(x),
	\end{equation}
	where, for all $\alpha\in[0,1]$ and $x\in \R$, we set
	\begin{equation}
		\label{def:r_generale}
		r_\alpha(x)\coloneqq \frac{1}{2\pi}\int_{-1}^1\frac{V_\alpha'(t)-V_\alpha'(x)}{t-x}\frac{\diff t}{\sigma(t)}.
	\end{equation}
	From the formula \eqref{def:mu_p} for the equilibrium measure at $\alpha=1$, and expanding $(1-y^2)^{-1/2}$, one gets for $x\in[-1,1]$ the expression, valid for all $p>0$
	\begin{equation}\label{def:r}
		r_\alpha(x)
		= \frac{\alpha p}{\sigma(x)} \left( \abs{x}^{p-1}\big(A_p - B_p\log|x|\big)-\sum_{n=0}^{+\infty}a_n(p)x^{2n}\1_{\{2n+1\neq p\}}\right)
		+2(1-\alpha),
	\end{equation}
	where 
	\[
	    A_p \coloneqq \sum_{n=0}^{+\infty}\1_{\{2n+1\neq p\}}\frac{\Gamma(n+1/2)}{n!(2n+1-p)\Gamma(1/2)} ,\qquad 
	    B_p \coloneqq \frac{\Gamma(p/2)}{((p-1)/2)!\Gamma(1/2)}\1_{\{p\in 2\N+1\}},
	\]
	\[
	    a_n(p)=\frac{\Gamma(n+1/2)}{n!(2n-p+1)\Gamma(1/2)}.
	\]
	Finally, we emphasize that $\mu_{V_\alpha}$ is symmetric w.r.t.\@ 0 because $V_\alpha$ is even (see \eqref{def: mu_Valpha}.

	\paragraph{Strategy.}
	We use the master operator approach already used to prove CLTs in 
	\cite{BorGui2013,Shc2,borot2024asymptotic,BekLebSer2018}.
	The \textit{master operator} $\Xi_\alpha$ acts on functions $\psi \in \cC^1(\R)$ and is defined as
	\begin{equation}\label{eq:def_master_op}
		\Xi_\alpha[\psi](x)
		\coloneqq - \frac{1}{2} \psi(x) V_\alpha'(x) + \int_{-1}^1 \frac{\psi(x)-\psi(t)}{x-t} \diff \mu_{V_\alpha}(t),
		\qquad x \in \R.
	\end{equation}
	This operator is invertible (see \cite[Lemma~3.2]{BekFigGui2015}, \cite[Lemma~3.3]{BekLebSer2018} and Lemma~\ref{lem:inverse_master_op} here) \textit{via} Tricomi's formula \cite{tricomi1957integral}, and the understanding of the function $\psi_\alpha\coloneqq\Xi_\alpha^{-1}[f]$ and its derivatives (for $f$ a smooth function) is of prime importance. Indeed, using a classic trick based on the change of variables 
	$\lambda_k=y_k+\psi(y_k)/N$ for $1\leq k \leq N$ in the partition function with respect to $V_\alpha$ and a Taylor expansion, one can get the following estimate:
	\begin{equation}\label{eq:introeq laplace}
		\E_\alpha\left[\exp\left(L_N(f)+\dfrac{A_N(\psi_\alpha)}{N}+\mrm{error}\right)\right]=\exp\left(m_{V_\alpha}(f)+\dfrac{\sigma^2(f)}{2}\right).
	\end{equation}
	Above
	$m_{V_\alpha}(f)$ and $\sigma^2(f)$ are  respectively the mean and the variance of the limiting Gaussian distribution defined in Theorem \ref{thm:CLTalpha}, the \textit{anisotropy} term $A_N(\psi_\alpha)$ is given by
	\[
	A_N(\psi_\alpha)=N^2\iint_{\R^2} \frac{\psi_\alpha(\lambda)-\psi_\alpha(\lambda')}{\lambda-\lambda'} 
	\diff (\mu_N-\mu_{V_\alpha})(\lambda) \diff (\mu_N-\mu_{V_\alpha})(\lambda'),
	\]
	and the error term can be showed to be negligible by applying the law of large numbers, \textit{i.e.\@} the convergence of $\mu_N$ towards $\mu_{V_\alpha}$.
    Thus, the following convergence in terms of moments holds
	\begin{equation}
		L_N(f)+\dfrac{A_N(\psi_\alpha)}{N} \xrightarrow[N\rightarrow\infty]{} \mc{N}\left(m_{V_\alpha}(f),\sigma^2(f)\right).
	\end{equation}
	Therefore, proving that $A_N(\psi_\alpha)/N$ converges to 0 in terms of moments is enough to conclude on the CLT.

	
	Proofs of CLTs via this method usually start with some \textit{a priori} concentration bounds obtained by studying the energy functional in \eqref{eq:functional E}, see \textit{e.g.\@} \cite[Corollary 2.5]{BekLebSer2018} or \cite[Corollary~4.16]{Gui2019} that we use here. This bounds show that $L_N(f) = O(\sqrt{N})$ and therefore $A_N(f) = O(N)$ for smooth enough $f$ (with sometimes additional logarithmic factors that we neglect in this discussion), and therefore are not enough to conclude.
	A one-step bootstrap is necessary: with this input in hand, coming back to \eqref{eq:introeq laplace} shows that $L_N(f) = O(1)$ and so $A_N(f) = O(1)$, which is enough to conclude that $A_N(\psi_\alpha)/N$ converges to 0. The improved concentration bounds and then the CLT usually require more and more regularity on the test function $f$, because the function $\psi_\alpha = \Xi_\alpha^{-1}[f]$ is less regular than $f$.
	
	In our case, the main difficulty is that, if $p \in (2,3)$, the function $\psi_\alpha$ is only $\cC^2$, with $\psi_\alpha'(x)$ exploding as $\abs{x}^{p-3}$ at the origin, even for $f \in \cC^\infty$ because this singularity is due to the singularity at 0 of the density of the equilibrium measure. This does not provide enough regularity even for applying the \textit{a priori} bound on $A_N(\psi_\alpha)$, which requires $\psi_\alpha''$ bounded in \cite[Corollary~4.16]{Gui2019} and could be improved to require only $\norme{\psi_\alpha''}_2 < \infty$ which does not hold either if $p \leq 5/2$.
	Hence, we need a way to reduce the regularity needed for our concentration bounds: this is done by proving local laws, namely an estimate on the moments of $s_N(z)-s_{V_\alpha}(z)$ where we recall \eqref{def: s_N, s_V}.
	Indeed, if we know that $s_N(z)-s_{V_\alpha}(z) = O(M/Ny)$ for some $M$ depending on $N$ and $z=x+\ii y$, then it follows from Helffer--Sjöstrand formula that $A_N(\psi_\alpha) = O(M^2)$ provided $\norme{\psi_\alpha''}_1 < \infty$, see Lemma~\ref{lem:HS_A_N} and Remark~\ref{rem:HS_A}.
	
	The proof of local laws relies on loop equations, which are relationships between moments of linear statistics obtained by integration by part.
	In \cite[Theorem 1.1]{BouModPai2022}, it has been observed that these loop equations can be combined in a way which is almost self-contained in terms of $s_N(z)$, up to a linear statistics $L_N(\varphi_z)$ where $\varphi_z(\lambda) = (V'(\lambda)-V'(z))/(\lambda - z)$ (note that this behaves well when $z$ approaches the axis, unlike the Stieltjes transform). 
	Their argument has two independent parts: first, using a contour integral trick which works if $V$ is analytic, they show that $L_N(\varphi_z) = O(1)$, then they analyze the rest of the loop equation to prove that $s_N(z)$ almost solves the same quadratic equation as $s_V(z)$ and therefore they are close from each other.
	The second part does not rely heavily on analyticity of $V$ and can be adapted to our context.
	An argument similar to the first part has already been applied to a potential $V \in \cC^4$ in \cite[Lemma~6.6]{BouErdYau2014} by proceeding by induction on the scales to obtain a (sub-optimal) local law, but it does not seem easy to adapt this approach to our less regular $V$.
	Instead, we simply replace the first part of the argument by bounds on $L_N(h_z)$ obtained by other means: \textit{a priori} bounds first, and then bounds obtained via the master operator approach.
	To summarize, this \textit{local law machinery} provides a tool which improves concentration bound, by lowering the regularity required for the test function, without improving the dependence on $N$.
	
	One subtlety in the local law argument that we would like to emphasize is that the function $\varphi_z$ mentioned above is also not regular enough to apply the \textit{a priori} bounds or the master operator approach to $L_N(\varphi_z)$. 
	To overcome this issue, we modify the loop equations in order to replace $\varphi_z$ by a function $f_z$ with an additional level of regularity, in exchange of a singular $1/z$ factor, see Remark~\ref{rem:loop_eq} for details. This argument is crucial to be able to cover the case $p\in(2,3]$ and the bootstrap argument would not work in that case without this regularization of the local law.
	
	Overall, the local law machinery is applied successively five times: a first time to obtain a local law from the \textit{a priori} bounds (Corollary \ref{cor:a_priori_local_law}) and then four times in the proof of the optimal local law in Section~\ref{sec:opt_local_law}. 
	Between each application, the bound on $L_N(f_z)$ is improved, either by applying the master operator approach and using the previous local law to bound $A_N(\psi_\alpha)$ to improve the dependence on $N$, or by applying the local law directly to bound $L_N(f_z)$ to improve the dependence on $z$.

	\subsection{Organization of the paper}
	
	In Section \ref{sec:LE and LL}, we state the loop equations and prove that from a certain concentration estimate, we can deduce a local law. We prove in Section \ref{sec:HS} that by using the Helffer--Sjöstrand formula, a local law translates into estimates on the moments of linear statistics and of the anisotropy. In Section \ref{sec:bounds_test_functions}, we state several bounds on some functions of interest and on a certain regularization of the test-functions considered. Section \ref{sec:bootstrap} is dedicated to establishing the first concentration bound from which we deduce a first local law. We then improve this local law iteratively, by establishing better and better concentration bounds, until obtaining an optimal local law. Finally, we prove respectively the CLT (Theorem \ref{thm:CLT}), the third order asymptotic expansion for the free energy (Corollary \ref{thm:subleadingZN}) and the check on the KLS conjecture (Theorem \ref{thm:KLS}) and the computation of the volume of the Schatten balls (Corollary \ref{cor:volume_balls}) in Sections \ref{sec:CLT}, \ref{sec:logZN} and \ref{sec:convex}. We gather in Appendix~\ref{app:technical estimate function} the proofs of the estimates stated in Section \ref{sec:bounds_test_functions}, and in Appendix \ref{app:eq_stieltjes} the proof of the stability lemma (Lemma \ref{lem:stability}), which is a key ingredient for proving the local laws of Section \ref{sec:LE and LL}.

    \subsection{Notation}

	Throughout the paper, $C$ and $c$ denote positive constants that can change between occurrences. They can depend on some other parameters in a way which is made clear in the statement of each result; in the proof of such a result allowed dependencies are the same as in the statement.
	
	We always consider the principal branch of the logarithm, extended to the negative real numbers by continuity from above, that is $\log(re^{\ii\theta}) = \log(r) + \ii\theta$ for any $r>0$ and $\theta \in (-\pi,\pi]$. We also consider the associated square root $\sqrt{z} = \exp(\frac{1}{2} \log z)$ for $z \in \C^*$ with continuous extension at 0. 
	For $f:\R\to \R$ with a singularity at $x_0$, we write 
	\begin{equation}
	    \label{def:PV}
	    \fint_\R f(x)\diff x \coloneqq \lim_{\varepsilon\to 0}\left(\int_{-\infty}^{x_0-\varepsilon}f(x)\diff x + \int_{x_0+\varepsilon}^{\infty}f(x)\diff x\right)
	\end{equation}
	for the Cauchy principal value, whenever the previous limit makes sense.
	
	We write $\R^* = \R \setminus \{0\}$ and, for $a \leq b$ integers, $\llbracket a,b \rrbracket = \{a,a+1,\dots,b\}$. For $x \in \R$, we write $x_+ = x \vee 0$.
	For $E \subset \R$, we denote by $\cC^k(E)$ the set of class $\cC^k$ real-valued functions $f$ on $E$ and write $\norme{f}_{\cC^k(E) }\coloneqq \sum_{i=0}^k \sup_{x\in E} \abso{f^{(i)}(x)}$, which is not a norm if $E$ is non compact because then it can be infinite.
	We simply write $\norme{f}_{\cC^k}$ if $E=\R$.
	Finally, we write $\cC^k_c(\R)$ for the set of functions in $\cC^k(\R)$ with compact support.
	
    \subsection*{Acknowledgements}
    The authors are grateful to Oliver Guédon, Alice Guionnet and Karol Kozlowski for suggesting this problem and sharing their first insights. They also thank Franck Barthe for helpful discussions on convex geometry, and Mathias Sonnleitner for pointing out a sign error in an earlier version of Corollary~\ref{thm:subleadingZN}. \\
	This project was partly funded by the LabEx CIMI - ANR-11-LABX-0040. \@C.D.G. acknowledges the support of the starting grant 2022-04882 from the Swedish Research Council and of the starting grant from the Ragnar Söderbergs Foundation. R.M. has received support under the Major Research Program of PSL Research University "Statistical Physics and Mathematics" launched by PSL Research University and implemented by ANR-10-IDEX-0001. 
	
	\section{Loop equations and local laws}
	\label{sec:LE and LL}
	
	\subsection{Loop equations}
	\label{subsec:loop_equations}
	
	Before stating the loop equations, we need to define some functions that are going to be central in our analysis.
	
	\begin{definition}\label{def:g,g analytic, f_z, Palpha, halpha}
		Let $p>2$.
		We define the function $g$ as follows:
		\begin{equation} \label{eq:def_g_R}
			\forall x \in \R, \quad 
			g(x) \coloneqq p c_p \abs{x}^p = xV'(x).
		\end{equation}
		We extend it pseudo-analytically on $\C$ as follows:
		\begin{equation} \label{eq:def_g_C}
			\forall z = x+\ii y \in \C, \quad
			g(z) \coloneqq g(x) + \ii y g'(x) - \frac{y^2}{2} g''(x) - \ii \frac{y^3}{6} g'''(x) \chi \left( \frac{y}{x} \right),
		\end{equation}
		where $\chi\in \cC_c^\infty(\R)$ is an even function such that $\1_{[-1/2,1/2]} \leq \chi \leq \1_{[-1,1]}$ and where the last term is meant to be 0 when $x=0$.
		We also define for all $\alpha\in[0,1]$ and $z\in\C$, the functions 
		\begin{equation}
		    g_\alpha(z)\coloneqq \alpha g(z)+4(1-\alpha)z^2
		\end{equation} 
		and $f_{\alpha,z}$ by setting:
		\begin{equation}\label{def:fz}
			\forall \lambda \in \R, \quad 
			f_{\alpha,z}(\lambda) \coloneqq \frac{g_\alpha(\lambda)-g_\alpha(z)}{\lambda-z}.
		\end{equation}
		Finally, we define the random variable $P_\alpha(z)$ for all $\alpha\in[0,1]$ and $z \in \C\setminus \R$ by:
		\begin{align} \label{eq:def_P}
			P_\alpha(z) & \coloneqq s_N(z)^2 + \frac{g_\alpha(z)}{z} s_N(z) + \frac{h_\alpha(z)}{z}, \\
			h_\alpha(z) & \coloneqq \int_{-1}^1 f_{\alpha,z}(t) \diff \mu_{V_\alpha}(t),
		\end{align}
		where $s_N$ is defined in \eqref{def: s_N, s_V} and $\mu_{V_\alpha}$ is defined in \eqref{def: mu_Valpha}.
	\end{definition}
	
	\begin{remark}\label{rem:g}
		Note that the pseudo-analytic extension of $g$ is well-defined since $p>2$ implies that $g\in\cC^2(\R) \cap \cC^\infty(\R^*)$.
		The truncation involving the function $\chi$ is needed for the last term because $g'''(x)$ diverges at 0 when $p \in (2,3)$. This definition makes $g$ a continuous function on $\C$; indeed, to prove the continuity at 0, one can note that, for any compact set $K \subset \C$, there exists $C>0$ such that, for any $z \in K$,
		\begin{equation} \label{eq:last_term_g}
			\abs{y^3 g'''(x) \chi \left( \frac{y}{x} \right)}
			\leq c_p p^4 \abs{y}^3 \abs{x}^{p-3} \1_{\abs{y}\leq\abs{x}} 
			\leq C \abs{y}^{p\wedge 3},
		\end{equation}
		using $\abs{x}^{p-3} \leq \abs{y}^{p-3}$ if $p<3$ and $\abs{x}^{p-3} \leq C$ if $p\geq 3$, and the RHS of \eqref{eq:last_term_g} vanishes as $z \to 0$.
	\end{remark}
	
	We are now ready to state the loop equations that are the main ingredient to derive the local laws. 	
	Denote by $\E_\alpha$ the expectation with respect to $\P_\alpha$, where we recall \eqref{def:P_alpha}.
	
	\begin{proposition}[Loop equations]\label{prop:loop equations}
		Let $p>2$. For all $N,q\geq1$, $\alpha\in[0,1]$, $z\in\C\setminus \R$, 
		\begin{align} 
			& \E_\alpha \left[ \left( 
			P_\alpha(z) + \frac{L_N(f_{\alpha,z})}{Nz} 
			+ \frac{1}{N}\left( \frac{2}{\beta}-1 \right) s_N'(z)
			\right) 
			P_\alpha(z)^{q-1} P_\alpha(\overline{z})^q \right] \nonumber \\
			& {} + \frac{2(q-1)}{\beta N^2 z} 
			\E_\alpha \left[ \left( 2s_N(z)+\frac{g_\alpha(z)}{z} \right)
			\left( \frac{1}{N}\sum_{k=1}^N \frac{\lambda_k}{(\lambda_k - z)^3} \right)
			P_\alpha(z)^{q-2} P_\alpha(\overline{z})^q \right] \nonumber \\
			& {} + \frac{2q}{\beta N^2 z} 
			\E_\alpha \left[ \left( 2s_N(\overline{z}) + \frac{g_\alpha(\overline{z})}{\overline{z}} \right)
			\left( \frac{1}{N} \sum_{k=1}^N \frac{\lambda_k}{(\lambda_k - \overline{z})^2(\lambda_k-z)} \right)
			P_\alpha(z)^{q-1} P_\alpha(\overline{z})^{q-1} \right]
			= 0. \label{eq:LE}
		\end{align}
	\end{proposition}
	
	\begin{proof}
		Let $f \colon \R \to \R$ and $G \colon \R^N \to \R$ be $\cC^1$ functions with at most polynomial growth at infinity.
		We write $\blambda = (\lambda_1,\dots,\lambda_N)$. Integrating by parts the factor $\partial_{\lambda_k} (f(\lambda_k) G(\blambda))$, we get
		\begin{align*}
			& \frac{2}{\beta N^2} \E_\alpha \left[ 
			\sum_{k=1}^N \left( 
			f'(\lambda_k) G(\blambda) 
			+ f(\lambda_k) \partial_{\lambda_k} G(\blambda) \right)
			\right] \\
			& = \frac{2}{\beta N^2} \frac{1}{Z_{N,\alpha}} 
			\int_{\R^N} \partial_{\lambda_k} \left( f(\lambda_k) G(\blambda) \right)
			\exp \left( \beta \sum_{1 \leq i < j \leq N} \log \abs{\lambda_i-\lambda_j} 
			- \frac{\beta N}{2} \sum_{i=1}^N V_\alpha(\lambda_i) \right)
			\diff \lambda_1 \cdots \diff \lambda_N \\
			& = \E_\alpha\left[ \left( 
			\frac{1}{N} \sum_{k=1}^N f(\lambda_k) V_\alpha'(\lambda_k)
			- \frac{1}{N^2} \sum_{k\neq\ell} \frac{f(\lambda_k)-f(\lambda_\ell)}{\lambda_k-\lambda_\ell} 
			\right) G(\blambda)\right],
		\end{align*}
		where the second term has been symmetrized noting that $\sum_{k\neq\ell} \frac{f(\lambda_k)}{\lambda_k-\lambda_\ell} = \frac{1}{2} \sum_{k\neq\ell} \frac{f(\lambda_k)-f(\lambda_\ell)}{\lambda_k-\lambda_\ell}$.
		Applying this to $f(\lambda) = 1/(\lambda-z)$, we get, for any $z \in \C \setminus \R$,
		\begin{equation} \label{eq:LE1}
			\E_\alpha \left[ \left( 
			s_N(z)^2
			+ \frac{1}{N} \sum_{k=1}^N \frac{V_\alpha'(\lambda_k)}{\lambda_k - z}
			+ \frac{1}{N}\left( \frac{2}{\beta}-1 \right) s'_N(z)
			\right) 
			G(\blambda) 
			\right] = \frac{2}{\beta N^2} \E_\alpha \left[ 
			\sum_{k=1}^N \frac{\partial_{\lambda_k} G(\blambda)}{\lambda_k-z} 
			\right].
		\end{equation}
		On the other hand, with $f(\lambda) = 1$, we get
		\begin{equation} \label{eq:LE2}
			\E_\alpha \left[ \left( 
			\frac{1}{N} \sum_{k=1}^N V_\alpha'(\lambda_k)
			\right) 
			G(\blambda) 
			\right] 
			= \frac{2}{\beta N^2} \E_\alpha \left[ 
			\sum_{k=1}^N \partial_{\lambda_k} G(\blambda)
			\right].
		\end{equation}
		Summing \eqref{eq:LE1} with \eqref{eq:LE2} divided by $z$, we get
		\begin{equation} \label{eq:LE3}
			\E_\alpha \left[ \left( 
			s_N(z)^2
			+ \frac{1}{N} \sum_{k=1}^N \frac{\lambda_k V_\alpha'(\lambda_k)}{z(\lambda_k - z)}
			+ \frac{1}{N}\left( \frac{2}{\beta}-1 \right) s_N'(z)
			\right) 
			G(\blambda) \right] 
			= \frac{2}{\beta N^2} \E_\alpha \left[ 
			\sum_{k=1}^N \frac{\lambda_k \partial_{\lambda_k} G(\blambda)}{z(\lambda_k - z)} 
			\right].
		\end{equation}
		Recalling the definition of $P_\alpha$ and $f_{\alpha,z}$ in Definition \ref{def:g,g analytic, f_z, Palpha, halpha} and of $L_N$ in \eqref{def: LN}, this can be then rewritten as
		\begin{equation} \label{eq:LE3'}
			\E_\alpha \left[ \left( 
			P_\alpha(z) + \frac{L_N(f_{\alpha,z})}{Nz} 
			+ \frac{1}{N}\left( \frac{2}{\beta}-1 \right) s_N'(z)
			\right) 
			G(\blambda) \right] 
			= \frac{2}{\beta N^2} \E_\alpha \left[ 
			\sum_{k=1}^N \frac{\lambda_k \partial_{\lambda_k} G(\blambda)}{z(\lambda_k - z)} 
			\right].
		\end{equation}
		Finally, we choose $G(\blambda) = P_\alpha(z)^{q-1} P_\alpha(\overline{z})^q$, and a direct computation on the sum on the right-hand side yields
		\begin{align*}
			\sum_{k=1}^N \frac{\lambda_k \partial_{\lambda_k} [P_\alpha(z)^{q-1} P_\alpha(\overline{z})^q]}{z(\lambda_k - z)}
			& = \frac{(q-1)}{z} P_\alpha(z)^{q-2} P_\alpha(\overline{z})^q 
			\left[-2s_N(z)-\dfrac{g_\alpha(z)}{z}\right]
			\frac{1}{N}\sum_{k=1}^N \frac{\lambda_k}{(\lambda_k - z)^3}
			\\ & \quad + \frac{q}{z} P_\alpha(z)^{q-1} P_\alpha(\overline{z})^{q-1} 
			\left[-2s_N(\overline{z})-\dfrac{g_\alpha(\overline{z})}{\overline{z}}\right]
			\frac{1}{N} \sum_{k=1}^N \frac{\lambda_k}{(\lambda_k - \overline{z})^2(\lambda_k-z)},
		\end{align*}
		and the result follows.
	\end{proof}
	
	\begin{remark} \label{rem:loop_eq}
		The loop equations used in \cite{BouModPai2022} to prove their local law are the one obtained by taking $f(\lambda) = 1/(\lambda-z)$ as in \eqref{eq:LE1} and $G(\blambda) = P_\alpha(z)^{q-1} P_\alpha(\overline{z})^q$. 
		Therefore, they differ from the loop equation above in \eqref{eq:LE} by the fact that $L_N(f_{\alpha,z})/(Nz)$ is replaced by 
		\[
			\frac{1}{N} L_N\left( \lambda \mapsto \frac{V_\alpha'(\lambda)-V_\alpha'(z)}{\lambda-z} \right),
		\]
		where $V_\alpha'(z)$ naturally makes sense in their context because the potential is analytic, as well as other differences in the 2nd and 3rd line of \eqref{eq:LE} which are not significant for the analysis. 
		Hence, the fact of combining the loop equations \eqref{eq:LE1} and \eqref{eq:LE2} to get \eqref{eq:LE3} allows us to change the test function appearing in the linear statistics, in order to gain regularity ($g_\alpha$ is $\cC^2$ whereas $V_\alpha'$ is only $\cC^1$) in exchange of a singular $1/z$ factor.
		This idea, which plays a key role in our proof, is inspired by an argument of a similar taste used in \cite{ForRahWit2017} in the context of Laguerre ensembles (see how (3.3) and (3.4) are combined into (3.8) there).
	\end{remark}
	
	\subsection{Local laws}
	\label{subsec:locallaws}
	
	In this section, our goal is to prove Proposition \eqref{prop:local_law_machinery}, which turns a bound on moments of $L_N(f_{\alpha,z})$ into a local law. 
	To this end, we adapt to our setting the arguments of \cite{BouModPai2022} that lead to their local law.
	
	Before stating the local law, we need the following lemmas on the equilibrium Stieltjes transform $s_{V_\alpha}$, whose proofs can be found in Appendix \ref{app:eq_stieltjes}. We recall that $s_N$ and $s_{V_\alpha}$ are defined in \eqref{def: s_N, s_V}.
	
	\begin{lemma} \label{lem:stieltjes}
		Let $p>2$.
		\begin{enumerate}
			\item\label{lemsta:point1} For all $z\in \C\setminus[-1,1]$, $s_{V_\alpha}(z)$ solves the quadratic equation
			\begin{equation} \label{eq:quadratic}
				X^2+\frac{g_\alpha(z)}{z}X + \frac{h_\alpha(z)}{z}=0,
			\end{equation}
			where $g_\alpha$ and $h_\alpha$ are defined in Definition \ref{def:g,g analytic, f_z, Palpha, halpha}.
			\item\label{lemsta:point2} Define for $z\in \C\setminus \{0\}$,
			\begin{equation}\label{eq:r'} 
				r_\alpha(z)=\dfrac{1}{2\pi z}\int_{-1}^{1}\dfrac{g_\alpha(t)-g_\alpha(z)}{t-z}\dfrac{\diff t}{\sigma(t)}.
			\end{equation} 
			This definition is consistent with the previous definition of $r_\alpha$ on $\R$ given in \eqref{def:r_generale}.
			\item\label{lemsta:point3} For all $z\in\C\setminus[-1,1]$,
			\begin{equation}
				\label{eq:s_formula}
				s_{V_\alpha}(z)=r_\alpha(z)b(z)-\dfrac{g_\alpha(z)}{2z},
			\end{equation}
			where $b(z) \coloneqq \sqrt{z+1} \sqrt{z-1}$.
			\item\label{lemsta:point4} For all $z\in\C\setminus[-1,1]$, the other solution of \eqref{eq:quadratic} is given by 
			\begin{equation}
				\label{def:s_tilde}
				\widetilde s_{V_\alpha}(z) \coloneqq -r_\alpha(z)b(z) - \frac{g_\alpha(z)}{2z}.
			\end{equation}
		\end{enumerate}
	\end{lemma}
	
	Relying on the fact that $s_{V_\alpha}(z)$ and $\widetilde{s}_{V_\alpha}(z)$ are the roots of \eqref{eq:quadratic}, we can state the following lemma which is crucial to derive the local laws. This lemma states, among other technical results, that: for all $z$ in a rectangular region $\cR$ above an enlargement of the spectrum, whenever $u\in\C$ almost solves \eqref{eq:quadratic}, then $u$ is close to either $s_{V_\alpha}(z)$ or $\widetilde{s}_{V_\alpha}(z)$. 
	Moreover when $z$ is in a trapezoid region $\cT$ above the spectrum and $u$ is such that $\im (u)>0$, then we can argue $u$ is indeed close to $s_{V_\alpha}(z)$, using that $\im(s_{V_\alpha}(z)) > 0$ whereas $\im(\widetilde{s}_{V_\alpha}(z))<0$.
	This result is reminiscent of various results in random matrix theory, and is in particular the analogue of \cite[Lemma B.1]{BouModPai2022} proven in the case where $V$ is analytic. The proof for the singular potential $V_\alpha$ requires some additional steps and can be found in Appendix \ref{subsec:stability}.
	
	\begin{lemma}[Stability lemma] \label{lem:stability}
		There exist $\delta_0,C>0$ depending only on $p>2$ such that the following bounds hold true for any $\alpha \in [0,1]$. 
		Let 
		\begin{equation} \label{eq:def_R_and_T}
			\cR \coloneqq [-1-\delta_0,1+\delta_0]+\ii(0,\delta_0] 
			\qquad \text{and} \qquad 
			\cT \coloneqq \{ x + \ii y \in \cR : \abs{x} \leq 1+y \}.
		\end{equation}
		Let $\zeta\in \C$, $z\in\cR$ and $u\in \C$ such that
		\[
		u^2+\frac{g_\alpha(z)}{z}u+\frac{h_\alpha(z)}{z}=\zeta.
		\]
		Then, recalling \eqref{def:s_tilde},
		\begin{equation}
			\label{eq:Stab1}
			\abs{u-s_{V_\alpha}(z)} \wedge \abs{u-\widetilde{s}_{V_\alpha}(z)} 
			\leq C\left( \frac{\abs{\zeta}}{\abs{b(z)}} \wedge  \abs{\zeta}^{1/2} \right).
		\end{equation}
		If $\im(u)>0$, then
		\begin{equation}
			\label{eq:Stab2}
			\abs{\im(u-s_{V_\alpha}(z))}
			\leq C\left( \abs{u-s_{V_\alpha}(z)} \wedge \abs{u-\widetilde{s}_{V_\alpha}(z)} \right).
		\end{equation}
		If $\im(u)>0$ and $z \in \cT$, then
		\begin{equation}
			\label{eq:Stab3}
			\abs{u-s_{V_\alpha}(z)} 
			\leq C\left( \frac{\abs{\zeta}}{\abs{b(z)}} \wedge \abs{\zeta}^{1/2} \right).
		\end{equation}
	\end{lemma}
	
	We now have all the ingredients needed to state and prove the local law machinery via the loop equations and stability lemma. 
	This results establishes a local law provided that we have a bound on moments of the linear statistics $L_N(f_{\alpha,z})$.
	
	\begin{proposition}[Local law machinery] \label{prop:local_law_machinery}
    	Let $p>2$. 
		Let $\delta_0$, $\cR$ and $\cT$ be as in Lemma \ref{lem:stability}. 
		Then there exists $C>0$ such that the following holds, for any $N \geq 1$ and $\alpha\in[0,1]$.
		Assume that there exist $a \geq 1$ and $B_N \colon \cR \to (0,\infty)$ such that, for any $z \in \cR$ and $q \geq 1$,
		\begin{equation} \label{eq:ass_L_N(f_alpha,z)}
			\E_\alpha\left[\abs{L_N(f_{\alpha,z})}^q\right] \leq (q^aB_N(z))^q.
		\end{equation}
		Then,
		\begin{enumerate}
			\item\label{it:in_the_trapezoid} \emph{[Inside the trapezoid]} For any $z \in \cT$ and $q \geq 1$,
			\[
			\E_\alpha \left[\abs{s_N(z)-s_{V_\alpha}(z)}^q\right] 
			\leq \frac{(Cq)^{q/2}}{(Ny)^q} + \frac{(Cq^aB_N(z))^q}{(N \abs{z} \abs{z^2-1}^{1/2})^q}.
			\]
			\item\label{it:out_the_trapezoid} \emph{[Outside the trapezoid]} 
			Let $B_{\max} \coloneqq 1 \vee \sup_{z \in \overline{\cR \setminus \cT}} B_N(z)$.
			Then, for any $z=\cR \setminus \cT$ and $q \geq 1$,
			\[
			\E_\alpha \left[\abs{s_N(z)-s_{V_\alpha}(z)}^q\right] \leq \frac{(Cq)^{2aq} B_{\max}^{2q}}{(Ny)^{q}}.
			\]
		\end{enumerate}
	\end{proposition}
	
	\begin{remark} \label{rem:even_moments}
		\begin{itemize}
			\item In the proof below as in many other places in the paper, it is convenient to restrict ourselves to bound even moments. 
			This is enough because the odd case can then be deduced via the inequality $\Ec{\abs{X}^q} \leq \E[\abs{X}^{2q}]^{1/2}$ up to a change of the constant $C$ in the bound. 
			\item We will start with the initial $B_N(z)=C\sqrt{N\log N}$ given by Corollary \ref{cor:a_priori_f_z}. The local law that we will deduce will allow us, through the use of the Helffer--Sjöstrand formula (see Section \ref{sec:HS}) and the master operator $\Xi_\alpha$ (see Section \ref{sec:bounds_test_functions}) to obtain in Section \ref{sec:bootstrap} the better (in $N$ but worse in $q$) estimate:
			$$\mathbb{E}_\alpha\left[|L_N(f_{\alpha,z})|^q\right]\leq \left(Cq^2(\log N)^2\dfrac{1}{\abs{z}}\right)^q.$$
			We will then use $B_N(z)=C(\log N)^2/\abs{z}$ and $a=2$. We will then obtain a local law from which we will deduce the improved estimate:
			$$\E_\alpha\left[\abs{L_N(f_{\alpha,z})}^q\right] \leq \left( Cq^4(\log N)^4 \right)^q.$$
			Applying the same strategy we will get the further improved bound
			$$ \E_\alpha\left[\abs{L_N(f_{\alpha,z})}^q\right] \leq (Cq^{32})^q,$$
			Allow us to obtain an optimal local law (in $N$) from which we will deduce Theorem~\ref{thm:CLT} and Corollary~\ref{thm:subleadingZN}.
		\end{itemize}
	\end{remark}
	
	We first prove the local law inside the trapezoid region $\mathcal{T}$ (recall \eqref{eq:def_R_and_T}) using the same approach as \cite[Theorem 1.1]{BouModPai2022}. Recall that the constants $C>0$ appearing in our bounds can depend on $\beta$ and $p$, can be increased from occurrence to occurrence, but are always uniform in $N$, $q$, $z$ and $\alpha$.
	
	\begin{proof}[Proof of Proposition \ref{prop:local_law_machinery}.\ref{it:in_the_trapezoid}] 
		We first work with $z\in \cR$ and restrict ourselves to $z \in \cT$ later in the proof. 
		First note that by definition of $g$ in \eqref{eq:def_g_C} and the fact that $\chi$ is chosen to be an even function, we have $g(\overline{z}) = \overline{g(z)}$ and therefore $P_\alpha(\overline{z}) = \overline{P_\alpha(z)}$.
		Then, it follows from the loop equation in Proposition \ref{prop:loop equations} together with the triangle inequality that 
		\begin{align}
			\E_\alpha\left[ \abs{P_\alpha(z)}^{2q} \right] 
			& \leq \frac{1}{N\abs{z}}
			\E_\alpha\left[\abs{L_N(f_{\alpha,z})} \cdot \abs{P_\alpha(z)}^{2q-1} \right]
			+ \frac{\abs{1-2/\beta}}{N} \E_\alpha\left[\abs{s_N'(z)}\cdot \abs{P_\alpha(z)}^{2q-1}\right] \nonumber \\
			& \qquad {} + \frac{4q-2}{\beta N^2 \abs{z}} 
			\E_\alpha \left[ \abs{ 2s_N(z)+\frac{g_\alpha(z)}{z} }
			\cdot \left( \frac{1}{N}\sum_{k=1}^N \frac{\abs{\lambda_k}}{\abs{\lambda_k - z}^3} \right)
			\cdot \abs{P_\alpha(z)}^{2q-2} \right].
			\label{eq:first_equality}
		\end{align}
		Recall Young's inequality: if $a^{-1}+b^{-1}=1$, then $xy\leq\frac{x^a}{a}+\frac{y^b}{b}$ for any $x,y \geq 0$.
		We apply Young's inequality to each term on the RHS of \eqref{eq:first_equality}, by introducing a factor $\lambda>0$ and $1/\lambda$, and taking $a=2q$ and $b=2q/(2q-1)$ for the first two terms, and $a=q$ and $b=q/(q-1)$ for the third term. 
		Thus, the RHS of \eqref{eq:first_equality} is at most
		\begin{align}
			& \left(2 \cdot \frac{2q-1}{2q\lambda^{\frac{2q}{2q-1}}}
			+ \frac{q-1}{q\lambda^{\frac{q}{q-1}}}\right) \E_\alpha\left[\abs{P_\alpha(z)}^{2q}\right]
			+ \frac{\lambda^{2q}}{2q (N\abs{z})^{2q}}
			\E_\alpha\left[\abs{L_N(f_{\alpha,z})}^{2q} \right]
			+ \frac{(\lambda \abs{1-2/\beta})^{2q}}{2q N^{2q}}
			\E_\alpha\left[\abs{s_N'(z)}^{2q}\right] \nonumber \\
			& {} + \frac{\lambda^q (4q-2)^q}{q(\beta N^2 \abs{z})^q} 
			\E_\alpha \left[ \abs{ 2s_N(z)+\frac{g_\alpha(z)}{z} }^q
			\cdot \left( \frac{1}{N}\sum_{k=1}^N \frac{\abs{\lambda_k}}{\abs{\lambda_k - z}^3} \right)^q
			\right].
			\label{eq:after_Young}
		\end{align}
		Now, we can choose $\lambda = 6$ so that
		\[
		2 \cdot \frac{2q-1}{2q\lambda^{\frac{2q}{2q-1}}}
		+ \frac{q-1}{q\lambda^{\frac{q}{q-1}}}
		\leq \frac{2}{\lambda} + \frac{1}{\lambda}
		= \frac{1}{2},
		\]
		and subtracting the $\frac{1}{2}\E_\alpha[\abs{P_\alpha(z)}^{2q}]$ term to the LHS of \eqref{eq:first_equality}, we get
		\begin{align}
			\E_\alpha\left[ \abs{P_\alpha(z)}^{2q} \right] 
			& \leq \frac{C^q}{(N\abs{z})^{2q}}
			\E_\alpha\left[\abs{L_N(f_{\alpha,z})}^{2q} \right]
			+ \frac{C^q}{N^{2q}}
			\E_\alpha\left[\abs{s_N'(z)}^{2q}\right] \nonumber \\
			& \qquad {} + \frac{(Cq)^q}{(N^2 \abs{z})^q} 
			\E_\alpha \left[ \abs{ 2s_N(z)+\frac{g_\alpha(z)}{z} }^q
			\cdot \left( \frac{1}{N}\sum_{k=1}^N \frac{\abs{\lambda_k}}{\abs{\lambda_k - z}^3} \right)^q
			\right].
		\end{align}
		Then, using the assumption \eqref{eq:ass_L_N(f_alpha,z)} on $L_N(f_{\alpha,z})$ and the following bounds
		\begin{equation*}
			\abs{s_N'(z)}
			\leq \frac{1}{N}\sum_{k=1}^N \frac{1}{\abs{\lambda_k-z}^2}
			= \frac{\im s_N(z)}{y},
		\end{equation*}
		and, using $\abs{\lambda_k} \leq \abs{\lambda_k-z} + \abs{z}$,
		\begin{equation*}
			\frac{1}{N}\sum_{k=1}^N \frac{\abs{\lambda_k}}{\abs{\lambda_k-z}^3}
			\leq \frac{1}{N}\sum_{k=1}^N \frac{1}{\abs{\lambda_k-z}^2}
			+ \frac{\abs{z}}{N}\sum_{k=1}^N \frac{1}{\abs{\lambda_k-z}^3}
			\leq \frac{\im s_N(z)}{y} 
			\left( 1+ \frac{\abs{z}}{y} \right),
		\end{equation*}
		we obtain:
		\begin{equation}\label{eq:intermboundP}
			\hspace{-0,15cm}\E_\alpha\left[\abs{P_\alpha(z)}^{2q} \right] 
			\leq \frac{(C q^a B_N(z))^{2q}}{(N\abs{z})^{2q}}
			+ \frac{(Cq)^q}{(Ny)^{2q}} 
			\E_\alpha\left[ \abs{\im s_N(z)}^{2q}
			+ \abs{ 2s_N(z)+\frac{g_\alpha(z)}{z} }^q
			\cdot \abs{\im s_N(z)}^q \right].
		\end{equation}
		We now restrict ourselves to $z \in \cT$.
		By \eqref{eq:Stab3}, we know that, since $z \in \cT$ and $\im s_N(z)>0$,
		\begin{equation} \label{eq:application_stability}
			\abs{s_N(z)-s_{V_\alpha}(z)}
			\leq C\left( \frac{\abs{P_\alpha(z)}}{\abs{b(z)}} \wedge \abs{P_\alpha(z)}^{1/2} \right).
		\end{equation}
		We introduce
		\begin{equation*}
			\Lambda(z) \coloneqq \im s_N(z) \vee \abs{2s_N(z)+\frac{g_\alpha(z)}{z}}
		\end{equation*}
		and we bound
		\begin{equation*}
			\Lambda(z) \leq 2 \abs{s_N(z)-s_{V_\alpha}(z)} + \left( \im s_{V_\alpha}(z) \vee \abs{2s_{V_\alpha}(z)+\frac{g_\alpha(z)}{z}} \right)
			\leq C\left( \abs{P_\alpha(z)}^{1/2}+\abs{b(z)}\right),
		\end{equation*}
		where, for the 1st term, we used \eqref{eq:application_stability} and, for the 2nd term, we recall that $s_{V_\alpha}(z) = r_\alpha(z)b(z)-\frac{g_\alpha(z)}{2z}$ (see Lemma~\ref{lem:stieltjes}.\ref{lemsta:point1}) and used \eqref{eq:lb_Im(rb)} and \eqref{eq:lb_r}.
		Hence,  we obtain:
		\begin{align*}
			\E_\alpha\left[\abs{\Lambda(z)}^{2q} \right]
			& \leq C^q\E_\alpha\left[\abs{P_\alpha(z)}^{q}  \right]+C^{q}\abs{b(z)}^{2q} \\
			& \leq \frac{(C q^a B_N(z))^{q}}{(N\abs{z})^{q}}
			+ \frac{(Cq)^{q/2}}{(Ny)^{q}}  \E_\alpha\left[\abs{\Lambda(z)}^{2q}\right]^{1/2}
			+ C^{q}\abs{b(z)}^{2q},
		\end{align*}
		where we used $\E_\alpha[\abs{P_\alpha(z)}^{q}] \leq \E_\alpha[\abs{P_\alpha(z)}^{2q}]^{1/2}$ and then \eqref{eq:intermboundP}.
		Using that $x\leq a\sqrt{x}+b\Rightarrow x\leq a^2+b$ for $x,a,b>0$, we obtain:
		\begin{equation} \label{eq:borneLambdaTrapeze}
			\E_\alpha\left[\abs{\Lambda(z)}^{2q} \right]
			\leq C^{q}\abs{b(z)}^{2q} 
			+ \frac{(C q^a B_N(z))^{q}}{(N\abs{z})^{q}}
			+ \frac{(Cq)^{q}}{(Ny)^{2q}}.
		\end{equation}
		Using this bound in \eqref{eq:intermboundP} yields
		\begin{align}
			\E_\alpha\left[\abs{P_\alpha(z)}^{2q} \right] 
			& \leq \frac{(C q^a B_N(z))^{2q}}{(N\abs{z})^{2q}}
			+ \frac{(Cq)^q}{(Ny)^{2q}}\abs{b(z)}^{2q}
			+ \frac{(Cq)^q}{(Ny)^{2q}} \frac{(q^aB_N(z))^{q}}{(N\abs{z})^{q}}
			+ \frac{(Cq)^{2q}}{(Ny)^{4q}} \nonumber \\
			& \leq \frac{(C q^a B_N(z))^{2q}}{(N\abs{z})^{2q}}
			+ \frac{(Cq)^q}{(Ny)^{2q}}\abs{b(z)}^{2q}
			+ \frac{(Cq)^{2q}}{(Ny)^{4q}}, \label{ineq:lastboundP}
		\end{align}
		where we neglected one term by using $ab\leq2a^2+2b^2$. We now distinguish according to which term is the largest in the RHS of \eqref{ineq:lastboundP}:
		\begin{itemize}
			\item If the largest term is the first or second one, then, up to an increasing of $C$, we get 
			\begin{equation*}
				\E_\alpha\left[\abs{P_\alpha(z)}^{2q} \right] 
				\leq \frac{(C q^a B_N(z))^{2q}}{(N\abs{z})^{2q}}
				+ \frac{(Cq)^q}{(Ny)^{2q}}\abs{b(z)}^{2q},
			\end{equation*}
			and, using the inequality $\abs{s_N(z)-s_{V_\alpha}(z)} \leq \frac{C}{\abs{b(z)}} \abs{P_\alpha(z)}$ coming from \eqref{eq:application_stability}, we deduce
			\begin{equation*}
				\E_\alpha\left[\abs{s_N(z)-s_{V_\alpha}(z)}^{2q}  \right]
				\leq \frac{(C q^a B_N(z))^{2q}}{(N\abs{zb(z)})^{2q}}
				+ \frac{(Cq)^q}{(Ny)^{2q}}.
			\end{equation*}
			\item If the largest term is the third one, then we get
			\begin{equation*}
				\E_\alpha\left[\abs{P_\alpha(z)}^{2q} \right] 
				\leq\dfrac{C^qq^{2q}}{(Ny)^{4q}}.
			\end{equation*}
			Thus, by using $\left|s_N(z)-s_{V_\alpha}(z)\right| \leq C \left|P_\alpha(z)\right|^{1/2}$, we obtain
			\begin{equation*}
				\E_\alpha\left[\abs{s_N(z)-s_{V_\alpha}(z)}^{2q}  \right]
				\leq \frac{(Cq)^q}{(Ny)^{2q}}
				\leq \frac{(C q^a B_N(z))^{2q}}{(N\abs{zb(z)})^{2q}}
				+ \frac{(Cq)^q}{(Ny)^{2q}}.
			\end{equation*}
		\end{itemize}
		In all cases, we obtained the desired bound so this concludes the proof of Part \ref{it:in_the_trapezoid} of Proposition~\ref{prop:local_law_machinery} (recall from Remark~\ref{rem:even_moments} that it is enough to bound even moments).
	\end{proof}
	
	Before proving Part \ref{it:out_the_trapezoid} in Proposition \ref{prop:local_law_machinery}, we need the following result which is the analogue of \cite[Proposition 2.5]{BouModPai2022}. It consists in a bound on the imaginary part of the differences between the two Stieltjes transforms $s_N$ and $s_{V_\alpha}$, which will be necessary to extend the local law \ref{it:in_the_trapezoid} outside the trapezoid region.

	
	\begin{lemma} \label{lem:local_law_outside}
	    Let $p>2$. 
		Let $\delta_0$, $\cR$ and $\cT$ be as in Lemma \ref{lem:stability}. 
		Then there exists $C,C'>0$ such that the following holds, for any $N \geq 1$ and $\alpha\in[0,1]$.
		Assume that there exist $a>0$ and $B_N \colon \cR \to (0,\infty)$ such that, for any $z \in \cR$ and $q \geq 1$,
		\begin{equation} \label{eq:ass_L_N(f_alpha,z)'}
			\E_\alpha\left[\abs{L_N(f_{\alpha,z})}^q\right] \leq (q^aB_N(z))^q.
		\end{equation}
		Then, for any $z \in \cR \setminus \cT$ and any $q \geq 1$, setting $\kappa \coloneqq \abs{x-1} \wedge \abs{x+1}$,
		\[
		\E_\alpha\left[\abs{\im(s_N(z)-s_{V_\alpha}(z))}^q\right] 
		\leq 
		\begin{cases}
			\dfrac{(Cq)^{q/2}}{N^{q}(y\kappa)^{q/2}}
			+ \dfrac{(Cq)^{q}}{(Ny)^{2q} \kappa^{q/2}}
			+ \dfrac{(Cq^aB_N(z))^q}{N^q \kappa^{q/2}}, 
			& \text{if } y\geq \sqrt{C'q}/(N\sqrt{\kappa}), \\
			\dfrac{(Cq)^{q/2}}{(Ny)^{q}}+\dfrac{(Cq^aB_N(z))^{q/2}}{N^{q/2}},
			& \text{if } y\leq \sqrt{C'q}/(N\sqrt{\kappa}).
		\end{cases}
		\]
		The same bounds hold for the moments of $\abs{s_N(z)-s_{V_\alpha}(z)}\wedge\abs{s_N(z)-\widetilde{s}_{V_\alpha}(z)}$.
	\end{lemma}
	
	\begin{proof}
		Let $z \in \cR \setminus \cT$.
		We recall that $\widetilde{s}_{V_\alpha}(z)$, the other solution of \eqref{eq:quadratic}, is given by \eqref{def:s_tilde}. 
		It follows from \eqref{eq:intermboundP} and the fact that $\abs{z}$ is bounded away from 0 for $z \in \cR \setminus \cT$ that 
		\begin{equation}\label{ineq:P(ii)}
			\hspace{-0,1cm}\E_\alpha\left[\abs{P_\alpha(z)}^{2q} \right] 
			\leq \frac{(C q^a B_N(z))^{2q}}{N^{2q}}
			+ \frac{(Cq)^q}{(Ny)^{2q}} 
			\E_\alpha\left[ \abs{\im s_N(z)}^{2q}
			+ \abs{ 2s_N(z)+\frac{g_\alpha(z)}{z} }^q
			\cdot \abs{\im s_N(z)}^q \right].
		\end{equation}
		We now define $\Upsilon(z) \coloneqq \abs{s_N(z)-s_{V_\alpha}(z)}\wedge\abs{s_N(z)-\widetilde{s}_{V_\alpha}(z)}$ where $\widetilde{s}_{V_\alpha}(z)$ is the other solution of $P_\alpha(z)$. From \eqref{eq:Stab2}, a bound on moments of $\Upsilon(z)$ will be enough to conclude. Using \eqref{eq:Stab4}, we have
		\begin{equation*}
			\abs{\im s_N(z)}
			\leq \Upsilon(z) + \abs{\im {s_{V_\alpha}(z)}} \vee \abs{\im\widetilde{s}_{V_\alpha}(z)}
			\leq \Upsilon(z) + \frac{Cy}{\abs{b(z)}}.
		\end{equation*}
		We also have by the expressions \eqref{eq:s_formula} and \eqref{def:s_tilde} for $s_{V_\alpha}(z)$ and $\widetilde{s}_{V_\alpha}(z)$,
		\begin{align*}
			\left|2s_N(z)+\dfrac{g_\alpha(z)}{z}\right|
			& \leq 2\Upsilon(z) + \abs{2s_{V_\alpha}(z)+\dfrac{g_\alpha(z)}{z}} \vee \abs{2\widetilde{s}_{V_\alpha}(z)+\dfrac{g_\alpha(z)}{z}} \\
			& = 2\Upsilon(z) + 2\abs{\widetilde{r}_\alpha(z)b(z)} \\
			& \leq 2\Upsilon(z) + C\abs{b(z)},
		\end{align*}
		using in the last inequality that $\abs{\widetilde{r}_\alpha(z)} \leq C$ by \eqref{eq:lb_r}.
		Plugging this in \eqref{ineq:P(ii)} and using the fact that $\abs{b(z)}\leq C\sqrt{\kappa}$ and therefore $y^q/\abs{b(z)}^{2q}\leq C$ (because $y\leq\kappa$), we obtain:
		\begin{equation}\label{eq:bound P Upsilon}
			\E\left[\abs{P_\alpha(z)}^{2q} \right] 
			\leq\frac{(C q^a B_N(z))^{2q}}{N^{2q}}
			+ \frac{(Cq)^q}{(Ny)^{2q}} \E\left[\Upsilon(z)^{2q}+y^{q}+\kappa^{q/2}\Upsilon(z)^q \right].
		\end{equation}
		On the other hand, by the stability lemma (see \eqref{eq:Stab1}) and the inequality $\abs{b(z)}\leq C\sqrt{\kappa}$, we have
		\begin{equation} \label{eq:stability_for_Upsilon}
			\Upsilon(z) \leq C \left( \frac{\abs{P_\alpha(z)}}{\sqrt{\kappa}} \wedge \abs{P_\alpha(z)}^{1/2} \right).
		\end{equation}
		We now distinguish according to which term is the largest in the RHS of \eqref{eq:bound P Upsilon}:
		\begin{itemize}
			\item If the largest term is the second one, then
			by following the arguments of the proof of \cite[Proposition 2.5]{BouModPai2022} (see after (2.44) in their proof), there exists $C'>0$ such that, if $y\geq C'\sqrt{q}/(N\sqrt{\kappa})$, by using $\Upsilon(z)\leq C\abs{P_\alpha(z)}/\sqrt{\kappa}$, we get
			\begin{equation*}
				\E_\alpha\left[\Upsilon(z)^{2q}\right]\leq\dfrac{(Cq)^q}{N^{2q}y^q\kappa^q}+\dfrac{(Cq)^{2q}}{(Ny)^{4q}\kappa^{q}}
			\end{equation*}
			and, if $y\leq C'\sqrt{q}/(N\sqrt{\kappa})$, by using $\Upsilon(z)\leq C|P_\alpha(z)|^{1/2}$ instead, we get
			\begin{equation*}
				\E_\alpha\left[\Upsilon(z)^{2q}\right]\leq\dfrac{(Cq)^q}{(Ny)^{2q}}.
			\end{equation*}
			\item If the largest term is the first one, then by \eqref{eq:stability_for_Upsilon}, we have
			\begin{equation*}
				\E_\alpha\left[\Upsilon(z)^{2q}\right]
				\leq C^{q} \left[\dfrac{(q^aB_N(z))^{2q}}{N^{2q}\kappa^{q}}\wedge\dfrac{(q^aB_N(z))^{q}}{N^{q}}\right],
			\end{equation*}
			and we choose to use the first part of the bound when $y\geq C'\sqrt{q}/(N\sqrt{\kappa})$ and the second part otherwise.
		\end{itemize}
		Summing over these two cases yields the result.
	\end{proof}
	
	We also need the following lemma which is the analogue of \cite[Lemma 3.6]{BouModPai2022}. It is a bound on the moments of the number of outliers, \textit{i.e.\@} the number of particles outside of the limiting spectrum $(-1,1)$. Let $I$ be an interval, we define 
	\begin{equation}
		\mc{N}(I) \coloneqq \#\left\{j\in\llbracket1,N\rrbracket : \lambda_i\in I\right\},
	\end{equation}
	which is the number of particles in $I$.
	
	\begin{lemma} \label{lem:nb particles edge}
	    Let $p>2$. 
		Let $\delta_0$, $\cR$ and $\cT$ be as in Lemma \ref{lem:stability}. 
		Then there exists $C>0$ such that the following holds, for any $N \geq 1$ and $\alpha\in[0,1]$.
		Assume that there exist $a\geq 1$ and $B_N \colon \cR \to (0,\infty)$ such that, for any $z \in \cR$ and $q \geq 1$,
		\begin{equation} \label{eq:ass_L_N(f_alpha,z)''}
			\E_\alpha\left[\abs{L_N(f_{\alpha,z})}^q\right] \leq (q^aB_N(z))^q.
		\end{equation}
		Let $B_{\max} \coloneqq 1 \vee \sup_{z \in \overline{\cR \setminus \cT}} B_N(z)$.
		Then, for any $q \geq 1$,
		\begin{equation}\label{ineq:N(I)}
			\E_\alpha\left[\mathcal{N}([1,\infty))^{q}\right]
			\leq (Cq)^{2aq} B_{\max}^{2q}.
		\end{equation}
	\end{lemma}

	\begin{proof}
		We follow the proof of \cite[Lemma 3.6]{BouModPai2022}, but we have to keep track of the role of the factor $(q^aB_N(z))^q$ appearing in the bound of Lemma~\ref{lem:local_law_outside} and which is not present in \cite{BouModPai2022}. 
		
		\textbf{Preliminaries.} We first gather some tools for the proof.
		First note that, for any $z = x+\ii y$ with $y > 0$,
		\begin{equation} \label{eq:N(I)_et_Im_s}
			\mc{N}([x-y,x+y])
			= \sum_{i=1}^{N}\1_{\abs{\lambda_i-x}\leq y}
			\leq \sum_{i=1}^{N} \dfrac{2 y^{2}}{\abs{z-\lambda_i}^2}
			= 2 y N \cdot \im s_N(z).
		\end{equation}
		On the other hand, by \eqref{eq:Stab4}, there exists $C_0>0$ such that for any $\alpha\in[0,1]$ and $z \in \cR\setminus \cT$, we have $\im s_{V_\alpha}(z)\leq (C_0^2 y)/(4\sqrt{\kappa})$, where we used $\abs{b(z)} \geq \sqrt{\kappa}$ with $\kappa \coloneqq \abs{x-1} \wedge \abs{x+1}$. 
		Therefore, for $z \in \cR\setminus \cT$ such that $y\leq\kappa^{1/4}/(C_0\sqrt{N})$, we have $\im s_{V_\alpha}(z)\leq1/(4 y N)$ and, using that $\mc{N}(I)$ is an integer, it follows from \eqref{eq:N(I)_et_Im_s} that
		\[
		\mc{N}([x-y,x+y]) \leq 4 y N \cdot \abs{\im (s_N(z)-s_{V_\alpha}(z))}.
		\]
		Combining this with Lemma~\ref{lem:local_law_outside}, there exists $C_1,C_2>0$ such that, for any $q \geq 1$ and $z \in \cR\setminus \cT$ such that $\sqrt{C_1q}/(N\sqrt{\kappa}) \leq y\leq\kappa^{1/4}/(C_0\sqrt{N})$,
		\begin{equation} \label{eq:bound_nb_particles}
			\E_\alpha \left[ \mc{N}([x-y,x+y])^q \right]
			\leq \left(\frac{C_2 qy}{\kappa} \right)^{q/2} 
			+ \left(\frac{C_2 q}{Ny\sqrt{\kappa}} \right)^q
			+ \left(\frac{C_2 q^a B_{\max} y}{\kappa} \right)^q,
		\end{equation}
		where we bounded $\sqrt{\kappa} \geq \kappa \delta_0^{-1/2}$ in the denominator fo the last term.
		Note that this inequality cannot be applied if $x >1+\delta_0$ because then $z \notin \cR$. Moreover, the condition on $y$ is empty if $\kappa < (C_1C_0^2q)^{2/3} N^{-2/3}$.
		Therefore, the proof proceeds by cutting the interval $[1,\infty)$ as follows
		\begin{equation} \label{eq:interval_decompo}
			[1,\infty) = [1,1+KN^{-2/3}] \cup (1+KN^{-2/3},1+\delta_0] \cup (1+\delta_0,\infty),
		\end{equation}
		for some $K \geq (C_1C_0^2q)^{2/3}$ chosen explicitly below.
		The main task is to control the number of particles in the second interval, which is done by cutting it into many pieces and applying \eqref{eq:bound_nb_particles} to each of them.
		The first and third interval have to be treated differently.
		In order to make the argument on the second interval work, we need to take $K$ larger than the constraint mentioned above, more precisely, we choose
		\[
		K \coloneqq (qL^3)^{4a/3}
		\qquad \text{with} \qquad
		L \coloneqq C_0 \vee C_1 \vee \left( e^2 C_2 \right)^{1/4} \vee \left( e C_2 \right)^{1/2} \vee \left( e C_2 B_{\max} \right)^{1/(1+3a)}.
		\]
		We can assume w.l.o.g. that $C_0,C_1 \geq 1$, so that $K,L \geq 1$.
		
		\textbf{A first case.} We need to treat separately the case where $K N^{-2/3} > \delta_0$, in which case the decomposition \eqref{eq:interval_decompo} is not meaningful. 
		In that case, we have $N \leq (K/\delta_0)^{3/2} \leq Cq^{2a} L^{6a}$ so that
		\[
		\E_\alpha\left[\mathcal{N}([1,\infty))^{q}\right]
		\leq N^q 
		\leq C^q q^{2aq} L^{6aq}
		\leq C^q q^{2aq} B_{\max}^{2q},
		\]
		using that $L \leq C B_{\max}^{1/(1+3a)} \leq C B_{\max}^{1/3a}$ because $B_{\max} \geq 1$.
		
		Therefore, we can now restrict ourselves to the case $K N^{-2/3} \leq \delta_0$ and successively deal with the intervals in the decomposition \eqref{eq:interval_decompo}.
		
		\textbf{First interval: $(1,1+KN^{-2/3}]$.} We define $y=KN^{-2/3}$ and $z=1+ y+\ii y$, thus by \eqref{eq:N(I)_et_Im_s}:
		\begin{align*}
			\E\left[\mc{N}\left((1,1+KN^{-2/3}]\right) ^q\right] 
			& \leq (2 y N)^{q}\E\left[\abs{\im s_N(z)}^{q}\right] 
			\leq (CyN)^q \left( \E\left[\abs{s_N(z)-s_{V_\alpha}(z)}^{q}\right] + \abs{\im s_{V_\alpha}(z)}^q \right).
		\end{align*}
		Note that $z \in \cT$ (recall $y \leq \delta_0$ because we ruled out the opposite case before), so we can apply Proposition \ref{prop:local_law_machinery}.\ref{it:in_the_trapezoid} to the first term. 
		On the other hand, by \eqref{eq:Stab4} (and continuity to apply it on the boundary of the trapezoid region), we have $\abs{\im s_{V_\alpha}(z)} \leq Cy/\abs{b(z)} \leq C \sqrt{y}$.
		Thus, we obtain, using in the second inequality $y \leq \delta_0$ and $y = (qL^3)^{4a/3} N^{-2/3}$,
		\begin{align*}
			\E\left[\mc{N}\left((1,1+KN^{-2/3}]\right)^q\right]
			& \leq (Cq)^{q/2} 
			+ (Cq^aB_{\max})^q y^{q/2}
			+ (CN)^q y^{3q/2} \\
			& \leq (Cq)^{q/2} 
			+ (Cq^aB_{\max})^q 
			+ C^q (qL^3)^{2aq}
			\leq (Cq^{2a}B_{\max}^{2})^q,
		\end{align*}
		using $a \geq 1$, $B_{\max} \geq 1$ and $L \leq C B_{\max}^{1/(1+3a)} \leq C B_{\max}^{1/3a}$.
		
		\textbf{Second interval: $I \coloneqq (1+KN^{-2/3},1+\delta_0]$.} We split this interval into several smaller intervals.
		Let $a_0 \coloneqq K$ and $a_{j+1} \coloneqq a_j+a_j^{1/4}/L$ for $j \geq 0$.
		We define
		\[
		\kappa_j \coloneqq a_j N^{-2/3}, \qquad
		x_j \coloneqq 1+\kappa_j, \qquad 
		y_j \coloneqq \dfrac{a_j^{1/4}}{L}N^{-2/3}, \qquad
		q_j \coloneqq \left\lfloor \frac{a_j^{3/(4a)}}{L^3} \right\rfloor,
		\]
		and consider the intervals $I_j \coloneqq (x_j,x_j+y_j]$. Note that $x_0=1+KN^{-2/3}$ and $x_{j+1} = x_j+y_j$ so that $I = [1+KN^{-2/3},1+\delta_0]\subset\bigcup_{j=0}^{j_\mrm{\max}}I_j$ with  $j_{\max}\coloneqq\max\{j\geq0,\kappa_j\leq\delta_0\}$.
		Using Minkowski's inequality first and then that $q_j\geq q_0=q$ (together with the fact that $\mc{N}(I_j)$ is integer valued), we have
		\begin{equation} \label{eq:sum_on_intervals}
			\E\left[\mc{N}(I)^{q}\right]^{1/q}
			\leq\sum_{j=1}^{j_{\max}}\E\left[\mc{N}(I_j)^{q}\right]^{1/q}
			\leq\sum_{j=1}^{j_{\max}}\E\left[\mc{N}(I_j)^{q_j}\right]^{1/q}.
		\end{equation}
		Now, in order to apply \eqref{eq:bound_nb_particles} with $(x,y) = (x_j,y_j)$ for $j \leq j_{\max}$, we need to check that the conditions are satisfied. We have $x_j+\ii y_j \in \cR\setminus \cT$ because $y_j \leq \kappa_j$ (using here $a_j \geq a_0 = K \geq 1$ and $L\geq1$) and $\kappa_j \leq \delta_0$ by definition of $j_{\max}$.
		The condition $y_j \leq \kappa_j^{1/4}/(C_0\sqrt{N})$ holds because $L \geq C_0$. Finally, $y_j \geq \sqrt{C_1q_j}/(N\sqrt{\kappa_j})$ is implied by $a_j^{3/4-3/(8a)} L^{1/2} \geq \sqrt{C_1}$ which holds because $a \geq 1$, $a_j \geq 1$ and $L \geq C_1$.
		Therefore, \eqref{eq:bound_nb_particles} yields
		\begin{equation}
			\E\left[\mc{N}(I_j)^{q_j}\right]
			\leq \left(\frac{C_2 a_j^{3/(4a)-3/4}}{L^4} \right)^{q_j/2} 
			+ \left(\frac{C_2 a_j^{3/(4a)-3/4}}{L^2} \right)^{q_j}
			+ \left(\frac{C_2 B_{\max} }{L^{3a+1}} \right)^{q_j}
			\leq 3 e^{-q_j},
		\end{equation}
		using $a_j^{3/(4a)-3/4} \leq 1$ and the definition of $L$.
		Combining this with \eqref{eq:sum_on_intervals} and using $q_j\geq a_j^{3/4a}/L^3-1$, $a_j\geq K+jK^{1/4}/L$ and then $K=(L^3q)^{4a/3}$ together with a sum-integral comparison, we obtain:
		\begin{align*}
			\E\left[\mc{N}(I)^{q}\right]^{1/q}
			& \leq 3 \sum_{j=0}^{\infty} e^{-q_j/q}
			\leq 3 \sum_{j=0}^{\infty} 
			\exp \left( - \frac{(K+jK^{1/4}/L)^{3/4a}}{qL^3} + \frac{1}{q} \right) \\
			& \leq 3e^{1/q} \left( 1 + \int_0^{\infty} e^{-(1+x/(LK^{3/4}))^{3/4a}} \diff x \right)
			= 3e^{1/q} LK^{3/4} \int_{0}^{\infty} e^{-(1+y)^{3/4a}} \diff y \\
			& \leq Ca q^{a} L^{1+3a}
			\leq (Cq)^{a} B_{\max},
		\end{align*}
		where we used in the last inequality that $L \leq C B_{\max}^{1/(1+3a)}$.
		
		\textbf{Third interval: $(1+\delta_0,\infty)$.} Here, we simply use the fact that the probability of having an eigenvalue larger than $1+\delta_0$ is exponentially small, see Lemma~\ref{lem:outliers}.\ref{it:ldp estimate}.
		This yields
		\[
		\E\left[\mc{N}((1+\delta_0,\infty))^{q}\right]
		\leq N^q \Pp{\exists k \in \llbracket1,N\rrbracket, \abs{\lambda_k} \geq 1+\delta_0}
		\leq CN^{q}e^{-cN}
		\leq (Cq)^{q},
		\]
		using here that $e^{cN} \geq (cN)^q/q! \geq N^q/(Cq)^q$.
		This concludes the proof.
	\end{proof}

    \begin{proof}[Proof of Proposition~\ref{prop:local_law_machinery}.\ref{it:out_the_trapezoid}]
    	Let $z = x + \ii y \in \cR \setminus \cT$. By symmetry, we can assume that $x \geq 0$, which necessarily means that $x > 1+y$. 
    	We define $z_0=1+y+\ii y$. 
    	Following the same argument as in \cite[Proposition 3.5]{BouModPai2022} (this uses \eqref{eq:Stab5}), we have
    	\begin{multline*}
    		\abs{s_N(z)-s_{V_\alpha}(z)}
    		\leq C\bigl(	\abs{s_N(z)-s_{V_\alpha}(z)}\wedge	\abs{s_N(z)-\widetilde{s}_{V_\alpha}(z)} 
    		+ \re \left(s_N(z_0)-s_N(z)\right) 
    		\\ + \abs{s_N(z_0)-s_{V_\alpha}(z_0)} \bigr),
    	\end{multline*}
    	as well as $\re(s_N(z_0)-s_N(z)) \leq \frac{2}{Ny} \cN((1,\infty))$.
    	We thus obtain:
    	\begin{align}
    		\E_\alpha\left[\abs{s_N(z)-s_{V_\alpha}(z)}^{q}\right]
    		\leq C^q \Biggl( \E_\alpha\left[ \abs{s_N(z)-s_{V_\alpha}(z)}^{q}\wedge	\abs{s_N(z)-\widetilde{s}_{V_\alpha}(z)}^{q} \right]
    		+ \frac{\E_\alpha[\cN((1,\infty))^q]}{(Ny)^q} \nonumber \\
    		+\E_\alpha\left[\abs{s_N(z_0)-s_{V_\alpha}(z_0)}^{q}\right] \Biggr). \label{eq:three_terms}
    	\end{align}
    	For the first term on the RHS of \eqref{eq:three_terms}, we use Lemma \ref{lem:local_law_outside} to obtain
    	\begin{equation}
    		\E[\abs{s_N(z)-s_{V_\alpha}(z)}^{q}\wedge	\abs{s_N(z)-\widetilde{s}_{V_\alpha}(z)}^{q}]
    		\leq \frac{(Cq^aB_{\max})^q}{(Ny)^q},
    	\end{equation}
    	by distinguishing the two cases in Lemma \ref{lem:local_law_outside} as follows.
    	For the first case, $y\geq \sqrt{C'q}/(N\sqrt{\kappa})$, we have, using that $\kappa \geq y$ for $z \in \cR \setminus \cT$,
    	\begin{align*}
    		\frac{(Cq)^{q/2}}{N^{q}(y\kappa)^{q/2}}
    		+ \frac{(Cq)^{q}}{(Ny)^{2q} \kappa^{q/2}}
    		+ \frac{(Cq^aB_N(z))^q}{N^q \kappa^{q/2}}
    		& \leq 
    		\frac{(Cq)^{q/2}}{(N y)^q} +\dfrac{C^qq^{q}}{(Ny)^{2q}\kappa^{q/2}} \cdot \left( \frac{Ny\sqrt{\kappa}}{(C'q)^{1/2}} \right)^{q/2} + \frac{(Cq^aB_{\max})^q}{N^q y^{q/2}} \\
    		& \leq \frac{(Cq)^{q/2}}{(N y)^q} + \frac{(Cq^aB_{\max})^q}{N^q y^{q/2}} \\
    		& \leq  \frac{(Cq^aB_{\max})^q}{(Ny)^q},
    	\end{align*}
    	using $y \leq \delta_0$ and $a,B_{\max} \geq 1$ in the last inequality.
    	In the second case, $y\leq \sqrt{C'q}/(N\sqrt{\kappa})$, we can bound the second term of the bound as follows
    	\[
    	\frac{(Cq^aB_N(z))^{q/2}}{N^{q/2}}
    	\leq \frac{(Cq^aB_{\max})^{q/2}}{N^{q/2}} \cdot \left( \frac{(C'q)^{1/2}}{Ny\sqrt{\kappa}} \right)^{q/2}
    	= \frac{(C q^{a/2+1/4} B_{\max}^{1/2})^q}{N^q y^{q/2} \kappa^{q/4}}
    	\leq \frac{(Cq^aB_{\max})^q}{(Ny)^q}.
    	\]
    	Then, using Lemma~\ref{lem:nb particles edge} for the second term on the RHS of \eqref{eq:three_terms} and Proposition~\ref{prop:local_law_machinery}.\ref{it:in_the_trapezoid} for the third term (note that $z_0 \in \cT$), we get
    	\begin{align*}
    		\E_\alpha\left[\abs{s_N(z)-s_{V_\alpha}(z)}^{q}\right]
    		& \leq \frac{(Cq^aB_{\max})^q}{(Ny)^q}
    		+ \frac{(Cq)^{2aq} B_{\max}^{2q}}{(Ny)^q}
    		+ \frac{(Cq)^{q/2}}{(Ny)^q} + \frac{(Cq^aB_N(z_0))^q}{(N \abs{z_0} \abso{z_0^2-1}^{1/2})^q} \\
    		& \leq \frac{(Cq)^{2aq} B_{\max}^{2q}}{(Ny)^q},
    	\end{align*}
    	using in particular that $\abs{z_0} \abso{z_0^2-1}^{1/2} \geq cy$.
    	This concludes the proof.
    \end{proof}
	
	\section{Applications of the Helffer--Sjöstrand formula}
	\label{sec:HS}
	
	In this section, we show how the \textit{Helffer--Sjöstrand} formula can be used to deduce, from a local law, bounds on linear statistics and on the anisotropy introduced below. 
	Let $f \colon \R \to \R$ be a test function, with sufficient integrability and regularity conditions so that the following quantities make sense.
	Recall our notation for linear statistics:
	\begin{equation} \label{eq:def_linear_statistics}
		L_N(f) \coloneqq \sum_{k=1}^N f(\lambda_k) - N \int_\R f(x) \diff \mu_{V_\alpha}(x)
		= N \int_\R f(x) \diff (\mu_N-\mu_{V_\alpha})(x).
	\end{equation}
	We also introduce the following quantity, sometimes called the \textit{anisotropy} (see e.g.\@ \cite{BekLebSer2018}), 
	\begin{equation} \label{eq:def_anisotropy}
		A_N(f) \coloneqq N^2\iint_{\R^2} \frac{f(\lambda)-f(\lambda')}{\lambda-\lambda'} 
		\diff (\mu_N-\mu_{V_\alpha})(\lambda) \diff (\mu_N-\mu_{V_\alpha})(\lambda'),
	\end{equation}
	where we take the convention $\frac{f(\lambda)-f(\lambda')}{\lambda-\lambda'}  = f'(\lambda)$ when $\lambda = \lambda'$. This quantity arises as an error term in the study of the convergence of $L_N(f)$ towards a Gaussian law and showing that this term in indeed negligible is often one of the main obstacle. Note however, that these quantities are scaled by appropriate factors of $N$ so that they are theoretically of order 1 for regular enough $f$.
	In this section, we show how to use Helffer--Sjöstrand formula to control these quantities.

	First recall Helffer--Sjöstrand formula, see e.g.\@ \cite[Proposition C.1]{BenKno2017} for a reference.
	\begin{proposition} \label{prop:helffer--Sjostrand}
		Let $n \in \N$ and $f \in \cC_c^{n+1}(\R)$. We define the pseudo-analytic extension of $f$ of degree $n$ through
		\[
		\widetilde{f}_n (x+\ii y) 
		\coloneqq \sum_{k=0}^n \frac{1}{k!} (\ii y)^k f^{(k)}(x).
		\]
		Let $\chi \in \cC_c^\infty(\R)$ such that $\chi(0) = 1$.
		Then, for any $\lambda \in \R$,
		\[
		f(\lambda) = \frac{1}{\pi} \iint_{\R^2} \frac{\overline{\partial}(\widetilde{f}_n (z) \chi(y))}{\lambda - z} \diff x \diff y,
		\]
		where $z = x+\ii y$ in the integral and $\overline{\partial} \coloneqq \frac{1}{2} (\partial_x + \ii \partial_y)$ is the antiholomorphic derivative.
	\end{proposition}
	
	There are two useful applications here.
	Firstly, for a linear statistics $L_N(f)$ with $f\in \cC_c^2(\R)$, we take $n=1$ and get
	\begin{equation} \label{eq:HS_L_N(f)}
		L_N(f)
		= \frac{N}{2\pi} \iint_{\R^2}
		\left( \ii y f''(x) \chi(y) + \ii(f(x) + \ii y f'(x))\chi'(y) \right) (s_N-s_{V_\alpha})(z) \diff x \diff y.
	\end{equation}
	Secondly, for the anisotropy $A_N(f)$ with $f\in \cC_c^3(\R)$, assuming that $\chi$ is identically equal to 1 in a neighbourhood of 0 and applying Helffer--Sjöstrand formula with $n=2$ yields, for $\lambda \neq \lambda'$,
	\[
	\frac{f(\lambda)-f(\lambda')}{\lambda-\lambda'} 
	=  \frac{1}{\pi} \iint_{\R^2} \frac{\overline{\partial}(\widetilde{f}_n (z) \chi(y))}{(\lambda - z)(\lambda' - z)} \diff x \diff y.
	\]
	This is also true for $\lambda = \lambda'$ by taking the limit $\lambda' \to \lambda$ on both sides (on the RHS, domination is obtained by noting that $\abso{\overline{\partial}(\widetilde{f}_n (z) \chi(y))} \leq Cy^2$ for $z$ in a neighbourhood of $\lambda$, because $\chi'(y) = 0$ for $\abs{y}$ small enough).
	Therefore
	\begin{equation} \label{eq:HS_A_N(f)}
		A_N(f)
		= \frac{N^2}{2\pi} \iint_{\R^2}
		\bigg( - \frac{y^2}{2} f'''(x) \chi(y) 
		+ \ii \big[ f(x) + \ii y f'(x) - \frac{y^2}{2} f''(x) \big] \chi'(y) \bigg)
		(s_N-s_{V_\alpha})(z)^2 \diff x \diff y.
	\end{equation}
	As mentioned before, the following lemmas allow us to translate the local laws obtained in Section~\ref{subsec:locallaws} into estimates on the moments of $L_N(f)$ and $A_N(f)$.
	
	\begin{lemma} \label{lem:HS_L_N}
    	Let $p>2$.
		For any $\delta>0$ and any compact set $K_0 \subset \R$, there exists $C>0$ such that the following holds.
		For any $N,q \geq 1$ and $\alpha \in [0,1]$, if there exists $M>0$ such that, for any $x \in K_0$ and $y \in (0,\delta]$, with $z = x+\ii y$,
		\begin{equation} \label{eq:ass_LL}
			\E_\alpha\left[ \abs{s_N(z)-s_{V_\alpha}(z)}^q \right] 
			\leq \left( \frac{M}{Ny} \right)^q,
		\end{equation}
		then, for any $\eta \in (0,(\delta \wedge 1)/2]$ and $f \in \cC_c^2(\R)$ supported in $K_0$, we have
		\begin{equation} \label{eq:HS_bound_L_N}
			\E_\alpha\left[ \abs{L_N(f)}^q \right]
			\leq \left( CM \left( \norme{f}_1 + \log(1/\eta) \norme{f'}_1 + \eta \norme{f''}_1 \right) \right)^q.
		\end{equation}
	\end{lemma}
	
	\begin{remark} \label{rem:HS_L}
		The parameter $\eta$ has to be thought of as small enough so that the term involving $\norme{f'}_1$ dominates.
		More precisely, in the sequel, we typically want to apply this result to functions $f$ where $f'(x)$ explodes at 0 like $\abs{x}^{p-3}$, therefore $\norme{f'}_1$ is finite because $p>2$.
		After some local $N$-dependent regularization of $f$, we can have a finite $\norme{f''}_1$ but which grows polynomially fast in $N$, and therefore we choose $\eta=N^{-k}$ with $k$ large enough to compensate.
		
		Having a control of $L_N(f)$ in terms of $\norme{f'}_1$ is the main interest of using the local laws together with Helffer--Sjöstrand formula. 
		This has to be compared with other methods to obtain estimates on linear statistics, which provide bounds in terms of norms exploding in the case mentioned above: for example, $\norme{f'}_2$ for the \textit{a priori} bound $L_N(f) = O(\sqrt{N\log N})$ in \cite[Corollary~4.16]{Gui2019} or $\normeo{f^{(4)}}_\infty$ in the sharper bound $L_N(f) = O(1)$ in \cite[Corollary~4.5]{BekLebSer2018}.
	\end{remark}

	\begin{proof}
		Let $\chi\in \cC_c^\infty(\R)$ be such that $\1_{[-\delta/2,\delta/2]} \leq \chi \leq \1_{[-\delta,\delta]}$.
		We start from \eqref{eq:HS_L_N(f)} and rewrite the term involving $f''(x)$, by first splitting it into two parts,
		\begin{align*}
			\iint_{\R^2} y f''(x) \chi(y) (s_N(z)-s_{V_\alpha}(z)) \diff x \diff y
			& = \iint_{\abs{y} \geq \eta} y f''(x) \chi(y) (s_N(z)-s_{V_\alpha}(z)) \diff x \diff y \\
			& \quad {} + \iint_{\abs{y} < \eta} y f''(x) \chi(y) (s_N(z)-s_{V_\alpha}(z)) \diff x \diff y,
		\end{align*}
		and then, for the part $\abs{y} \geq \eta$, we first integrate by part $f''(x)$ w.r.t.\@ $x$ and then $s_N'(z)-s_{V_\alpha}'(z)$ w.r.t.\@~$y$:
		\begin{align*}
			\iint_{\abs{y} \geq \eta} f''(x) y\chi(y) (s_N(z)-s_{V_\alpha}(z)) \diff x \diff y
			& = - \iint_{\abs{y} \geq \eta} f'(x) y\chi(y) (s_N'(z)-s_{V_\alpha}'(z)) \diff x \diff y \\
			& = \iint_{\abs{y} \geq \eta} f'(x) \partial_y(y \chi(y)) \frac{1}{\ii}(s_N(z)-s_{V_\alpha}(z)) \diff x \diff y \\
			& \quad {} + \int_{\R} f'(x) \eta \frac{1}{\ii} (s_N(x+\ii\eta)-s_{V_\alpha}(x+\ii\eta)) \diff x \\
			& \quad {} - \int_{\R} f'(x) \eta \frac{1}{\ii} (s_N(x-\ii\eta)-s_{V_\alpha}(x-\ii\eta)) \diff x,
		\end{align*}
		where we used that $\chi(\eta) = 1$ because $\eta \leq \delta/2$.
		Note that the two last term have the same modulus.
		Therefore, we can bound $\abs{L_N(f)}^q \leq (\frac{2}{\pi})^q (I_1^q + I_2^q + I_3^q + I_4^q)$, where
		\begin{align*}
			I_1 & \coloneqq N \iint_{\R^2} \left(\abs{f(x)}+ \abs{yf'(x)}\right) \abs{\chi'(y)} \cdot \abs{s_N(z)-s_{V_\alpha}(z)} \diff x\diff y, \\
			I_2 & \coloneqq N \iint_{\abs{y} <\eta} \abs{ y f''(x)} \cdot \abs{s_N(z)-s_{V_\alpha}(z)} \diff x\diff y, \\
			I_3 & \coloneqq N\iint_{\abs{y} \geq \eta} \abs{\partial_y(y\chi(y)) f'(x)} \cdot \abs{s_N(z)-s_{V_\alpha}(z)} \diff x \diff y, \\
			I_4 & \coloneqq 2N\eta \int_{\R} \abs{f'(x)} \cdot \abs{s_N(x+\ii\eta)-s_{V_\alpha}(x+\ii\eta)} \diff x.
		\end{align*}
		These terms are bounded using a similar method. 
		We detail it for $I_3$. Writing $I_3^q$ as an integral over $x_1,\dots,x_q,y_1,\dots,y_q$ and bounding $\abs{\partial_y(y\chi(y))} \leq C \1_{\abs{y} \leq \delta}$, we have
		\begin{equation*}
			\E_\alpha\left[ I_3^q \right]
			\leq (CN)^q \idotsint_{\forall k, \abs{y_k} \in [\eta,\delta]} \abs{f'(x_1)} \cdots \abs{f'(x_q)} \cdot
			\E_\alpha\left[ \prod_{k=1}^q \abs{s_N(z_k)-s_{V_\alpha}(z_k)}\right] 
			\diff x_1\cdots\diff x_q \diff y_1 \cdots \diff y_q.
		\end{equation*}
		Then, by Hölder's inequality, we have
		\begin{equation} \label{eq:Holder_for_HS}
			\E_\alpha\left[ \prod_{k=1}^q \abs{s_N(z_k)-s_{V_\alpha}(z_k)}\right] 
			\leq \prod_{k=1}^q \E_\alpha\left[ \abs{s_N(z_k)-s_{V_\alpha}(z_k)}^q\right]^{1/q}
			\leq \left( \frac{M}{N} \right)^q \prod_{k=1}^q \frac{1}{\abs{y_k}},
		\end{equation}
		by the assumption \eqref{eq:ass_LL}, noting that only points $z_k$ with $x_k \in K_0$ and $\abs{y_k} \leq \delta$ matter in the previous integral and also using $(s_N-s_{V_\alpha})(\overline{z_k}) = \overline{(s_N-s_{V_\alpha})(z_k)}$ to apply \eqref{eq:ass_LL} in the case $y_k < 0$.
		Therefore, we get 
		\begin{equation*}
			\E_\alpha[I_3^q] 
			\leq (CM)^q \left( \int_\R \abs{f'(x)} \diff x \right)^q 
			\left( \int_\eta^\delta \frac{1}{y} \diff y \right)^q 
			\leq (CM \norme{f'}_1 \log(1/\eta))^q.
		\end{equation*}
		Proceeding similarly, we get the bounds (for $I_1$ note that $\abs{\chi'(y)} \leq C \1_{\abs{y} \in [\delta/2,\delta]}$)
		\begin{equation*}
			\E_\alpha[I_1^q] 
			\leq (CM (\norme{f}_1+\norme{f'}_1))^q,
			\qquad 
			\E_\alpha[I_2^q] 
			\leq (CM \eta \norme{f''}_1)^q,
			\qquad 
			\E_\alpha[I_4^q] 
			\leq (CM \norme{f'}_1)^q,
		\end{equation*}
		which together yield \eqref{eq:HS_bound_L_N}.
	\end{proof}
	
	An analogue of the previous lemma also holds for the anisotropy term.
	
	\begin{lemma} \label{lem:HS_A_N}
	    Let $p>2$.
		For any $\delta>0$ and any compact set $K_0 \subset \R$, there exists $C>0$ such that the following holds.
		For any $N,q \geq 1$ and $\alpha \in [0,1]$, if there exists $M>0$ such that, for any $x \in K_0$ and $y \in (0,\delta]$, with $z = x+\ii y$,
		\begin{equation} \label{eq:ass_LL_2}
			\E_\alpha\left[\abs{s_N(z)-s_{V_\alpha}(z)}^{2q}\right] 
			\leq \left( \frac{M}{Ny} \right)^{2q},
		\end{equation}
		then, for any $\eta \in (0,(\delta \wedge 1)/2)$ and $f \in \cC_c^3(\R)$ supported in $K_0$, we have
		\begin{equation} \label{eq:HS_bound_A_N}
			\E_\alpha\left[\abs{A_N(f)}^q\right]
			\leq \left( CM^2 \left( \norme{f}_1 + \norme{f'}_1 + \log(1/\eta) \norme{f''}_1 + \eta \norme{f'''}_1 \right) \right)^q.
		\end{equation}
	\end{lemma}
	
	\begin{remark} \label{rem:HS_A}
		The parameter $\eta$ has to be thought of as small enough so that the term involving $\norme{f''}_1$ dominates (instead of $\norme{f'}_1$ in the case of $L_N(f)$ treated in Lemma \ref{lem:HS_L_N}).
		However, we typically want to apply this result to functions $f$ where $f''(x)$ explodes at 0 like $\abs{x}^{p-3}$, instead of $f'(x)$ exploding like $\abs{x}^{p-3}$ for linear statistics, see Remark \ref{rem:HS_L}.
		Similarly, a regularisation step will be required first to apply Lemma~\ref{lem:HS_A_N} to a function $f$ with $\norme{f'''}_1$ finite, although diverging with $N$.
	\end{remark}
	
	\begin{proof}
		The proof is similar to the one of Lemma \ref{lem:HS_L_N}:
		starting from \eqref{eq:HS_A_N(f)} and using the same integration by part tricks for the term involving $f'''(x)$, we get $\abs{A_N(f)}^q \leq (\frac{2}{\pi})^q (J_1^q + J_2^q + J_3^q + J_4^q)$, where
		\begin{align*}
			J_1 & \coloneqq N^2 \iint_{\R^2} \left(\abs{f(x)}+ \abs{yf'(x)} + y^2 \abs{f''(x)} \right) \abs{\chi'(y)} \cdot \abs{s_N(z)-s_{V_\alpha}(z)}^2 \diff x\diff y, \\
			J_2 & \coloneqq N^2 \iint_{\abs{y} <\eta} y^2 \abs{f'''(x)} \cdot \abs{s_N(z)-s_{V_\alpha}(z)}^2 \diff x\diff y, \\
			J_3 & \coloneqq N^2 \iint_{\abs{y} \geq \eta} \abs{\partial_y(y^2\chi(y)) f''(x)} \cdot \abs{s_N(z)-s_{V_\alpha}(z)}^2 \diff x \diff y, \\
			J_4 & \coloneqq 2N^2\eta^2 \int_{\R} \abs{f''(x)} \cdot \abs{s_N(x+\ii\eta)-s_{V_\alpha}(x+\ii\eta)}^2 \diff x.
		\end{align*}
		Then, the same method as before yields the bounds
		\begin{align*}
			\E_\alpha[J_1^q] 
			& \leq (CM^2 (\norme{f}_1+\norme{f'}_1+\norme{f''}_1))^q,
			\qquad 
			\E_\alpha[J_2^q] 
			\leq (CM^2 \eta \norme{f'''}_1)^q, \\
			\E_\alpha[J_3^q] 
			& \leq (CM^2 \norme{f''}_1 \log(1/\eta))^q,
			\hspace{2.25cm}
			\E_\alpha[J_4^q] 
			\leq (CM^2 \norme{f''}_1)^q,
		\end{align*}
		which concludes the proof.
	\end{proof}
	

When performing the master operator approach and the local law machinery to improve the bounds on the linear statistics, 
one of the intermediary versions of the local law contains a bound which is not of the form of the assumption \eqref{eq:ass_LL} in Lemma \ref{lem:HS_L_N}, because it contains an additional term proportional to $1/\abs{z}^2$.
The next lemma, also using the Helffer--Sjöstrand formula, will allow us to convert this intermediary local law into a moment estimate for $L_N(f)$ in the specific case where $f$ is a truncated version of $f_{\alpha,w}$ defined in \eqref{def:fz}.

\begin{lemma}\label{lem:HS_z^2}
Let $p>2$.
For any compact set $K\subset \C$, any $\delta>0$ and any $\phi\in \mathcal{C}^\infty_c(\R)$ supported in $[-1-\delta,1+\delta]$, there exists $C>0$ such that the following holds.
For any $N,q \geq 1$ and $\alpha \in [0,1]$, if there exists $M>0$ such that, for any $x \in [-1-\delta,1+\delta]$ and $y \in (0,\delta]$, with $z = x+\ii y$,
\begin{equation} \label{eq:ass_LL_z^2}
	\E_\alpha\left[ \abs{s_N(z)-s_{V_\alpha}(z)}^q \right] 
	\leq \frac{M^q}{N^q}\left( \frac{1}{y}+\frac{1}{\abs{z}^2} \right)^q,
\end{equation}
then, for any $w\in K$,
\begin{equation} \label{eq:HS_bound_L_N_fz}
	\E_\alpha\left[\abs{L_N(f_{\alpha,w}\phi)}^q\right]
	\leq C^qM^q.
\end{equation}
\end{lemma}

We make use in the proof of Lemma \ref{lem:f_z}, which is stated in the next section and contains bounds on $f_{\alpha,w}$ and its derivatives.

\begin{proof}
Let $\chi\in \cC_c^\infty(\R)$ be such that $\1_{[-\delta/2,\delta/2]} \leq \chi \leq \1_{[-\delta,\delta]}$.
First, by \eqref{eq:HS_L_N(f)}, denoting for $z=x+\ii y$
\[ 
F_{\alpha,w}(z) \coloneqq \ii y (f_{\alpha,w}\phi)''(x) \chi(y) + \ii(f_{\alpha,w}\phi)(x) + \ii y (f_{\alpha,w}\phi)'(x))\chi'(y),
\]
we get:
\begin{align}
	\left(\frac{2\pi}{N}\right)^q \E_\alpha[\abs{L_N(f_{\alpha,w}\phi)}^q]
	& = \E_\alpha\left[\abs{\iint_{\R^2}
		F_{\alpha,w}(z) (s_N-s_{V_\alpha})(z) \diff x \diff y}^q\right] \nonumber \\
	& \leq \int_{(\R^2)^q} 
	\left(\prod_{i=1}^q \abs{F_{\alpha,w}(z_i)}\right) 
	\E_\alpha\left[\prod_{i=1}^q \abs{s_N(z_i)-s_{V_\alpha}(z_i)} \right]\diff x_1 \diff y_1 \ldots \diff x_q \diff y_q \nonumber \\
	& \leq \left(\frac{M}{N} \iint_{\R^2} \abs{F_{\alpha,w}(z)} \left(\frac{1}{\abs{y}}+\frac{1}{\abs{z}^2}\right)\diff x \diff y\right)^q, \label{eq:HS_fz}
\end{align}
where in the last line we used Hölder's inequality as in \eqref{eq:Holder_for_HS} and the assumption \eqref{eq:ass_LL_z^2}.
It remains to bound the integral on the right-hand side.
Bounding $\chi(y) \leq \1_{\abs{y} \leq \delta}$ and $\abs{\chi'(y)} \leq C \1_{\abs{y} \in [\delta/2,\delta]}$, recalling that $\phi$ is supported in $[-1-\delta,1+\delta]$ with bounded derivatives and using Lemma~\ref{lem:f_z} (which gives controls on $f_{\alpha,w}$ and its derivatives), we get
\[ 
\abs{F_{\alpha,w}(z)} 
\leq C \1_{\abs{x} \leq 1+\delta} \left( \abs{y} \abs{x}^{-(3-p)_+} \1_{\abs{y} \leq \delta} + \1_{\abs{y} \in [\delta/2,\delta]} \right).
\]
Therefore, noting that $1/\abs{y} + 1/\abs{z}^2 \leq C$ for $\abs{y} \geq \delta/2$, we get
\begin{align} 
	& \iint_{\R^2} \abs{F_{\alpha,w}(z)} \left(\frac{1}{\abs{y}}+\frac{1}{\abs{z}^2}\right)
	\diff x \diff y \\
	& \leq C \int_{\abs{x}\leq 1+\delta} 
	\left( \abs{x}^{-(3-p)_+} \int_{\abs{y}\leq \delta} \left(1+\frac{\abs{y}}{x^2+y^2}\right) \diff y
	+ \int_{\abs{y} \in [\delta/2,\delta]} \diff y \right) \diff x \nonumber \\
	& \leq C \int_{\abs{x}\leq 1+\delta} 
	\left( \abs{x}^{-(3-p)_+} \left(1+ \log \frac{1}{x} \right)
	+ 1 \right) \diff x, \label{eq:int_F}
\end{align}
using that $\int_{\abs{y}\leq \delta} \frac{\abs{y}}{x^2+y^2} \diff y = \log \frac{x^2+\delta^2}{x^2}$.
The integral on the RHS of \eqref{eq:int_F} is a finite constant because $p>2$, so, coming back to \eqref{eq:HS_fz}, we get the result.
\end{proof}

\section{Bounds on the inverse of the master operator}
\label{sec:bounds_test_functions}

As explained in Section~\ref{sec:strategy}, the control of the linear statistics $L_N(f)$ via the master operator strategy requires to bound moments of $A_N\left(\Xi_\alpha^{-1}[f]\right)$, where we recall $\Xi_\alpha$ is the \textit{master operator} defined in \eqref{eq:def_master_op} as
\begin{equation*} 
\Xi_\alpha[\psi](x)
\coloneqq - \frac{1}{2} \psi(x) V_\alpha'(x) + \int_{-1}^1 \frac{\psi(x)-\psi(t)}{x-t} \diff \mu_{V_\alpha}(t).
\end{equation*}
In fact, this operator is invertible \textit{up to a constant}, meaning that for $f\in \mathcal{C}^2(\R)$, there is a constant $a$ (depending on $f$) such that the equation $\Xi_\alpha[\psi]=f-a$ has a unique solution $\psi$, see Lemma \ref{lem:inverse_master_op}. The presence of this constant does not matter since if two functions $f$ and $g$ differ by a constant, then the linear statistics coincide: $L_N(f)=L_N(g)$. 

This strategy will be used in Section~\ref{sec:bootstrap} to improve the bounds on $\E_\alpha\left[\abs{L_N(f_{\alpha,z})}^q\right]$ and in Section~\ref{sec:CLT} to prove the CLT for a generic function $f\in\cC^3(\R)$.
Therefore, in this section, we state bounds on the derivatives of $\Xi_\alpha^{-1}[f]$ both for generic $f \in \cC^3(\R)$ in Lemma~\ref{lem:inverse_master_op} and for $f = f_{\alpha,z}$ in Lemma~\ref{lem:psi_z} (note that $f_{\alpha,z} \notin \cC^3(\R)$, see Lemma~\ref{lem:f_z}).
In particular, as mentioned in Remark~\ref{rem:HS_A}, in order to control $A_N\left(\Xi_\alpha^{-1}[f]\right)$ \emph{via} the use of Helffer--Sjöstrand formula and local laws, we need to prove that $\Xi_\alpha^{-1}[f]''$ is locally integrable, which is indeed the case.
Finally, as explained in Remarks~\ref{rem:HS_L} and~\ref{rem:HS_A}, a regularization argument is needed before applying Helffer--Sjöstrand formula. The tool for this argument is proved in Lemma~\ref{lem:regularize}.
Most proofs are postponed to Appendix~\ref{app:technical estimate function}.

\subsection{Inverse of the master operator}

We state here the well-known inversion of the master operator $\Xi_\alpha$ and prove its continuity with respect to $\cC^k$-norms defined by $\norme{f}_{\cC^k}\coloneqq \sum_{i=0}^k \normeo{f^{(i)}}_\infty$. This type of results dates back to \cite[Eq.\@ (12), p.175]{tricomi1957integral} and \cite[Lemma 3.2]{BekFigGui2015}, but we rely here on arguments of \cite[Lemma 3.3]{BekLebSer2018}. 
The proof is postponed to Appendix \ref{subsec:inverse_master_op}.

\begin{lemma} \label{lem:inverse_master_op}
Let $p>2$.
There exists $C>0$ depending only on $p$ such that the following holds.
\begin{enumerate}
	\item\label{it:inverse_C2} Let $f \in \cC^2(\R)$ and $\alpha \in [0,1]$. Then, letting $a \coloneqq \int_{-1}^1 f(t) \frac{\diff t}{\sigma(t)}$, the function $\psi_\alpha$ defined by
	\begin{equation} \label{eq:def_psi}
		\psi_\alpha(\lambda) = 
		\begin{cases}
			\displaystyle
			-\frac{1}{\pi r_\alpha(\lambda)} 
			\int_{-1}^1 \frac{f(t)-f(\lambda)}{t-\lambda}
			\frac{\diff t}{\sigma(t)},
			& \text{if } \lambda \in [-1,1], \medskip \\
			\displaystyle 
			\frac{\displaystyle\int_{-1}^1 \frac{\psi_\alpha(t)}{\lambda-t} \diff \mu_{V_\alpha}(t)+f(\lambda)-a}{\displaystyle\int_{-1}^1 \frac{1}{\lambda-t} \diff \mu_{V_\alpha}(t) -\frac{1}{2} V_\alpha'(\lambda)},
			& \text{if } \lambda \in \R \setminus [-1,1],
		\end{cases}
	\end{equation}
	is $\cC^1$ on $\R$, satisfies $\Xi_\alpha[\psi_\alpha] = f-a$ and $\norme{\psi_\alpha}_{\cC^1} \leq C \norme{f}_{\cC^2}$. We write $\psi_\alpha = \Xi_\alpha^{-1}[f]$, with a slight abuse of notation.
	\item \label{it:inverse_C3}If moreover $f \in \cC^3(\R)$, then $\psi_\alpha \in \cC^2(\R^*)$ and, for any $\lambda \in \R$,
	\[
	\abs{\psi_\alpha''(\lambda)} \leq C \norme{f}_{\cC^3} \times
	\begin{cases}
	    (1+\abs{\lambda}^{p-3}) & \text{if } p < 3, \\
	    (1+(\log1/\abs{\lambda})_+) & \text{if } p = 3, \\
	    1 & \text{if } p>3.
	\end{cases}
	\]
\end{enumerate}
\end{lemma}

\subsection{Bounds on \texorpdfstring{$f_{\alpha,z}$}{falpha,z} and its preimage via the master operator}

The following lemma establishes bounds on $f_{\alpha,z}$, defined in \eqref{def:fz}, and its derivatives on compact sets.
It is proved in Appendix~\ref{subsec:fz_and_psiz}.

\begin{lemma} \label{lem:f_z}
    Let $p>2$.
    For all $\alpha \in [0,1]$ and $z \in \C\setminus\R$,  $f_{\alpha,z}\in\cC^2(\R)\cap\cC^3(\R^*)$.
    Moreover, for any compact set $K \subset \C$, there exists $C>0$ such that, for any $\alpha \in [0,1]$, $z = x +\ii y  \in K \setminus \R$ and $\lambda \in K\cap\R$, 
    \[
    \abs{f_{\alpha,z}(\lambda)} \leq C, 
    \qquad 
    \abs{f_{\alpha,z}'(\lambda)} \leq C, 
    \qquad 
    \abs{f_{\alpha,z}''(\lambda)} \leq C (\abs{x} \vee \abs{\lambda})^{-(3-p)_+}, 
    \]
    and, if $\lambda \neq 0$,
    \[
    \abs{f_{\alpha,z}'''(\lambda)} \leq \frac{C}{\abs{z}}
    \abs{\lambda}^{-(3-p)_+}.
    \]
\end{lemma}

As mentioned at the beginning of the section, in order to control the linear statistics $L_N(f_{\alpha,z})$ appearing in the loop equations, we need to study the function 
\begin{equation}
\psi_{\alpha,z} \coloneqq \Xi_\alpha^{-1}[f_{\alpha,z}],
\end{equation}
which is well-defined by Lemma~\ref{lem:inverse_master_op} because $f_{\alpha,z}\in\cC^2(\R)$.
The next lemma establishes bounds on $\psi_{\alpha,z}$ and its first two derivatives. 
It is also proved in Appendix~\ref{subsec:fz_and_psiz}.

\begin{lemma} \label{lem:psi_z}
    Let $p>2$.
    For all $\alpha \in [0,1]$ and $z \in \C\setminus\R$,  $\psi_{\alpha,z}\in\cC^1(\R)\cap\cC^2(\R^*)$.
    Moreover, for any compact set $K \subset \C$, there exists $C>0$ such that, for any $\alpha \in [0,1]$, $z = x +\ii y  \in K \setminus \R$ and $\lambda \in K\cap\R$, 
    \[
    \abs{\psi_{\alpha,z}(\lambda)} \leq C, 
    \qquad 
    \abs{\psi_{\alpha,z}'(\lambda)} \leq C, 
    \]
    and, if $\lambda \neq 0$,
    \[
    \abs{\psi_{\alpha,z}''(\lambda)} \leq \frac{C}{\abs{z}}
    \times
	\begin{cases}
	    (1+\abs{\lambda}^{p-3}) & \text{if } p < 3, \\
	    (1+(\log1/\abs{\lambda})_+) & \text{if } p = 3, \\
	    1 & \text{if } p>3.
	\end{cases}.
    \]
\end{lemma}

Let $\phi\in \mathcal{C}^\infty_c$ such that 
\[
\1_{[-1-\delta_0/2,1+\delta_0/2]}\leq \phi \leq \1_{[-1-\delta_0,1+\delta_0]},
\]
where $\delta_0$ is as in Lemma \ref{lem:stability}. From Lemma \ref{lem:psi_z}, we deduce the following regularity properties for
\begin{equation}
\label{def:widetilde_psi}
\widetilde\psi_{\alpha,z} \coloneqq \Xi_{\alpha}^{-1}[f_{\alpha,z}\phi].
\end{equation} In contrast with the estimates of Lemma \ref{lem:psi_z}, the regularity bounds for $\widetilde\psi_{\alpha,z}(\lambda)$ are valid for $\lambda\in \R$.
\begin{corollary} \label{cor:psi_z_tilde}
    Let $p>2$.
    For all $\alpha \in [0,1]$ and $z \in \C\setminus\R$,  $\widetilde\psi_{\alpha,z}\in\cC^1(\R)\cap\cC^2(\R^*)$.
    Moreover, for any compact set $K \subset \C$, there exists $C>0$ such that, for any $\alpha \in [0,1]$, $z = x +\ii y \in K \setminus \R$ and $\lambda \in \R$, 
    \[
    \abs{\widetilde\psi_{\alpha,z}(\lambda)} \leq \frac{C}{1+\abs{V_\alpha'(\lambda)}}, 
    \qquad 
    \abs{\widetilde\psi_{\alpha,z}'(\lambda)} \leq C, 
    \]
    and, if $\lambda \neq 0$,
    \[
    \abs{\widetilde\psi_{\alpha,z}''(\lambda)} \leq \frac{C}{\abs{z}}
    \times
	\begin{cases}
	    (1+\abs{\lambda}^{p-3}) & \text{if } p < 3, \\
	    (1+(\log1/\abs{\lambda})_+) & \text{if } p = 3, \\
	    1 & \text{if } p>3.
	\end{cases}.
    \]
\end{corollary}

\begin{proof}
In view of \eqref{eq:def_psi}, and since $\phi=1$ on $[-1-\delta_0/2,1+\delta_0/2]$, we have $\widetilde\psi_{\alpha,z}=\psi_{\alpha,z}$ on $[-1-\delta_0/2,1+\delta_0/2]$. The claimed bounds for $\lambda\in (-1-\delta_0/2,1+\delta_0/2)$ follow from Lemma \ref{lem:psi_z}. We now check that these bounds hold for $\lambda \in \R\setminus[-1-\delta_0/2,1+\delta_0/2]$.
The denominator in the definition of $\widetilde\psi_{\alpha,z}(\lambda)$ in \eqref{eq:def_psi} satisfies, by \eqref{eq:rewriting_G} and then Lemma~\ref{lem:bound on ralpha}, for any $\abs{\lambda} \geq 1+\delta_0/2$,
\begin{equation} \label{eq:LB_G}
    \abs{\int_{-1}^1 \frac{1}{\lambda-t} \diff \mu_{V_\alpha}(t) -\frac{1}{2} V_\alpha'(\lambda)}
    = r_\alpha(\lambda) \sqrt{\lambda^2-1}
    \geq c \abs{\lambda} \left(1+\alpha\abs{\lambda}^{p-2} \right)
    \geq c(1+\abs{V_\alpha'(\lambda)}).
\end{equation}
Now, denote 
\[
\mathcal{I}_{\alpha,z} \colon \lambda\mapsto \int_{-1}^1 \frac{\widetilde\psi_{\alpha,z}(t)}{\lambda-t}\diff\mu_{V_\alpha}(t), 
\qquad 
\mathcal{J}_{\alpha,z} \colon \lambda\mapsto\int_{-1}^1 \frac{1}{\lambda-t}\diff\mu_{V_\alpha}(t).
\]
By differentiation under the integral and a crude bound, one checks that both functions are twice differentiable (even $\mathcal{C}^\infty$) on $\R\setminus[-1-\delta_0/4,1+\delta_0/4]$, with $\abso{\mathcal{I}^{(n)}(\lambda)},\abso{\mathcal{J}^{(n)}(\lambda)} \leq C$ for $n\in \{0,1,2\}$ with some $C$ independent of $\alpha\in [0,1]$ and $z\in K\setminus \R$ (in fact, one checks by dominated convergence that $\abso{\mathcal{I}^{(n)}(\lambda)},\abso{\mathcal{J}^{(n)}(\lambda)} \leq C/\abs{\lambda}^n$).
Recalling the definition of $\widetilde\psi_{\alpha,z}$ in \eqref{eq:def_psi}, the result follows by direct computation of the derivatives of $\widetilde\psi_{\alpha,z}$, using \eqref{eq:LB_G} to lower bound the denominator.
\end{proof}

\subsection{Regularizing test functions}

As explained in Remark~\ref{rem:HS_A}, we need to regularize test functions before applying Helffer--Sjöstrand formula \emph{via} Lemma~\ref{lem:HS_A_N}.
This regularization argument is contained in the result below, the function $\varphi$ has to be thought as $\Xi_\alpha^{-1}[f]$ for either a generic $f \in \cC^3(\R)$ or $f = f_{\alpha,z}$.

\begin{lemma}[Regularization lemma] \label{lem:regularize}
Let $\gamma \in [0,1)$.
Let $M>0$ and $\varphi \in \cC^1(\R) \cap \cC^2(\R^*)$ such that $\abs{\varphi''(\lambda)} \leq M (1+\abs{\lambda}^{-\gamma})$ for any $\lambda \in \R$. 
Then, for any $\varepsilon \in (0,1]$, there exists $\varphi_\varepsilon \in \cC^\infty(\R)$ such that 
\[
\supp \varphi_\varepsilon \subset \supp \varphi + [-\varepsilon,\varepsilon], \qquad 
\norme{\varphi-\varphi_\varepsilon}_\infty \leq \varepsilon \norme{\varphi'}_\infty,
\qquad
\norme{\varphi'-\varphi_\varepsilon'}_\infty \leq C M \varepsilon^{1-\gamma}
\]
and, for any $\lambda \in \R$,
\[
\abs{\varphi_\varepsilon''(\lambda)} \leq C M \left(1+ (\abs{\lambda} \vee \varepsilon)^{-\gamma} \right)
\qquad \text{and} \qquad 
\abs{\varphi_\varepsilon'''(\lambda)} \leq \frac{C M}{\varepsilon} \left(1+ \abs{\lambda}^{-\gamma} \right),
\]
where $C>0$ depends only on $\gamma$.
\end{lemma}

\begin{proof}
Let $\xi \in \cC^\infty(\R)$ be such that $\supp \xi \subset [-1,1]$, $0 \leq \xi \leq 1$ and $\int_{\R} \xi(x) \diff x = 1$.
For $\varepsilon \in (0,1]$, let $\xi_\varepsilon = \frac{1}{\varepsilon} \xi(\cdot/\varepsilon)$ and $\varphi_\varepsilon = \varphi * \xi_\varepsilon$.
Then, these are standard facts that $\varphi_\varepsilon \in \cC^\infty(\R)$, $\supp \varphi_\varepsilon \subset \supp \varphi + [-\varepsilon,\varepsilon]$, 
$\norme{\varphi-\varphi_\varepsilon}_\infty \leq \varepsilon \norme{\varphi'}_\infty$, $\varphi_\varepsilon' = \varphi' * \xi_\varepsilon$, $\varphi_\varepsilon'' = \varphi'' * \xi_\varepsilon$ and $\varphi_\varepsilon''' = \varphi'' * \xi_\varepsilon'$.


Let $\lambda \in \R$. We have
\[
\abs{\varphi'(\lambda)-\varphi_\varepsilon'(\lambda)}
= \abs{\int_{-\varepsilon}^\varepsilon (\varphi'(\lambda)-\varphi'(\lambda-x)) \xi_\varepsilon(x) \diff x}
\leq \int_{-\varepsilon}^\varepsilon \left( \int_{\lambda-\varepsilon}^{\lambda+\varepsilon} \abs{\varphi''(u)} \diff u \right) \xi_\varepsilon(x) \diff x,
\]
and, by assumption,
\begin{equation} \label{eq:bound_2nd_derivative_varphi}
	\int_{\lambda-\varepsilon}^{\lambda+\varepsilon} \abs{\varphi''(u)} \diff u 
	\leq \int_{\lambda-\varepsilon}^{\lambda+\varepsilon} M (1+\abs{u}^{-\gamma}) \diff u 
	\leq 2 \varepsilon M + \begin{cases}
		2 \varepsilon M \abs{\lambda/2}^{-\gamma} & \text{if } \abs{\lambda} \geq 2\varepsilon, \\
		\frac{2}{1-\gamma} M (3\varepsilon)^{1-\gamma} & \text{if } \abs{\lambda} < 2\varepsilon,
	\end{cases}
\end{equation}
bounding $\abs{u} \geq \abs{\lambda} - \varepsilon \geq \abs{\lambda/2}$ in the first case and computing explicitly the integral $\int_{-3\varepsilon}^{3\varepsilon} M \abs{u}^{-\gamma} \diff u$ in the second one. 
In particular, this proves $\norme{\varphi'-\varphi_\varepsilon'}_\infty \leq C M \varepsilon^{1-\gamma}$ because $\int_{-\varepsilon}^{\varepsilon} \xi_\varepsilon(x) \diff x = 1$.
For the 2nd derivative, using $\xi \leq 1$, we have
\[
\abs{\varphi_\varepsilon''(\lambda)}
= \abs{\int_{-\varepsilon}^\varepsilon \varphi''(\lambda-x) \xi_\varepsilon(x) \diff x}
\leq \frac{1}{\varepsilon} \int_{\lambda-\varepsilon}^{\lambda+\varepsilon} \abs{\varphi''(u)} \diff u
\leq C M \left(1+ (\abs{\lambda} \vee \varepsilon)^{-\gamma} \right),
\]
where the last inequality follows from \eqref{eq:bound_2nd_derivative_varphi}.
The bound on the third derivative follows similarly, by using $\varphi_\varepsilon''' = \varphi'' * \xi_\varepsilon'$ together with $\norme{\xi_\varepsilon'}_\infty \leq \norme{\xi'}_\infty/\varepsilon^2$.
\end{proof}

\section{The bootstrap argument to get the local law}
\label{sec:bootstrap}

We now have the main ingredients to run the local law machinery. We will state the first \textit{a priori} bounds on the moments of $L_N(f)$ and $A_N(f)$, defined in \eqref{eq:def_linear_statistics} and \eqref{eq:def_anisotropy}, and a truncation Lemma that will allow us to only consider compactly supported functions. Together these results will allow us to deduce a first bound $\Ec{\abs{L_N(f_{\alpha,z})}^q}$ and thus a first local law via Proposition \ref{prop:local_law_machinery}. We will see how to deduce a better \textit{a priori} bound on $\E_\alpha\left[\abs{L_N(f_{\alpha,z})}^q\right]$ in Lemma \ref{lem:better_bound} via the use of Helffer--Sjöstrand formula and the bounds proven in the previous section.

\subsection{A priori bounds}
\label{sec:a_priori}

We first gather known bounds to control outliers.
The first one bounds the tail of the one-point function $p_N^\alpha$, which is the density of the marginal distribution of $\lambda_1$ under $\P_\alpha$. 
The second one bounds the probability of having an outlier, hence providing a control closer to the limiting spectrum $[-1,1]$.

\begin{lemma} \label{lem:outliers}
Let $p>1$.
\begin{enumerate}
    \item \label{it:bound_one_point_function} There exist $X_1\geq 1$ and $c_1>0$
    such that, for any $N \geq 1$, $\alpha \in [0,1]$ and $\abs{x} \geq X_1$, 
    \begin{equation*}
        p_N^\alpha(x) \leq e^{-c_1 N V_{\alpha}(x)}.
    \end{equation*}
    \item \label{it:ldp estimate} For any $\varepsilon >0$, there exist $C,c_2>0$ such that, for any $N \geq 1$ and $\alpha \in [0,1]$,
    \begin{equation*}
    \P_\alpha ( \exists k \in \llbracket1,N\rrbracket, \abs{\lambda_k} \geq 1+\varepsilon )
    \leq Ce^{-c_2 N}.
    \end{equation*}
\end{enumerate}
\end{lemma}

\begin{proof}
Part \ref{it:bound_one_point_function}. This is proved in \cite[Theorem~2.2.(i)]{PasShc2008}: the fact that $X_1$ can be chosen uniformly for $\alpha \in [0,1]$ is explicitly stated there, and taking a look at their proof shows that this is also the case for $c_1$, which ultimately depends on the constants they call $\beta$, $\varepsilon$, $m$, $M$.

Part \ref{it:ldp estimate}. This follows from the large deviation principle for the extreme eigenvalues proved for example in \cite[Proposition 2.1]{BorGui2013}. Uniformity in $\alpha$ of $C$ and $c_2$ is not stated there, but it follows from the proof that the large deviation principle holds uniformly in $\alpha$: we deduce that, uniformly in $\alpha$,
\[
    \limsup_{N\to\infty} \frac{1}{N} \log \P_\alpha (\exists k \in \llbracket1,N\rrbracket, \abs{\lambda_k} \geq 1+\varepsilon)
    \leq - \inf_{\abs{x}\geq1+\varepsilon} V_{\mrm{eff},\alpha}(x)
    = - \inf_{x\geq1+\varepsilon} V_{\mrm{eff},\alpha}(x),
\]
where $V_{\mrm{eff},\alpha}$ was defined in \eqref{eq:characterization muValpha} and we used that it is even. Now by strict convexity of $V_\alpha$, $V_{\mrm{eff},\alpha}$ is strictly convex on $[1,\infty)$, but it is also non-negative and equal to 0 at 1, so it is increasing on $[1,\infty)$. 
Therefore, $\inf_{x\geq1+\varepsilon} V_{\mrm{eff},\alpha}(x) = V_{\mrm{eff},\alpha}(1+\varepsilon) > 0$.
Moreover, $\alpha \in [0,1] \mapsto V_{\mrm{eff},\alpha}(1+\varepsilon)$ is continuous so the following minimum is well-defined and positive:
\[
c_2 \coloneqq \frac{1}{2} \min_{\alpha\in[0,1]} V_{\mrm{eff},\alpha}(1+\varepsilon) > 0,
\]
and we deduce the result with this choice of $c_2$.
\end{proof}

We now state first rough \textit{a priori} bounds on moments of linear statistics and anisotropy terms. 
They are obtained as a consequence of \cite[Corollary~4.16]{Gui2019}, which is proved for any $\cC^1$ potential~$V$.

\begin{lemma}[\textit{A priori} bounds] \label{lem:a_priori}
Let $p>1$. Let $K \subset \R$ be a compact set.
There exists $C>0$ such that, for any $N\geq 2$, $\alpha\in[0,1]$ and $q \geq 1$, the following holds.
\begin{enumerate}
	\item\label{it:a_priori_L} For any $f \in \cC^1(\R)$ with $\supp f \subset K$,
	\begin{equation*} 
		\E_\alpha\left[\abs{L_N(f)}^q\right]
		\leq \left( C\sqrt{q N \log N} \norme{f}_{\cC^1} \right)^q.
	\end{equation*}
	\item\label{it:a_priori_A} For any $f \in \cC^2(\R)$ with $\supp f \subset K$,
	\begin{equation*} 
		\E_\alpha\left[\abs{A_N(f)}^q\right]
		\leq \left( Cq N \log N \norme{f}_{\cC^2} \right)^q.
	\end{equation*}
\end{enumerate} 
\end{lemma}

\begin{proof}
This result is a consequence of \cite[Corollary~4.16]{Gui2019}, which is stated for a single $\cC^1$ potential $V$, but actually holds for the potentials $V_\alpha$ uniformly in $\alpha$.
Indeed, a read through the proofs of Lemma~4.14 and of its consequence Corollary~4.16 in \cite{Gui2019} shows that the constants appearing in these results can be chosen uniformly for any class of potentials  $\mathcal{V}$ such that there exists $M,C,c>0$ such that, for any $V \in \mathcal{V}$, $\P(\max \abs{\lambda_i} > M) \leq C e^{-cN}$ and $\sup_{\abs{x} \leq M+1} \abs{V'(x)} \leq C$.

Part \ref{it:a_priori_L}.
In the first part of \cite[Corollary~4.16]{Gui2019}, two norms appear: The Lipschitz norm $\norme{f}_L$ which is by Taylor's inequality less than $\norme{f}_{\cC^1}$, and the Sobolev $1/2$-norm $\norme{f}_{1/2}$ which is by \cite[Property~4.13]{Gui2019} less than $2(\norme{f}_{L^2} + \norme{f'}_{L^2}) \leq C \norme{f}_{\cC^1}$ with $C$ depending on the length of $K$.
Hence, this result implies that there exists $C_0,c_0>0$ independent of $\alpha$ such that
\[
\mathbb{P}_\alpha\left(\abs{L_N(f)} \geq C_0 \sqrt{N \log N} \norme{f}_{\cC^1}\right)
\leq e^{-c_0N}.
\]
Bounding $\abs{L_N(f)} \leq N \norme{f}_{\infty}$ on this event, it follows that 
\begin{equation} \label{eq:a_priori_L_1}
	\E_\alpha\left[\abs{L_N(f)}^q\right]
	\leq \left( C_0 \sqrt{N \log N} \norme{f}_{\cC^1} \right)^q
	+ \left( N \norme{f}_{\infty} \right)^q e^{-c_0N}.
\end{equation}
On the other hand, note that $(c_0 N)^q / q! \leq e^{c_0N}$, so we have $N^q \leq (q/c_0)^q e^{c_0N}$, which yields $N^{q/2} \leq (q/c_0)^{q/2} e^{c_0N/2}$. 
Applying this to the second term on the right-hand side of \eqref{eq:a_priori_L_1} yields the result.

Part \ref{it:a_priori_A}. The proof is similar: the second part of \cite[Corollary~4.16]{Gui2019} shows that there exists $C_0,c_0>0$ independent of $\alpha$ such that $\mathbb{P}_\alpha\left(\abs{A_N(f)} \geq C_0 N \log N \norme{f}_{\cC^2}\right) 
\leq e^{-c_0N}$.
We combine this with the deterministic bound $\abs{A_N(f)} \leq 4N^2 \norme{f'}_{\infty}$ and then the bound $N^q \leq (q/c_0)^q e^{c_0N}$ as before.
\end{proof}

In order to restrict ourselves to compactly supported test functions, we also need the following result, which is obtained as a consequence of results in \cite{DadFraGueZit2023}.

\begin{lemma}[Truncation lemma] \label{lem:truncation}
Let $p>1$.
Let $\delta_0>0$ be the constant given by Lemma~\ref{lem:stability}.
There exist $C,c >0$ such that the following holds for any 
$N,q\geq1$, $\alpha\in[0,1]$ and $\phi \in \cC_c^\infty(\R)$ satisfying
\begin{equation} \label{eq:assumption_phi}
    \1_{[-1-\delta_0/2,1+\delta_0/2]}
    \leq \phi 
    \leq \1_{[-1-\delta_0,1+\delta_0]}.
\end{equation}
\begin{enumerate}
	\item\label{it:truncation1} For any $M>0$ and any $f \colon \R \to \R$ such that $\abs{f(x)} \leq M(1+V_\alpha(x))$ for all $x \in \R$,
	\[
	\E_\alpha\left[\abs{L_N(f) - L_N(f\phi)}^q\right] \leq (CMq)^q e^{-cN}.
	\]
	\item\label{it:truncation2}If $q \leq \sqrt{N}$, then for any $M,\gamma>0$ and any $f \colon \R \to \R$ such that $\abs{f(x)} \leq M e^{\gamma\abs{x}}$ for all $x \in \R$,
	\[
	\E_\alpha\left[\abs{L_N(f) - L_N(f\phi)}^q\right] 
	\leq (CM)^q e^{-cN}.
	\]  
	\item\label{it:truncation3} For any $M>0$ and any differentiable $f \in \cC^1(\R)$ such that $\abs{f'(x)} \leq M$ for all $x \in \R$ and $\abs{f(x)} \leq M$ for all $\abs{x} \leq 1+\delta_0$,
	\[
	\E_\alpha \left[ \abs{A_N(f) - A_N(f\phi)}^q \right] 
	\leq (C M \left( 1 + \norme{\phi'}_\infty \right))^q q^{2q} e^{-cN}.
	\]
\end{enumerate}
\end{lemma}

\begin{proof}
Part \ref{it:truncation1}.
The proof proceeds in two steps: we first restrict the support to a large but bounded interval, and then to $[-1-\delta_0,1+\delta_0]$.
For the first step, let $X_1 \geq 1$ be the constant appearing in Lemma~\ref{lem:outliers}.\ref{it:bound_one_point_function} and $\phi_0 \in \cC_c^\infty(\R)$ be such that
\[
\1_{[-X_1,X_1]}
\leq \phi_0
\leq \1_{[-X_1-1,X_1+1]}.
\]
Then, using the assumption on $f$, we get
\[
\abs{L_N(f) - L_N(f\phi_0)}^q 
\leq \left( \sum_{k=1}^N M(V_\alpha(\lambda_k) + 1) \1_{\abs{\lambda_k} \geq X_1} \right)^q
\leq N^{q-1} \sum_{k=1}^N M^q (V_\alpha(\lambda_k) + 1)^q \1_{\abs{\lambda_k} \geq X_1},
\]
by Jensen's inequality.
Using Lemma~\ref{lem:outliers}.\ref{it:bound_one_point_function}, we deduce
\begin{align*}
	\E_\alpha\left[\abs{L_N(f) - L_N(f\phi_0)}^q\right]
	& \leq 2 (NM)^q \int_{X_1}^\infty (V_\alpha(x) + 1)^q e^{-c_1 N V_\alpha(x)} \diff x \\
	& \leq 2 (NM)^q \int_{N V_\alpha(X_1)}^\infty \left( \frac{Cy}{N} \right)^q e^{-c_1 y} \frac{\diff y}{N V_\alpha'(X_1)},
\end{align*}
using that $V_\alpha(x) + 1 \leq C V_\alpha(x)$ for $x \geq X_1 \geq 1$ and changing variables with $y = N V_\alpha(x)$, noting that $V_\alpha'$ is increasing on $[X_1,\infty)$ to bound $\diff x \leq \diff y/[N V_\alpha'(X_1)]$.
Using $y^q \leq (2q/c_1)^q e^{c_1y/2}$ and $N V_\alpha'(X_1) \geq c$, we get
\begin{equation} \label{eq:truncation_step_1}
	\E_\alpha\left[\abs{L_N(f) - L_N(f\phi_0)}^q\right]
	\leq (C M q)^q \int_{N V_\alpha(X_1)}^\infty e^{-c_1 y /2} \diff y
	\leq (CMq)^q e^{-c_1 N V_\alpha(X_1)/2},
\end{equation}
which concludes the first step of the proof.
Now, consider $\phi \in \cC_c^\infty(\R)$ as in the statement.
By Lemma~\ref{lem:outliers}.\ref{it:ldp estimate}, there exists $C,c_2>0$ independent of $\alpha$ such that
\begin{equation}
	\P_\alpha ( \exists k \in \llbracket1,N\rrbracket, \abs{\lambda_k} \geq 1+\delta_0/2 )
	\leq C e^{-c_2 N}.
\end{equation}
Bounding $\abs{L_N(f\phi_0) - L_N(f\phi)} \leq NM ((X_1+1)^2+1)$ on the event above, and noting that $L_N(f\phi_0) = L_N(f\phi)$ on the complement, we get
\begin{equation} \label{eq:truncation_step_2}
	\E_\alpha\left[\abs{L_N(f\phi_0) - L_N(f\phi)}^q\right]
	\leq (CNM)^q e^{-c_2 N}
	\leq (CMq)^q e^{-c_2 N /2},
\end{equation}
using that $N^q \leq (2q/c_2)^q e^{c_2N/2}$.
Combining \eqref{eq:truncation_step_1} and \eqref{eq:truncation_step_2} yields Part \ref{it:truncation1}.

Part \ref{it:truncation2}. We follow the same two steps as before.
In the first step, we bound
\[
\abs{L_N(f) - L_N(f\phi_0)}^q 
\leq \left( \sum_{k=1}^N M e^{\gamma\abs{\lambda_k}} \1_{\abs{\lambda_k} \geq X_1} \right)^q
\leq N^{q-1} \sum_{k=1}^N M^q e^{\gamma q \abs{\lambda_k}} \1_{\abs{\lambda_k} \geq X_1},
\]
by Jensen's inequality again. Then, using Lemma~\ref{lem:outliers}.\ref{it:bound_one_point_function}, we deduce
\begin{equation}
	\E_\alpha\left[\abs{L_N(f) - L_N(f\phi_0)}^q\right]
	\leq 2 (NM)^q \int_{X_1}^\infty e^{\gamma qx-c_1 N x^2} \diff x
	\leq (C M N)^q e^{-c_1 \gamma N X_1^2/2},
\end{equation}
where in the second inequality we noted that, by choosing $X_1 \geq 2 \gamma/c_1$, we have $e^{\gamma qx-c_1 N x^2} \leq e^{-\gamma c_1 N X_1 x/2}$ for any $x \geq X_1$ and $q \leq \sqrt{N} \leq N$. This concludes the first step.
For the second step, we use the first inequality in \eqref{eq:truncation_step_2} and conclude that, for some $C,c>0$,
\[
\E_\alpha\left[\abs{L_N(f) - L_N(f\phi)}^q\right] \leq (CMN)^q e^{-cN}
\leq (CM)^q e^{-cN},
\] 
up to a modifications of $C,c$ in the second inequality where we used $q \leq \sqrt{N}$.

Part \ref{it:truncation3}. 
Note that on the event $\{\forall k \in \llbracket1,N\rrbracket, \abs{\lambda_k} < 1+\delta_0/2 \}$, we have $A_N(f)=A_N(f\phi)$. On the complement, we bound deterministically 
\[
\abs{A_N(f) - A_N(f\phi)} 
\leq \abs{A_N(f)} + \abs{A_N(f\phi)}
\leq 4N^2 \left( \norme{f'}_\infty + \norme{(f\phi)'}_\infty \right),
\]
and then apply Lemma~\ref{lem:outliers}.\ref{it:ldp estimate}.
Moreover, we note that $\norme{(f\phi)'}_\infty \leq \norme{f' \phi'}_\infty + \norme{f\phi'}_\infty \leq M (1+\norme{\phi'}_\infty)$ using the assumption on $f$ and that $\phi'=0$ outside $[-1-\delta_0,1+\delta_0]$.
This yields
\[
\E_\alpha\left[\abs{A_N(f) - A_N(f\phi)}^q\right]
\leq (C M \left( 1 + \norme{\phi'}_\infty \right))^q N^{2q} e^{-c_2N}.
\]
Using that $N^{2q} \leq (4q/c_2)^{2q} e^{c_2N/2}$ concludes the proof.
\end{proof}

The two previous results allow us to deduce a first \textit{a priori} bound on $\E_\alpha\left[\abs{L_N(f_{\alpha,z})}^q\right]$. This bound is very crude and will be improved later.

\begin{corollary} \label{cor:a_priori_f_z}
Let $p>2$.
Let $K \subset \C$ be a compact set. Then there exists $C>0$ such that, for any $N \geq 1$, $\alpha\in[0,1]$,  $z \in K \setminus \R$ and $q \geq 1$,
\begin{equation*} 
	\E_\alpha\left[\abs{L_N(f_{\alpha,z})}^q\right] \leq (Cq\sqrt{N \log N})^q.
\end{equation*}
\end{corollary}

\begin{proof}
Let $M\geq 1$ be such that $\re K \subset [-M,M]$.
By Lemma~\ref{lem:f_z}, there exists $C>0$ such that, for any $\alpha \in [0,1]$, $z \in K \setminus \R$ and $\lambda \in [-2M,2M]$, $\abs{f_{\alpha,z}(\lambda)} \leq C$.
Moreover, for $\lambda \notin [-2M,2M]$, we have 
\[
    \abs{f_{\alpha,z}(\lambda)} 
    = \frac{\abs{g_\alpha(\lambda)-g_\alpha(z)}}{\abs{\lambda-z}}
    \leq \frac{\abs{g_\alpha(\lambda)}+\abs{g_\alpha(z)}}{\abs{\lambda}/2}
    \leq \frac{C V_\alpha(\lambda)}{\abs{\lambda}},
\]
using $\abs{g_\alpha(\lambda)} \leq CV_\alpha(\lambda)$ and that $g_\alpha$ is bounded on $K$ (see Remark~\ref{rem:g}).
In particular, this proves there exists $C>0$ such that, for any $\alpha \in [0,1]$, $z \in K \setminus \R$ and $\lambda \in\R$, $\abs{f_{\alpha,z}(\lambda)} \leq C(V_\alpha(\lambda) + 1)$.
Therefore, by Lemma~\ref{lem:truncation}.\ref{it:truncation1}, for any $\phi \in \cC_c^\infty(\R)$ satisfying \eqref{eq:assumption_phi}, there exists $C>0$ such that, for any $\alpha \in [0,1]$, $z \in K \setminus \R$ and $N,q \geq 1$,
\begin{equation} \label{eq:truncation_f_z}
    \E_\alpha\left[\abs{L_N(f_{\alpha,z})}^q \right]
    \leq 2^q \E_\alpha\left[\abs{L_N(f_{\alpha,z}\phi)}^q\right] + (Cq)^q.
\end{equation}
Then, the result follows from Lemma~\ref{lem:a_priori}.\ref{it:a_priori_L}, noting that $\norme{f_{\alpha,z} \phi}_{\cC^1}$ is uniformly bounded for $z \in K$ by Lemma~\ref{lem:f_z}.
\end{proof}

Combining this with Proposition~\ref{prop:local_law_machinery}.\ref{it:in_the_trapezoid} shows the following result. 
Note that we do not apply Proposition~\ref{prop:local_law_machinery}.\ref{it:out_the_trapezoid} here to obtain a local law outside of the trapezoid region $\cT$, because the resulting bound would be too bad due to the fact that $B_{\max} = C\sqrt{N\log N}$ appears squared in this bound. 
In the proof of Lemma~\ref{lem:better_bound} below, we use instead the \textit{a priori} bounds of Lemma~\ref{lem:a_priori} to control what happens outside of $[-1,1]$.

\begin{corollary} \label{cor:a_priori_local_law}
Let $p>2$.
Let $\delta_0$ and $\cT$ be as in Lemma \ref{lem:stability}. 
There exists $C>0$ such that, for any $N \geq 2$, $q \geq 1$, $\alpha\in[0,1]$ and $z \in \cT$,
\[
\E_\alpha\left[\abs{s_N(z)-s_{V_\alpha}(z)}^q\right] \leq \left( \frac{Cq \sqrt{\log N}}{\sqrt{N} y} \right)^q.
\]
\end{corollary}

\subsection{Master operator approach to bound \texorpdfstring{$L_N(f_{\alpha,z})$}{LN(falpha,z)}}

We now explain in this subsection how to improve the bound on $\Ec{L_N(f_{\alpha,z})^q}$ using Corollary \ref{cor:a_priori_local_law}. 
For this, we use the master operator strategy, using a bound on the Laplace transform of $L_N(f_{\alpha,z})$ obtained by a change of variables in the partition function $Z_N^{\alpha}$.
This strategy is used for example in \cite[Proposition~4.1]{BekLebSer2018} or \cite[Lemma~4.24]{Gui2019}.

\begin{lemma} \label{lem:better_bound}
Let $p>2$.
Let $K \subset \C$ be a compact set. Then there exists $C>0$ such that, for any $N \geq 2$, $q \geq 1$, $\alpha\in[0,1]$ and $z \in K$,
\begin{equation}
	\E_\alpha\left[|L_N(f_{\alpha,z})|^q\right] \leq \left( Cq^2 (\log N)^2 \frac{1}{\abs{z}}  \right)^q.
\end{equation}
\end{lemma}

\begin{proof}
Let $z\in K\setminus \R$, $\phi\in \mathcal{C}^\infty_c(\R)$ such that $\1_{[-1-\delta_0/2,1+\delta_0/2]} \leq \phi \leq \1_{[-1-\delta_0,1+\delta_0]}$, where $\delta_0$ is as in Lemma \ref{lem:stability}, and recall $\widetilde\psi_{\alpha,z} \coloneqq \Xi_{\alpha}^{-1}[f_{\alpha,z}\phi]$. All constants $C$ appearing below are independent of $\alpha$ and $z$. We make the change of variables 
\[
\lambda_k=y_k + \frac{\widetilde\psi_{\alpha,z}(y_k)}{N},\quad 1\leq k \leq N,
\]
in the partition function 
\begin{equation}
	\label{eq:partition}
	Z_N^\alpha=\int_{\R^N}\left(\prod_{1\leq k < \ell \leq N}|\lambda_k-\lambda_\ell|^\beta\right)e^{-\frac{\beta N}{2}\sum_{k=1}^N V_\alpha(\lambda_k)}\diff\lambda_1\ldots \diff\lambda_N.
\end{equation}
By Corollary \ref{cor:psi_z_tilde}, there is some $C>0$ such that for $z\in K\setminus \R$ and $\lambda\in \R$, $|\widetilde\psi_{\alpha,z}'(\lambda)|\leq C$, so this indeed defines a diffeomorphic change of variables for $N$ large enough (for $N$ small, the result is a consequence of Corollary~\ref{cor:a_priori_f_z}). We find
\begin{multline}\label{eq:rewriting_Z}
Z_N^\alpha=\int_{\R^N}\left(\prod_{1\leq k <\ell \leq N}\left|y_k-y_\ell + \frac{1}{N}(\widetilde\psi_{\alpha,z}(y_k)-\widetilde\psi_{\alpha,z}(y_\ell))\right|\right)^\beta e^{-\frac{\beta N}{2}\sum_{k=1}^N V_\alpha(y_k+\widetilde\psi_{\alpha,z}(y_k)/N)}\times\\\times\prod_{k=1}^N \left(1+\widetilde\psi_{\alpha,z}'(y_k)/N\right)\diff y_k.
\end{multline}
Factoring out $\prod_{k <\ell}|y_k-y_\ell|^\beta$ in the product term, and Taylor expanding the different terms involved, we will now convert this rewriting into an expectation with respect to $\E_\alpha$. By Taylor's theorem, there exist $\theta_{k,\ell}\in [0,1]$ such that
\begin{multline*}
	\sum_{1\leq k < \ell \leq N}\log\left|1 +\frac{1}{N}\frac{\widetilde\psi_{\alpha,z}(y_k)-\widetilde\psi_{\alpha,z}(y_\ell)}{y_k-y_\ell}\right| =  \frac{1}{N}\sum_{1\leq k < \ell \leq N}\frac{\widetilde\psi_{\alpha,z}(y_k)-\widetilde\psi_{\alpha,z}(y_\ell)}{y_k-y_\ell} 
	\\- \frac{1}{N^2}\sum_{1\leq k < \ell \leq N} \theta_{k,\ell}\left( \frac{\widetilde\psi_{\alpha,z}(y_k)-\widetilde\psi_{\alpha,z}(y_\ell)}{y_k-y_\ell}\right)^2 
\end{multline*}
and $\theta_k\in[0,1]$ such that
\[
V_\alpha\left(y_k + \frac{1}{N}\widetilde\psi_{\alpha,z}(y_k)\right)=V_\alpha(y_k)+\frac{1}{N}\widetilde\psi_{\alpha,z}(y_k)V_\alpha'(y_k) + \frac{1}{N^2}\widetilde\psi_{\alpha,z}(y_k)^2V_\alpha''\left(y_k+ \frac{\theta_k}{N}\widetilde\psi_{\alpha,z}(y_k)\right).
\]
Using this and dividing both sides of \eqref{eq:rewriting_Z} by $Z_N^\alpha$, we obtain:
\begin{equation}
	\label{eq:egalite_partitions}
	1=\E_\alpha\left[\exp\left( \frac{\beta}{N}\sum_{1\leq k < \ell \leq N}\frac{\widetilde\psi_{\alpha,z}(\lambda_k)-\widetilde\psi_{\alpha,z}(\lambda_\ell)}{\lambda_k-\lambda_\ell} - \frac{\beta}{2}\sum_{k=1}^N \widetilde\psi_{\alpha,z}(\lambda_k)V_\alpha'(\lambda_k)\right)\times e^{\mathrm{Rem}_N}\right],
\end{equation}
where the remainder term $\mathrm{Rem}_N$ can be bounded as follows
\begin{equation} \label{eq:bound_remainder}
    \abs{\mathrm{Rem}_N}
    \leq \frac{\beta}{2}\left( \sup_{|h|\leq \|\widetilde\psi_{\alpha,z}\|_{\infty}/N} \|\widetilde\psi_{\alpha,z}^2V_\alpha''(\cdot + h)\|_{\infty}  + \|\widetilde\psi_{\alpha,z}'\|_\infty^2\right)+\|\widetilde\psi_{\alpha,z}'\|_\infty
    \leq C,
\end{equation} 
for some finite constant $C$ by Corollary \ref{cor:psi_z_tilde}, which in particular implies that, for any $x \in \R$ and $\abs{h} \leq 1$ (we use here $\|\widetilde\psi_{\alpha,z}\|_{\infty}/N \leq 1$ for $N$ large enough), 
\[
    \abs{\widetilde\psi_{\alpha,z}(x)^2V_\alpha''(x+h)}
    \leq \frac{C \abs{V_\alpha''(x+h)}}{(1+\abs{V_\alpha'(x)})^2}
    \leq \frac{C (\alpha \abs{x+h}^{p-2} + (1-\alpha))}{(1+\alpha \abs{x}^{p-1} + (1-\alpha)\abs{x})^2}
    \leq C.
\]
Now, recalling the definition of $\Xi_\alpha$ in \eqref{eq:def_master_op} and of the anisotropy in \eqref{eq:def_anisotropy}, and using \eqref{eq:identiteintegralmuvalpha}, notice that
\begin{multline*}
	\frac{1}{N}\sum_{1\leq k < \ell \leq N}\frac{\widetilde\psi_{\alpha,z}(\lambda_k)-\widetilde\psi_{\alpha,z}(\lambda_\ell)}{\lambda_k-\lambda_\ell} - \frac{1}{2}\sum_{k=1}^N \widetilde\psi_{\alpha,z}(\lambda_k)V_\alpha'(\lambda_k) \\= L_N(\Xi_\alpha[\widetilde\psi_{\alpha,z}])+\frac{1}{2N}\left(A_N(\widetilde\psi_{\alpha,z}) - \sum_{k=1}^N\widetilde\psi_{\alpha,z}'(\lambda_k) \right).
\end{multline*}
Moreover, recall that $\Xi_\alpha[\widetilde\psi_{\alpha,z}] = f_{\alpha,z}\phi$. Therefore, coming back to \eqref{eq:egalite_partitions}, using \eqref{eq:bound_remainder} and bounding $\abso{\frac{1}{N}\sum_{k=1}^N\widetilde\psi_{\alpha,z}'(\lambda_k)} \leq \|\widetilde\psi_{\alpha,z}'\|_\infty \leq C$,
we get 
\[
    \E_\alpha\left[ \exp\left(L_N(f_{\alpha,z}\phi) 
    +\frac{1}{2N}A_N(\widetilde\psi_{\alpha,z})  \right) \right]
    = \E_\alpha\left[ \exp\left(L_N(\Xi_\alpha[ \widetilde\psi_{\alpha,z}]) 
    +\frac{1}{2N}A_N(\widetilde\psi_{\alpha,z})  \right) \right]
    \leq C.
\]
Replacing $\widetilde\psi_{\alpha,z}$ by $-\widetilde\psi_{\alpha,z}$ in the change of variables, we get the same bound with a minus sign in the exponential. therefore, using that $e^{\abs{x}} \leq e^x + e^{-x}$ for $x \in \R$, we find that
\[
\E_\alpha\left[ \exp\left(\abs{L_N(f_{\alpha,z}\phi) +\frac{1}{2N}A_N(\widetilde\psi_{\alpha,z})}\right) \right] \leq C.
\]
Expanding the exponential in series, we deduce that for all $q\geq 1$
\[
\E_\alpha\left[ \frac{1}{q!}\left|L_N(f_{\alpha,z}\phi)-\frac{1}{2N}A_N(\widetilde\psi_{\alpha,z})\right|^q\right] \leq C.
\]
Combining this with \eqref{eq:truncation_f_z}, we get
\begin{equation}
	\E_\alpha\left[ \left|L_N(f_{\alpha,z})\right|^q\right] 
	\leq (Cq)^q + \frac{C^q}{N^q} \E_\alpha\left[\abs{A_N(\widetilde\psi_{\alpha,z})}^q\right].
	\label{ineq:ineq_intermediaire}
\end{equation}
It remains to bound the anisotropy term. We use two successive truncations. Firstly, applying Lemma~\ref{lem:truncation}.\ref{it:truncation3} (recall $\widetilde\psi_{\alpha,z}$ and its derivative are bounded by Corollary~\ref{cor:psi_z_tilde}), we get
\begin{equation}
	\E_\alpha\left[\abs{A_N(\widetilde\psi_{\alpha,z}) - A_N(\phi \widetilde\psi_{\alpha,z})}^q\right]
	\leq (Cq)^{2q}. \label{ineq:ineq_intermediaire2}
\end{equation}
Then, we would like to use the local law in Corollary~\ref{cor:a_priori_local_law} together with Lemma~\ref{lem:HS_A_N}, but this local law is true only in the trapezoid region so we have to reduce again the support so that it is contained in $[-1,1]$ (with some margin for the regularization step). 
So we consider $\chi \in \mathcal{C}^\infty_c(\R)$ such that 
\[
\1_{[-1/4,1/4]} \leq \chi \leq \1_{[-1/2,1/2]}.
\]
Noting that $\phi-\chi = 0$ on $[-1/4,1/4]$, we get, by Corollary \ref{cor:psi_z_tilde}, that $(\phi-\chi)\widetilde\psi_{\alpha,z}$ is in $\mathcal{C}^2_c(\R)$ with bounded derivatives (uniformly in $\alpha$ and $z \in K$), therefore we can apply Lemma~\ref{lem:a_priori}.\ref{it:a_priori_A} to bound 
\begin{equation}
	\E_\alpha\left[\abs{A_N(\phi\widetilde\psi_{\alpha,z}) - A_N(\chi \widetilde\psi_{\alpha,z})}^q\right]
	= \E_\alpha\left[\abs{A_N((\phi-\chi)\widetilde\psi_{\alpha,z})}^q\right]
	\leq \left( Cq N \log N \right)^q. \label{ineq:ineq_intermediaire3}
\end{equation}
Now, we want to apply Lemma~\ref{lem:HS_A_N} to $\chi\widetilde\psi_{\alpha,z}$, but it requires $\mathcal{C}^3$ regularity, so we first use a regularization argument. 
By Corollary \ref{cor:psi_z_tilde} again, with $\gamma = (3-p)_+$ if $p\neq3$ and some arbitrary $\gamma \in(0,1)$ if $p=3$, we have, for any $\lambda\in \R$, 
\begin{equation} \label{eq:borne_chipsi}
\abso{(\chi\widetilde\psi_{\alpha,z})(\lambda)} \leq C, \qquad 
\abso{(\chi\widetilde\psi_{\alpha,z})'(\lambda)} \leq C, \qquad
\abso{(\chi\widetilde\psi_{\alpha,z})''(\lambda)} 
\leq \frac{C}{\abs{z}} (1+\abs{\lambda}^{-\gamma}). 
\end{equation}
Let $\widetilde\psi_{\alpha,z}^{\varepsilon}\in \mathcal{C}_c^\infty(\R)$ be the regularization of $\chi\widetilde\psi_{\alpha,z}$ given by Lemma \ref{lem:regularize} for $\varepsilon \in (0,1/2)$, then $\supp \widetilde\psi_{\alpha,z}^{\varepsilon} \subset \supp \chi + [-\varepsilon,\varepsilon] \subset [-1,1]$,
\begin{equation} \label{eq:borne_psieps_1}
\normeo{\chi\widetilde\psi_{\alpha,z}-\widetilde\psi_{\alpha,z}^{\varepsilon}}_\infty \leq \varepsilon \normeo{(\chi\widetilde\psi_{\alpha,z})'}_\infty \leq C \varepsilon,
\qquad
\normeo{(\chi\widetilde\psi_{\alpha,z})'-(\widetilde\psi_{\alpha,z}^{\varepsilon})'}_\infty \leq \frac{C}{\abs{z}} \varepsilon^{1-\gamma},
\end{equation}
and, for any $\lambda\in \R$, 
\begin{equation} \label{eq:borne_psieps_2}
|(\widetilde\psi_{\alpha,z}^\varepsilon)''(\lambda)|\leq \frac{C}{\abs{z}} (1+\abs{\lambda}^{-\gamma}),
\qquad 
|(\widetilde\psi_{\alpha,z}^\varepsilon)'''(\lambda)|\leq \frac{C}{\varepsilon \abs{z}} (1+\abs{\lambda}^{-\gamma}).
\end{equation}
Then, 
\begin{equation}
	\label{ineq:anisotropie_singularite}
	|A_N(\chi\widetilde\psi_{\alpha,z})|\leq |A_N(\widetilde\psi_{\alpha,z}^\varepsilon)| + |A_N(\chi\widetilde\psi_{\alpha,z}-\widetilde\psi_{\alpha,z}^\varepsilon)| \leq |A_N(\widetilde\psi_{\alpha,z}^\varepsilon)| + CN^2\frac{\varepsilon^{1-\gamma}}{\abs{z}},
\end{equation}
where we used that $|A_N(f)|\leq 4N^2\|f'\|_{\infty}$. Now, with the local law in Corollary \ref{cor:a_priori_local_law} as input and noting that  $\widetilde\psi_{\alpha,z}^\varepsilon$ is supported in $[-1,1]$, we can apply Lemma \ref{lem:HS_A_N} with $K_0=[-1,1]$ and $M=Cq\sqrt{N\log N}$ to deduce that for any $\eta\in (0,(\delta_0\wedge1)/2)$,
\begin{align}
	\E_\alpha[|A_N(\widetilde\psi_{\alpha,z}^\varepsilon)|^q]
	&\leq \left(Cq^2N\log N \left( \|\widetilde\psi_{\alpha,z}^\varepsilon\|_1 +  \|(\widetilde\psi_{\alpha,z}^\varepsilon)'\|_1 + \log(1/\eta) \|(\widetilde\psi_{\alpha,z}^\varepsilon)''\|_1 + \eta \|(\widetilde\psi_{\alpha,z}^\varepsilon)'''\|_1 \right)\right)^q
	\nonumber \\
	&\leq \left( Cq^2 N \log N\left(1+ \frac{\varepsilon^{1-\gamma}}{\abs{z}} + \frac{\log(1/\eta)}{\abs{z}} + \frac{\eta}{\varepsilon \abs{z}}\right) \right)^q, \label{eq:borne_anisotropy_psieps}
\end{align}
where we used \eqref{eq:borne_chipsi}, \eqref{eq:borne_psieps_1} and \eqref{eq:borne_psieps_2} to bound the $L^1$ norms. 
We now choose $\varepsilon,\eta$ so that $N^2\varepsilon^{1-\gamma} \leq N\log N$ and $\eta/\varepsilon \leq 1$. Taking for example $\varepsilon=\eta=((\log N)/N)^{1/(1-\gamma)}$ and combining \eqref{ineq:anisotropie_singularite} and \eqref{eq:borne_anisotropy_psieps}, we get
\begin{equation} \label{ineq:ineq_intermediaire4}
\E_\alpha\left[\abs{A_N(\chi\widetilde\psi_{\alpha,z})}^q\right]
\leq \left(\frac{Cq^2N(\log N)^2}{\abs{z}}\right)^q.
\end{equation}
Combining \eqref{ineq:ineq_intermediaire}, \eqref{ineq:ineq_intermediaire2}, \eqref{ineq:ineq_intermediaire3} and \eqref{ineq:ineq_intermediaire4} concludes the proof.
\end{proof}

\subsection{Proof of the optimal local law}
\label{sec:opt_local_law}

We prove here the following optimal local law under $\P_\alpha$ uniformly in $\alpha$. This implies Theorem \ref{thm:local_law} by taking $\alpha =1$ when $p>2$ and $\alpha=0$ when $p=2$.

\begin{theorem}[Optimal local law] \label{thm:local_law_alpha}
	Let $p>2$. 
	Let $\delta_0$, $\cR$ and $\cT$ be as in Lemma \ref{lem:stability}. 
	Then there exists $C>0$ such that the following holds, for any $N,q \geq 1$ and $\alpha\in[0,1]$.
	\begin{enumerate}
		\item For any $z \in \cT$,
		\[
			\E_\alpha \left[\abs{s_N(z)-s_{V_\alpha}(z)}^q\right] 
			\leq \frac{(Cq)^{q/2}}{(Ny)^q} + \frac{(Cq)^{32q}}{(N \abs{z} \abs{z^2-1}^{1/2})^q}.
		\]
		\item For any $z=\cR \setminus \cT$,
		\[
			\E_\alpha \left[\abs{s_N(z)-s_{V_\alpha}(z)}^q\right] \leq \frac{(Cq)^{64q}}{(Ny)^{q}}.
		\]
	\end{enumerate}
\end{theorem}

\begin{proof}
    The proof consists in applying four times the local law machinery, Proposition~\ref{prop:local_law_machinery}. Between each application, we improve the bound we have on $\E_\alpha\left[\abs{L_N(f_{\alpha,z})}^q\right]$. There are two ways to improve this bound: we improve the dependence either on $z$ or on $N$. Improving the dependence on $z$ is done by applying the current best local law together with Helffer--Sjöstrand formula by applying Lemma~\ref{lem:HS_z^2}. Improving the dependence on $N$, is done via the master operator strategy, similarly to the proof of Lemma~\ref{lem:better_bound}.
    Let $\delta_0$, $\cR$ and $\cT$ be as in Lemma \ref{lem:stability}. 
    
    \textbf{First local law.} Applying Proposition~\ref{prop:local_law_machinery} with the bound on $\E_\alpha\left[\abs{L_N(f_{\alpha,z})}^q\right]$ from Lemma~\ref{lem:better_bound} as input, we get: there exists $C>0$ such that, for any $N \geq 2$, $q \geq 1$ and $\alpha\in[0,1]$,
    \begin{equation*}
	\E_\alpha\left[\abs{s_N(z)-s_{V_\alpha}(z)}^q\right] \leq
	\begin{cases}
	    \dfrac{(Cq)^{q/2}}{(Ny)^q} + \dfrac{(Cq^2 (\log N)^2)^q}{(N \abs{z}^2 \abs{z^2-1}^{1/2})^q}, & \text{if } z \in \cT, \vspace{.1cm} \\
	    \dfrac{\left(Cq^4(\log N)^4\right)^q}{(Ny)^q}, & \text{if } z \in \cR \setminus \cT,
	\end{cases}
	\end{equation*}
    which can be unified into the less precise bound, holding for any $z \in \cR$,
    \begin{equation} \label{eq:firstLL}
	\E_\alpha\left[\abs{s_N(z)-s_{V_\alpha}(z)}^q\right] \leq
	\frac{\left(Cq^4(\log N)^4\right)^q}{N^q}
	\left( \frac{1}{y} + \frac{1}{\abs{z}^2} \right)^q,
	\end{equation}
	where we used that $\abs{z}^{-2} \abs{z^2-1}^{-1/2} \leq C (\abs{z}^{-2} + \abs{z-1}^{-1/2} + \abs{z+1}^{-1/2}) \leq C (\abs{z}^{-2} + y^{-1})$ for any $z \in \cR$.
	
	\textbf{Improvement of the bound on $\E_\alpha\left[\abs{L_N(f_{\alpha,z})}^q\right]$ via Helffer--Sjöstrand.} 
	Applying first the truncation argument in \eqref{eq:truncation_f_z} and then Lemma \ref{lem:HS_z^2} with \eqref{eq:firstLL} as input, we deduce the following bound: for any compact set $K \subset \C$, there exists $C>0$ such that, for any $N \geq 2$, $q \geq 1$, $\alpha\in[0,1]$ and $z \in K$,
    \begin{equation} \label{eq:first_improvement}
        \E_\alpha\left[\abs{L_N(f_{\alpha,z})}^q\right] 
        \leq \left(Cq^4 (\log N)^4 \right)^q. 
    \end{equation}
    
    \textbf{Second local law.} Applying Proposition~\ref{prop:local_law_machinery} with the bound \eqref{eq:first_improvement} as input, we get: there exists $C>0$ such that, for any $N \geq 2$, $q \geq 1$ and $\alpha\in[0,1]$,
    \begin{equation*}
	\E_\alpha\left[\abs{s_N(z)-s_{V_\alpha}(z)}^q\right] \leq
	\begin{cases}
	    \dfrac{(Cq)^{q/2}}{(Ny)^q} + \dfrac{(Cq^4 (\log N)^4)^q}{(N \abs{z} \abs{z^2-1}^{1/2})^q}, & \text{if } z \in \cT, \vspace{.1cm} \\
	    \dfrac{\left(Cq^8(\log N)^8\right)^q}{(Ny)^q}, & \text{if } z \in \cR \setminus \cT,
	\end{cases}
	\end{equation*}
    which can be unified into the less precise bound, holding for any $z \in \cR$,
    \begin{equation} \label{eq:secondLL}
	\E_\alpha\left[\abs{s_N(z)-s_{V_\alpha}(z)}^q\right] \leq
	\frac{\left(Cq^8(\log N)^8\right)^q}{(Ny)^q}.
	\end{equation}
	
	\textbf{Improvement of the bound on $\E_\alpha\left[\abs{L_N(f_{\alpha,z})}^q\right]$ via master operator.} We follow the argument used in the proof of Lemma~\ref{lem:better_bound}.
	By \eqref{ineq:ineq_intermediaire} and \eqref{ineq:ineq_intermediaire2}, 
    \begin{equation} \label{eq:2nd_improvement_fz_1}
    	\E_\alpha\left[ \abs{L_N(f_{\alpha,z})}^q\right] 
    	\leq (Cq)^{2q} + \frac{C^q}{N^q} \E_\alpha\left[\abs{A_N(\phi \widetilde\psi_{\alpha,z})}^q\right],
    \end{equation}
    where we choose here $\phi\in \mathcal{C}^\infty_c(\R)$ that satisfies
    \[ 
    \1_{[-1-\delta_0/2,1+\delta_0/2]}
    \leq\phi\leq \1_{[-1-3\delta_0/4,1+3\delta_0/4]},
    \]
    and set again $\widetilde\psi_{\alpha,z} \coloneqq \Xi_{\alpha}^{-1}[f_{\alpha,z}\phi]$. For $\varepsilon\in(0,\delta_0/2)$, let $\Psi^\varepsilon_{\alpha,z}$ be the regularization of $\phi\widetilde\psi_{\alpha,z}$ given by Lemma~\ref{lem:regularize}. 
    Proceeding as for the proof of \eqref{ineq:ineq_intermediaire4}, but using the local law \eqref{eq:secondLL} as input to apply Lemma~\ref{lem:HS_A_N} with $K_0 = [-1-\delta_0,1+\delta_0]$ and $M = Cq^8 (\log N)^8$ (note that $\supp \Psi^\varepsilon_{\alpha,z} \subset K_0$ by our choice of $\phi$ and $\varepsilon$), we get the bound
    \begin{align*}
    	\E_\alpha\left[\abs{A_N(\phi\widetilde\psi_{\alpha,z})}^q\right]
    	&\leq 2^{q} \left(\E_\alpha\left[\abs{A_N(\phi\widetilde\psi_{\alpha,z}-\Psi^\varepsilon_{\alpha,z})}^q\right]
    	+ \E_\alpha\left[\abs{A_N(\Psi^\varepsilon_{\alpha,z})}^q\right]\right)\\
    	&\leq \left(\frac{C N^2\varepsilon^{1-\gamma}}{\abs{z}}\right)^q + \left(Cq^{16}(\log N)^{16}\left(1+\frac{\varepsilon^{1-\gamma}}{\abs{z}}+\frac{\log(1/\eta)}{\abs{z}}+\frac{\eta}{\varepsilon \abs{z}}\right) \right)^q.
    \end{align*}
    We now take $\varepsilon=\eta=N^{-2/(1-\gamma)}$, so that $N^2\varepsilon^{1-\gamma} = 1$, and this yields
    \begin{equation} \label{eq:second_bound_anisotropy}
    \E_\alpha\left[\abs{A_N(\phi\widetilde\psi_{\alpha,z})}^q\right]
    \leq \left(\frac{Cq^{16}(\log N)^{17}}{\abs{z}}\right)^q.
    \end{equation}
    Coming back to \eqref{eq:2nd_improvement_fz_1}, we obtain
    \begin{equation} \label{eq:second_improvement}
        \E_\alpha\left[ L_N(f_{\alpha,z})^q \right] \leq \left( \frac{Cq^{16}}{\abs{z}} \right)^q.
    \end{equation}
    
    \textbf{Third local law.} Applying Proposition~\ref{prop:local_law_machinery} with the bound \eqref{eq:second_improvement} as input and unifying the cases as for \eqref{eq:firstLL}, we get: there exists $C>0$ such that, for any $N \geq 2$, $q \geq 1$, $\alpha\in[0,1]$ and $z \in \cR$,
    \begin{equation} \label{eq:thirdLL}
	\E_\alpha\left[\abs{s_N(z)-s_{V_\alpha}(z)}^q\right] \leq
	\frac{\left(Cq^{32}\right)^q}{N^q}
	\left( \frac{1}{y} + \frac{1}{\abs{z}^2} \right)^q.
	\end{equation}
	
	\textbf{Improvement of the bound on $\E_\alpha\left[\abs{L_N(f_{\alpha,z})}^q\right]$ via Helffer--Sjöstrand.} 
	Applying first the truncation argument in \eqref{eq:truncation_f_z} and then Lemma \ref{lem:HS_z^2} with \eqref{eq:thirdLL} as input, we deduce the following bound: for any compact set $K \subset \C$, there exists $C>0$ such that, for any $N \geq 2$, $q \geq 1$, $\alpha\in[0,1]$ and $z \in K$,
    \begin{equation} \label{eq:third_improvement}
        \E_\alpha\left[\abs{L_N(f_{\alpha,z})}^q\right] 
        \leq \left(Cq^{32} \right)^q. 
    \end{equation}
    
    \textbf{Fourth local law.} Applying Proposition~\ref{prop:local_law_machinery} with the bound \eqref{eq:third_improvement} as input, we get the desired result.
\end{proof}

\section{Central limit theorem}
\label{sec:CLT}
We are now able to give a proof of Theorem \ref{thm:CLT}, but before that we state a preliminary lemma to rewrite the limiting mean and variance appearing in the proof.

\begin{lemma}[Mean/variance formulas]\label{lem:variance formula}
    Let $p>2$.
    Let $f \in \mathcal{C}^3(\R)$, and $\psi_\alpha=\Xi_\alpha^{-1}[f]$. Then, 
    \begin{multline}\label{eq:lemme_moyenne}
        \int_{-1}^1 \psi_\alpha'(x)\diff\mu_{V_{\alpha}}(x) 
        \\= \frac{1}{\pi^2}\int_{-1}^1 \frac{r_\alpha'(x)}{r_\alpha(x)}\int_{-1}^1\frac{f(t)-f(x)}{t-x}\frac{\diff t}{\sigma(t)}\sigma(x)\diff x 
        - \frac{f(1)+f(-1)}{2}+\frac{1}{\pi}\int_{-1}^1f(x)\frac{\diff x}{\sigma(x)}
    \end{multline}
        and
    \begin{multline}\label{eq:lemme_variance}
	\dfrac{1}{2}\int_{-1}^1V''_\alpha(x)\psi_\alpha(x)^2\diff\mu_{V_\alpha}(x)
	+\dfrac{1}{2}\iint_{[-1,1]^2}\left(\dfrac{\psi_\alpha(x)-\psi_\alpha(y)}{x-y}\right)^2\diff\mu_{V_\alpha}(x)\diff\mu_{V_\alpha}(y) \\
	= \dfrac{1}{\pi^2}\iint_{[-1,1]^2}\left(\dfrac{f(x)-f(y)}{x-y}\right)^2\dfrac{1-xy}{\sigma(x)\sigma(y)}\diff x\diff y.
    \end{multline}
\end{lemma}

\begin{proof}
    First recall that, by Lemma~\ref{lem:inverse_master_op}, for any $x \in [-1,1]$,
    \[
		\psi_\alpha(x) = 
		-\frac{1}{\pi r_\alpha(x)} 
		\int_{-1}^1 \frac{f(t)-f(x)}{t-x}
		\frac{\diff t}{\sigma(t)},
	\]
    and $\psi_\alpha$ is $\mathcal{C}^1$ so that all terms in the claimed equalities make sense. Moreover, it follows from Lemma~\ref{lem:inverse_master_op}.\ref{it:inverse_C3} that $\psi_\alpha'$ is Hölder continuous, which is used in the proof of \eqref{eq:lemme_variance}.
    
    \textbf{Proof of \eqref{eq:lemme_moyenne}.} 
    Recalling from \eqref{eq:mu=rsigma} that $\diff\mu_{V_\alpha}(x) = \frac{1}{\pi} \sigma(x)r_\alpha(x) \1_{[-1,1]}(x) \diff x$ and differentiating the expression of $\psi_\alpha$ given above, we find
    \begin{multline*}
        \int_{-1}^1 \psi_\alpha'(x)\diff\mu_{V_{\alpha}}(x)= \frac{1}{\pi^2}\int_{-1}^1 \frac{r_\alpha'(x)}{r_\alpha(x)}\int_{-1}^1\frac{f(t)-f(x)}{t-x}\frac{\diff t}{\sigma(t)}\sigma(x)\diff x \\-\frac{1}{\pi^2}\int_{-1}^1 \partial_x\left( \int_{-1}^1\frac{f(t)-f(x)}{t-x}\frac{\diff t}{\sigma(t)}\right) \sigma(x)\diff x.
    \end{multline*}
    Integrating by parts and using $\sigma'(x)=-x/\sigma(x)$, the second term rewrites
    \begin{align*}
        -\frac{1}{\pi^2}\int_{-1}^1 \partial_x\left( \int_{-1}^1\frac{f(t)-f(x)}{t-x}\frac{\diff t}{\sigma(t)}\right) \sigma(x)\diff x
        & = -\frac{1}{\pi^2}\iint_{[-1,1]^2}\frac{f(t)-f(x)}{t-x}\frac{\diff t}{\sigma (t)} \frac{x}{\sigma(x)}\diff x \\
        & = -\frac{f(1)+f(-1)}{2}+\frac{1}{\pi}\int_{-1}^1 f(x)\frac{\diff x}{\sigma(x)},
    \end{align*}
    by \cite[Lemma A.5]{LamLedWeb2019}, concluding the proof of \eqref{eq:lemme_moyenne}.

    \textbf{Proof of \eqref{eq:lemme_variance}.} 
    By \cite[Lemma 4.2]{BekLebSer2018} (note that they define $\psi_\alpha$ with an additional factor $\frac{1}{2}$ compared to us), we have 
    \begin{align}
    	& \dfrac{1}{2}\int_{-1}^1V''_\alpha(x)\psi_\alpha(x)^2\diff\mu_{V_\alpha}(x)+\dfrac{1}{2}\iint_{[-1,1]^2}\left(\dfrac{\psi_\alpha(x)-\psi_\alpha(y)}{x-y}\right)^2\diff\mu_{V_\alpha}(x)\diff\mu_{V_\alpha}(y) 
    	\nonumber \\
    	& = -\int_{-1}^1\psi_\alpha(x)f'(x)\diff\mu_{V_\alpha}(x).
    	\label{eq:formulevariance1}
    \end{align}
    This fact uses that $\psi_\alpha'$ is Hölder continuous (this ensures that the integrals in the intermediate steps of the proof of \cite[Lemma 4.2]{BekLebSer2018} are finite).    
    On the other hand, it can be noticed that the integral on the RHS of \eqref{eq:formulevariance1} can be recast as a weighted $\mathsf{H}^{1/2}$-norm as follows. Using the formula for $\psi_\alpha(x)$ recalled above, that $\diff\mu_{V_\alpha}(x) = \frac{1}{\pi} \sigma(x)r_\alpha(x) \1_{[-1,1]}(x) \diff x$, and Fubini's theorem, we get
    \begin{equation*}
    	\int_{-1}^1 \psi_\alpha(x) f'(x)\diff\mu_{V_\alpha}(x)
    	=\dfrac{1}{\pi^2}\int_{-1}^1\dfrac{\diff y}{\sigma(y)}\int_{-1}^1\dfrac{\sigma(x)f'(x)\left(f(x)-f(y)\right)}{y-x}\diff x.
    \end{equation*}
    Now using that:
    \[
    \partial_x\left[(f(x)-f(y))^2\right]=2f'(x)(f(x)-f(y))\hspace{1cm}\mrm{and}\hspace{1cm}\partial_x\left[\dfrac{\sigma(x)}{y-x}\right]=\dfrac{1-xy}{\sigma(x)(x-y)^2},
    \]
    we find by integration by parts:
    \[
    \dfrac{1}{\pi^2}\int_{-1}^1\dfrac{\diff y}{\sigma(y)}\int_{-1}^1\dfrac{\sigma(x)f'(x)\left(f(x)-f(y)\right)}{y-x}\diff x
    =-\dfrac{1}{\pi^2}\iint_{[-1,1]^2}\left(\dfrac{f(x)-f(y)}{x-y}\right)^2\dfrac{1-xy}{\sigma(x)\sigma(y)}\diff x\diff y,
    \]
    from which we conclude.
\end{proof}

We now state and prove a CLT under $\P_\alpha$, which implies Theorem \ref{thm:CLT} by taking $\alpha = 1$ for $p>2$ and $\alpha=0$ for $p=2$.

\begin{theorem}[CLT]\label{thm:CLTalpha}
Let $p>2$. For any $f\in\mc{C}^3(\R)$ with at most exponential growth we have for all $k\geq1$, uniformly in $\alpha\in[0,1]$,
\[
\E_\alpha\left[L_N(f)^k\right]
\xrightarrow[N\rightarrow\infty]{}
M_k^\alpha(f),
\]
where $M_k^\alpha(f)$ is the $k$-th moment of a normal distribution of mean and variance given by
\begin{align}\label{eq:mpalpha}
	m_{V_\alpha}(f)&=\left(\dfrac{1}{2}-\dfrac{1}{\beta}\right)\int_{-1}^{1}\Xi_\alpha^{-1}[f]'(x)\diff\mu_{V_\alpha}(x)
	\\\sigma^2(f)&=\dfrac{1}{\beta\pi^2}\iint_{[-1,1]^2}\left(\dfrac{f(x)-f(y)}{x-y}\right)^2 \cdot \dfrac{1-xy}{\sqrt{1-x^2}\sqrt{1-y^2}}\diff x\diff y\label{eq:sigmapalpha}.
\end{align}
\end{theorem}

\begin{proof} Let $f$ be $\mathcal{C}^3(\R)$.
By Lemma~\ref{lem:truncation}.\ref{it:truncation2}, we know that the moments of $L_N(f)$ for $f\in\mc{C}^3(\R)$ with at most exponential growth are approximated by those of $L_N(f\phi)$ where $\phi\in\mc{C}_c^\infty(\R)$, so it is now enough to work with $f$ with compact support.
This ensures that $\norme{f}_{\cC^3} < \infty$.
We set $\psi_\alpha\coloneqq \Xi_\alpha^{-1}[f]$. 
Recall from Lemma~\ref{lem:inverse_master_op} that $\psi_\alpha\in\cC^1(\R) \cap \cC^2(\R^*)$ and satisfies $\norme{\psi_\alpha}_{\cC^1} \leq C$ and, for $\lambda\in \R$, $\abs{\psi_\alpha''(\lambda)} \leq C(1+\abs{\lambda}^{-\gamma})$ for some $\gamma \in [0,1)$ (one can choose $\gamma = (3-p)_+$ if $p\neq3$ and some arbitrary $\gamma \in(0,1)$ if $p=3$).
Note that nothing else on $\psi_\alpha$ and $f$ are used afterwards

The proof is based on the change of variables method in the partition function, as in the proof of Lemma~\ref{lem:better_bound}, but in a more precise version.
Making the change of variables $\lambda_i=y_i + \psi_\alpha(y_i)/N$ for all $i\in\llbracket1,N\rrbracket$ in the partition function, as in \eqref{eq:rewriting_Z} and then rewriting the integral as an expectation, we obtain: 
\begin{multline*}
	\mathbb{E}_\alpha\Biggl[\exp\Biggl(\dfrac{\beta}{2}\sum_{i\neq j}\log\abs{1+\dfrac{1}{N}\dfrac{\psi_\alpha(\lambda_i)-\psi_\alpha(\lambda_j)}{\lambda_i-\lambda_j}}
	\\-\dfrac{\beta N}{2}\sum_{i=1}^N\left[V_\alpha\left(\lambda_i+\dfrac{\psi_\alpha(\lambda_i)}{N}\right)-V_\alpha(\lambda_i)\right]
	+\sum_{i=1}^N \log\left(1+\dfrac{\psi_\alpha'(\lambda_i)}{N}\right) \Biggr)\Biggr]=1.
\end{multline*}
We then proceed to a Taylor expansion as in Lemma \ref{lem:better_bound} but one order further, using in particular that, for any $\lambda \in \R$, by decomposing $V_\alpha = \alpha V + (1-\alpha)V_0$, using \eqref{eq:g_Taylor_order_3} for the part involving $V$ and noting that the part involving $V_0$ is zero,
\begin{align*}
    \abs{V_\alpha\left(\lambda+\frac{\psi_\alpha(\lambda)}{N}\right) 
    - V_\alpha(\lambda) 
    - \frac{\psi_\alpha(\lambda)}{N} V_\alpha'(\lambda) 
    - \frac{\psi_\alpha(\lambda)^2}{2N^2} V_\alpha''(\lambda)} 
    & \leq C \alpha \frac{\abs{\psi_\alpha(\lambda)}^3}{N^3} \left( \abs{\lambda+\frac{\psi_\alpha(\lambda)}{N}} \vee \abs{\lambda}\right)^{p-3} \\
    & \leq C \frac{\abs{\psi_\alpha(\lambda)}^p}{N^p}
    + \1_{p>3} C \alpha \frac{\abs{\psi_\alpha(\lambda)}^3}{N^3} \abs{\lambda}^{p-3},
\end{align*}
using that, for any $x,y \in \R$, $\abs{x} \vee \abs{y} \geq \abs{x-y}/2$ when $p \leq 3$, and $\abs{x} \vee \abs{y} \leq \abs{x-y}+\abs{y}$ if $p>3$.
This yields that
\[
    \E_\alpha\left[ \exp\left( \frac{\beta}{2N}\sum_{i\neq j}\frac{\psi_\alpha(\lambda_i)-\psi_\alpha(\lambda_j)}{\lambda_i-\lambda_j}-\frac{\beta}{2}\sum_{i=1}^N\psi_\alpha(\lambda_i)V_\alpha'(\lambda_i)+I_N+\mathrm{Rem}^{(2)}_N\right)\right]=1,
\]
where 
\begin{align*}
   I_N & \coloneqq -\frac{\beta}{4N^2}\sum_{i\neq j }\left(\frac{\psi_\alpha(\lambda_i)-\psi_\alpha(\lambda_j)}{\lambda_i-\lambda_j}\right)^2 + \frac{1}{N}\sum_{i=1}^N \psi_\alpha'(\lambda_i) - \frac{\beta}{4N}\sum_{i=1}^N \psi_\alpha(\lambda_i)^2V_\alpha''(\lambda_i), \\
   \abs{\mathrm{Rem}^{(2)}_N}
   & \leq \frac{C}{N} \left( \norme{\psi_\alpha'}_\infty^3+\norme{\psi_\alpha'}_\infty^2 \right)
   + \frac{C}{N^{p-2}} \norme{\psi_\alpha}_\infty^p
   + \1_{p>3} \frac{C}{N} \norme{\psi_\alpha}_\infty^2 \sup_{\lambda \in \R} \alpha \abs{\psi_\alpha(\lambda)} \abs{\lambda}^{p-3}
   \xrightarrow[N\to\infty]{} 0,
\end{align*}
using that $\norme{\psi_\alpha}_{\cC^1} < \infty$ and, for the last term, that $\abs{\psi_\alpha(\lambda)} \leq C/(\alpha\abs{\lambda}^{p-1})$ for $\abs{\lambda}\geq 2$ by definition of $\psi_\alpha$ in \eqref{eq:def_psi} together with the bound \eqref{eq:LB_G} on the denominator.
Next, we use \eqref{eq:identiteintegralmuvalpha}, with $h=\psi_\alpha$ to deduce that
\[
\frac{\beta}{2N}\sum_{i\neq j}\frac{\psi_\alpha(\lambda_i)-\psi_\alpha(\lambda_j)}{\lambda_i-\lambda_j}-\frac{\beta}{2}\sum_{i=1}^N\psi_\alpha(\lambda_i)V_\alpha'(\lambda_i) = \beta L_N(\Xi_\alpha[\psi_\alpha]) -\frac{\beta}{2N}\sum_{i=1}^N \psi_\alpha'(\lambda_i) + \frac{\beta}{2N}A_N(\psi_\alpha),
\]
to get that
\[
\E_\alpha\left[\exp\left( \beta L_N(f) +I_N' + \frac{\beta}{2N}A_N(\psi_\alpha) \right)\right] \xrightarrow[N\to\infty]{} 1,
\]
where we set
\begin{align*}
    I_N' & \coloneqq I_N - \frac{\beta}{2N} \sum_{i=1}^N \psi_\alpha'(\lambda_i) \\
    & =-\frac{\beta}{4N^2}\sum_{i\neq j }\left(\frac{\psi_\alpha(\lambda_i)-\psi_\alpha(\lambda_j)}{\lambda_i-\lambda_j}\right)^2 + \frac{1}{N} \left(1-\frac{\beta}{2} \right) \sum_{i=1}^N \psi_\alpha'(\lambda_i) - \frac{\beta}{4N}\sum_{i=1}^N \psi_\alpha(\lambda_i)^2 V_\alpha''(\lambda_i).
\end{align*}
We now use the large deviations for the empirical measure $\mu_N$ to justify that we can replace the term $I_N'$ by its limit.
This argument is classical: by \cite[Theorem 2.6.1]{AndGuiZei10}, and for a metric $d$ compatible with the weak convergence of probability measures, for any $\delta>0$, there exists $c>0$ such that, for $N$ large enough and for any $\alpha \in [0,1]$,
\[
\P_\alpha(d(\mu_N,\mu_{V_\alpha})\geq\delta) \leq e^{-c N^2}.
\]
The uniformity in $\alpha$ is not claimed in \cite[Theorem 2.6.1]{AndGuiZei10} but can be obtained by adapting their argument.
Since all the terms involved in $I_N'$ are continuous functions with respect to $d$ of the empirical measure $\mu_N$ (as integrals of bounded continuous functions tested against $\mu_N$), and by the crude bound
\[
\beta L_N(f) + \frac{\beta}{2N}A_N(\psi_\alpha) = O(N),
\]
decomposing on the events $\{d(\mu_N,\mu_{V_\alpha})<\delta\}$ and $\{d(\mu_N,\mu_{V_\alpha})\geq\delta\}$, then letting $N\rightarrow\infty$ and then $\delta\rightarrow0$, we conclude that
\[
\lim_{N\rightarrow\infty} \E_\alpha\left[ \exp\left( \beta L_N(f) + \frac{\beta}{2N}A_N(\psi_\alpha) \right) \right] 
= \exp\left( \beta m_{V_\alpha}(f) + \dfrac{\beta^2}{2}\sigma^2(f) \right),
\]
where we used Lemma \ref{lem:variance formula} to rewrite the limit of $I_N'$ in terms of $m_{V_\alpha}(f)$ and $\sigma^2(f)$.
Since this convergence holds for $tf$ instead of $f$ for any fixed $t \in \R$, this proves the convergence of the Laplace transform of $L_N(f) + \frac{1}{2N}A_N(\psi_\alpha)$ to the Laplace transform of $\cN(m_{V_\alpha}(f),\sigma^2(f))$. 
Since the convergence of Laplace transform implies the convergence of moments, it only remains to show that the moments of $\frac{1}{N}A_N(\psi_\alpha)$ converge to zero, \textit{i.e.\@} that for all $q \geq 1$, 
\begin{equation} \label{eq:anisotropy_vanishes}
\E_\alpha\left[\left|\frac{1}{N}A_N(\psi_\alpha)\right|^q\right] \xrightarrow[N\to\infty]{} 0.
\end{equation}
This was already proved for $\widetilde\psi_{\alpha,z}$ instead of $\psi_\alpha$, by first truncating outside of $[-1-\delta_0,1+\delta_0]$ in \eqref{ineq:ineq_intermediaire2}, and then proving the result for the truncated version of $\widetilde\psi_{\alpha,z}$ in \eqref{eq:second_bound_anisotropy}.
Recalling that $\psi_\alpha$ and $\psi_\alpha'$ are bounded and $\abs{\psi_\alpha''(\lambda)} \leq C (1+\abs{\lambda}^{-\gamma})$ with $\gamma \in [0,1)$, the same argument works to prove \eqref{eq:anisotropy_vanishes}. 
This concludes the proof.
\end{proof}

We proved the central limit theorem in moments for a general function $f\in\mc{C}^3(\R)$ with at most exponential growth but we can also prove it for the test function $\partial_\alpha V_\alpha$. This will allow us in the next section to deduce the subleading correction of the free energy.  As we only need the convergence of the first moment, we don't state the full central limit theorem.

\begin{theorem}\label{thm:ConvMean}
Let $p>2$.
The following convergence holds uniformly in $\alpha\in[0,1]$:
\[
\E_\alpha\left[L_N(\partial_\alpha V_\alpha)\right] 
\xrightarrow[N\rightarrow\infty]{} m_{V_\alpha}(\partial_\alpha V_\alpha)
\]
\end{theorem}

\begin{proof}
As in the proof of Theorem~\ref{thm:CLTalpha}, we first truncate the function $\partial_\alpha V_\alpha$ to replace it by the function $f= \phi \partial_\alpha V_\alpha \in \cC^2(\R) \cap \cC^3(\R^*)$ with compact support.
Let $\psi_\alpha = \Xi_\alpha^{-1}[f]$. As mentioned at the beginning of the proof of Theorem~\ref{thm:CLTalpha}, the CLT in terms of moment holds as soon as we have $\psi_\alpha\in\cC^1(\R) \cap \cC^2(\R^*)$ with $\norme{\psi_\alpha}_{\cC^1} \leq C$ and, for any $\lambda\in \R$, $\abs{\psi_\alpha''(\lambda)} \leq C(\abs{\lambda}^{-\gamma}+1)$ for some $\gamma \in (0,1]$.
This can be showed by following the proof of Lemma~\ref{lem:psi_z}, noting that $f$ satisfies better bounds than $f_{\alpha,z}$, that is $f,f',f''$ are bounded and  $|f'''(\lambda)| \leq C \abs{\lambda}^{-(3-p)_+}$ for any $\lambda \in \R$.
\end{proof}
\section{Subleading order of the partition function}
\label{sec:logZN}
Having obtained the CLT for any $\alpha$ in Theorem~\ref{thm:CLTalpha}, we can follow the classical interpolation method, dating back at least to \cite{BorGui2013}, to compute the next order in the asymptotic expansion of $\log Z_N$ stated in Corollary~\ref{thm:subleadingZN}. We call $Z_N^G$ the following partition function:

\begin{equation}\label{eq:ZN_gaussien}
Z_N^G\coloneqq\int_{\R^N}\prod_{i<j}|\lambda_i-\lambda_j|^{\beta} \cdot \prod_{i=1}^{N}e^{-N\frac{\beta}{2}V_G(\lambda_i)}\diff \lambda_i
\end{equation}
with $V_G(x)\coloneqq 2x^2$. Under the associated probability distribution $\mathbb{P}_N^{G}$, the empirical distribution $\mu_N=N^{-1}\sum_{i=1}^{N}\delta_{\lambda_i}$ converges towards the following semi-circle distribution
$$\diff\mu_G(x)=\dfrac{2}{\pi}\sqrt{1-x^2}\1_{\abs{x}\leq1}\diff x.$$
Thanks to Mehta's formula \cite[17.6.7]{mehta2004random}, the partition function $Z_N^G$ is known explicitly:
\begin{equation}
    Z_N^G
    =\dfrac{N!(2\pi)^{N}\left(\frac{\beta}{2}\right)^{(\frac{\beta}{2}-1)N}}{\Gamma\left(\frac{\beta}{2}\right)^{N}\Gamma\left(N+1,\frac{2}{\beta},1\right) } \cdot (4N)^{-\frac{\beta}{4}N^{2}-(1-\frac{\beta}{2})\frac{N}{2}},
\end{equation}
where $\Gamma(\cdot,\cdot,\cdot)$ stands for the Barnes double Gamma function, see \cite[Section 7.2.2]{borot2024asymptotic}.
From the previous formula, we can explicitly deduce the all-order asymptotic expansion of $Z_N^G$ that we only write till $o(N)$.
\begin{lemma}[Asymptotic expansion of $\log Z_N^{G}$]\label{lem:logZNVG}
As $N \to \infty$,
\begin{equation}
	\label{eq: asympt_ZNG}
	\log Z_N^{G}=-\left[\dfrac{\beta}{2}\left(\dfrac{3}{4}+\log 2\right) \right]N^{2}+\dfrac{\beta}{2}N\log N+F^{\{-1\}}_G\beta N+o(N),
\end{equation}
where \begin{equation}\label{eq:FG-1}
	F^{\{-1\}}_G\coloneqq\left(\dfrac{1}{2}-\dfrac{1}{\beta}\right)\left(\log\dfrac{\beta}{2}+\log 2\right)-\dfrac{1}{2\beta}-\dfrac{1}{4} +\dfrac{1}{\beta}\log\dfrac{2\pi}{\Gamma\left(\beta/2\right)}.
\end{equation}
\end{lemma}

For the next lemma, we recall that $V_\alpha=\alpha V+(1-\alpha)V_G$.

\begin{lemma}[Interpolation formula]\label{lem:interpolZN}
For any $N \geq 1$, we have
$$\log Z_N=\log Z_N^{G}-\dfrac{N^{2}\beta}{2}\int_0^1\E_{\alpha}\left[\int_\R\partial_\alpha V_\alpha(x)\diff\mu_N(x)\right]\diff\alpha.$$
\end{lemma}
\begin{proof}
Denoting $Z_N^{\alpha}$ the partition function associated with the potential $V_\alpha$, we have	
$$\log\dfrac{Z_N}{Z_N^{G}}=\int_{0}^{1}\partial_\alpha\log Z_N^{\alpha}\diff\alpha=-\dfrac{N^{2}\beta}{2}\int_0^1\E_{\alpha}\left[\int_\R\partial_\alpha V_\alpha(x)\diff\mu_N(x)\right]\diff\alpha,$$
by inserting $\partial_\alpha$ inside the integral in the definition of $Z_N^\alpha$.
\end{proof}

To recast a term in the asymptotic expansion of $\log Z_N$ as the entropy of $\mu_V$, \textit{i.e.\@} $\mu_{V_\alpha}$ for $\alpha=1$, we need the following symmetry property verified by $\Xi_\alpha^{-1}$.

\begin{lemma}\label{lem:symmetryXi}
For any $f,g\in\mc{C}^2(\R)$,
$$\int_{-1}^1f'(x)\Xi_\alpha^{-1}[g](x)\diff\mu_{V_\alpha}(x)=\int_{-1}^1g'(x)\Xi_\alpha^{-1}[f](x)\diff\mu_{V_\alpha}(x).$$
\end{lemma}

\begin{proof}
From \eqref{eq:equilibrium_relation}, we see that using the Hilbert transform $\mathcal{H}$:
$$\Xi_\alpha[f](x)=\mc{H}\left[f\rho_{V_\alpha}\right](x)=\fint_{\R}\dfrac{f(t)}{t-x}\diff\mu_{V_\alpha}(t).$$
$\mathcal{H}$ seen as bounded operator on $L^p(\R)$, $1<p<\infty$, satisfies for $u\in L^p(\R)$ and $v\in L^q(\R)$ with $p^{-1}+q^{-1}=1$,
$$\braket{\mc{H}[u],v}_{L^2(\R)}=-\braket{u,\mc{H}[v]}_{L^2(\R)}.$$ 
It also verifies $\mc{H}[u']=\mc{H}[u]'$ for all $u\in L^p(\R)$ such that $u'\in L^p(\R)$, $p>1$ ; see for example \cite{king2009Hilbert}. Using these facts and by integration by parts, we get
\begin{multline*}
	\int_{-1}^1\Xi_\alpha[f]'(x)g(x)\diff\mu_{V_\alpha}(x)=\braket{\mc{H}[(f\rho_{V_\alpha})'],g\rho_{V_\alpha}}_{L^2(\R)}=-\braket{(f\rho_{V_\alpha})',\mc{H}[g\rho_{V_\alpha}]}_{L^2(\R)}
	\\=\braket{f\rho_{V_\alpha},\mc{H}[g\rho_{V_\alpha}]'}_{L^2(\R)}=\int_{-1}^1\Xi_\alpha[g]'(x)f(x)\diff\mu_{V_\alpha}(x).
\end{multline*}
We used that, $(f\rho_{V_\alpha})'\in L^p(\R)$ for $1<p<2$ and that $g\rho_{V_\alpha}\in L^q(\R)$ for $q>2$. Now, since $f,g\in\mc{C}^2(\R)$ and because of Lemma \ref{lem:inverse_master_op} taking $\widetilde{f}=\Xi_\alpha^{-1}[f]$ and $\widetilde{g}=\Xi_\alpha^{-1}[g]$ yields the result.
\end{proof}

We can now prove Corollary~\ref{thm:subleadingZN} as a direct consequence of a the previous lemmas and Theorem~\ref{thm:CLTalpha}.



\begin{proof}[Proof of Corollary \ref{thm:subleadingZN}]
The fact that the first order of the free energy is given by $$\lim\limits_{N\rightarrow\infty}\dfrac{1}{N^{2}\beta}\log Z_N
=- \frac{1}{2} \inf\limits_{\mu\in\mc{M}_1(\R)}\mc{E}(\mu)
=- \frac{1}{2} \mc{E}(\mu_V)$$ can be proven by standard large deviation techniques, see \cite[Theorem 2.6.1]{AndGuiZei10}. From \cite[Lemma 4.1]{DadFraGueZit2023}, we have that $-\mc{E}(\mu_V)=\log 2+\frac{3}{2p}$. 
Finally by Lemmas \ref{lem:interpolZN} and \ref{lem:logZNVG} and Theorem \ref{thm:CLTalpha}, we can identify the terms proportional to $N\log N$ and $N$ as well as the size of the remainder. We find
\begin{equation} \label{eq:F-F_G}
F^{\{-1\}}-F_G^{\{-1\}}=-\dfrac{1}{2}\left(\frac{1}{2}-\dfrac{1}{\beta}\right)\int_{0}^1 \diff\alpha\int_{-1}^1\Xi_\alpha^{-1}[V_\alpha]'(x) \diff \mu_{V_\alpha}(x).
\end{equation}
Following \cite[Theorem 4.28]{Gui2019}, \cite[Lemma 6.3]{guionnet2024asymptotic} it remains to show that this double integral can be recasted in terms of $\mrm{Ent}[\mu_V]$. By differentiating \eqref{eq:identiteintegralmuvalpha} with respect to $\alpha\in[0,1]$, we have:
$$-\int_{-1}^1\Xi_\alpha[h](x)\partial_\alpha\rho_{V_\alpha}(x)\diff x+\dfrac{1}{2}\int_{-1}^1\partial_\alpha V'_\alpha(x)h(x)\rho_{V_\alpha}(x)\diff x=0$$
for any $h\in\mc{C}^1([-1,1])$. 
Now, replacing $h$ with $-\Xi_\alpha^{-1}[f]$ with $f\in\mc{C}^2(\R)$ and using Lemma \ref{lem:symmetryXi}, we find:
\begin{equation}\label{eq:eqentropie1}
	\int_{-1}^1f(x)\partial_\alpha\rho_{V_\alpha}(x)\diff x-\dfrac{1}{2}\int_{-1}^1\Xi_\alpha^{-1}[\partial_\alpha V_\alpha](x)f'(x)\rho_{V_\alpha}(x)\diff x=0
\end{equation}
where we used that $\Xi_\alpha[h] = f-a$ for some constant $a$ and that $\int_{-1}^1\partial_\alpha\rho_{V_\alpha}(x)\diff x=0$.
By integrating by parts, we find:
\begin{multline}\label{eq:eqentropie2}
	\dfrac{1}{2}\int_{-1}^1\Xi_\alpha^{-1}[\partial_\alpha V_\alpha](x)f'(x)\rho_{V_\alpha}(x)\diff x=-\dfrac{1}{2}\int_{-1}^1\Xi_\alpha^{-1}[\partial_\alpha V_\alpha]'(x)f(x)\rho_{V_\alpha}(x)\diff x\\-\dfrac{1}{2}\int_{-1}^1\Xi_\alpha^{-1}[\partial_\alpha V_\alpha](x)f(x)\rho_{V_\alpha}'(x)\diff x.
\end{multline}
We can extend the previous relation to $f=\log\rho_{V_\alpha}$ by density, since it is smooth and has singularity at most like $O\left(\log(1-x^2)\right)$ at the endpoints. Thus, we obtain by integrating by parts twice, then using \eqref{eq:eqentropie2} and \eqref{eq:eqentropie1} and the fact that $\int_{-1}^1\partial_\alpha\rho_{V_\alpha}(x)\diff x=0$:
\begin{multline*}
	\int_{-1}^1\Xi_{\alpha}^{-1}[\partial_\alpha V_\alpha]'(x)\rho_{V_\alpha}(x)\diff x=-\int_{-1}^1\Xi_{\alpha}^{-1}[\partial_\alpha V_\alpha](x)\dfrac{\rho_{V_\alpha}'(x)}{\rho_{V_\alpha}(x)}\rho_{V_\alpha}(x)\diff x
	\\=\int_{-1}^1\Xi_\alpha^{-1}[\partial_\alpha V_\alpha]'(x)\log\rho_{V_\alpha}(x)\rho_{V_\alpha}(x)\diff x+\int_{-1}^1\Xi_\alpha^{-1}[\partial_\alpha V_\alpha](x)\log\rho_{V_\alpha}(x)\rho_{V_\alpha}'(x)\diff x
	\\=-2\int_{-1}^1\log\rho_{V_\alpha}(x)\partial_\alpha\rho_{V_\alpha}(x)\diff x=-2\partial_\alpha\int_{-1}^1\log\rho_{V_\alpha}(x)\rho_{V_\alpha}(x)\diff x=2\partial_\alpha\mrm{Ent}[\mu_{V_\alpha}].
\end{multline*}
Plugging this in \eqref{eq:F-F_G}, integrating w.r.t.\@ $\alpha$ and noticing that $\mrm{Ent}[\mu_G]=\log\pi-\frac{1}{2}$ is enough to conclude.
\end{proof}

\section{Schatten balls: KLS conjecture, volume of unit balls}
\label{sec:convex}

\subsection{Consistency check of the KLS conjecture}
\label{subsec:checkKLS}
The goal of this section is to show Theorem \ref{thm:KLS}. We consider $p\geq 2$ and we denote by $X$ a uniformly distributed random variable on $B_E(S_p^N)\subset\R^{d_N}$. 
To have more compact expressions, we define 
\[
G_{r,q} \coloneqq \Ec{\braket{\mu_N,x^r}^q} 
= \Ec{\left( \frac{1}{N} \sum_{k=1}^N \lambda_k^r \right)^q}.
\]
We also recall that $d_N=\frac{\beta}{2}N(N-1)+N$ where $\beta=1,2,4$ when $\F=\R,\C, \HH$ and that $c_p$ was introduced in \eqref{def:cp}.
\begin{lemma}\label{lem:straymond2}
	Let $r\geq 2$ be an even integer, $q \in [0,\infty)$ and $f(X)=\mrm{Tr}\left(X^{r}\right)^q$ for all $X\in B_E(S_p^N)$.
	Then
	\begin{equation*}
		\Ec{f(X)}=\frac{\Gamma\left(1+\frac{d_N}{p}\right)}{\Gamma\left(1+\frac{d_N+rq}{p}\right)}\left(\frac{\beta Nc_p}{2}\right)^{rq/p}N^{q}G_{r,q}.
	\end{equation*}
\end{lemma}

\begin{proof}
	Note that $f(X)=\sigma_r(X)^q$ depends only on the eigenvalues of $X$. 
	Seen as a function of the eigenvalues, it is symmetric and positively homogeneous of degree $rq$. Hence Weyl's formula Lemma \ref{lem:straymond} applies and, together with \eqref{eq:ball_and_Z}, it gives the result.
\end{proof}

As our goal is to establish the KLS conjecture in a specific case, we have to consider the Hilbert-Schmidt norm of the gradient of $f$.

\begin{lemma}[Bound on the gradient]\label{lem:gradiantf}
	Let $r\geq 2$ be an even integer, $q \in [1,\infty)$ and $f(X)=\mrm{Tr}\left(X^{r}\right)^q$ for all $X\in B_E(S_p^N)$. Then
	\begin{equation*}
		\Ec{\|\nabla f(X)\|_{\mrm{HS}}^2}\geq(rq)^2\dfrac{\Gamma\left(u\right)}{\Gamma\left(u+v_2\right)}\left(\dfrac{\beta Nc_p}{2}\right)^{v_2}N^{2q-1}\Ec{\braket{\mu_N,x^r}^{2q-2}\braket{\mu_N,x^{2r-2}}},	
	\end{equation*}
	with $u=1+d_N/p$ and $v_2=2(rq-1)/p$.
\end{lemma}

\begin{proof}
	We view $X$ as an element of $\R^{d_N}$, as $\F$ is a $\beta$-dimensional real vector space. We denote by $\partial_{i,j,k}$ the partial derivative with respect to the $k$-th real component of the entry $X_{ij}$. For $i>j$ and $k\in\llbracket1,\beta\rrbracket$, we have
	\[
	\partial_{i,j,k}f(X)
	= q\mrm{Tr}\left(X^{r}\right)^{q-1} \cdot r \mrm{Tr}\left(X^{r-1}\partial_{i,j,k}X\right)
	=rq\mrm{Tr}\left(X^{r}\right)^{q-1} 
	\sum_{a,b=1}^{N} (X^{r-1})_{ab} (\partial_{i,j,k}X)_{ba}.
	\]
	But $\partial_{i,j,k}X=e_kE_{ij}+(\1_{k=1}-\1_{k>1})e_kE_{ji}$ where $e_k$ is the $k$-th unit vector of $\F$ seen as a real vector space and $E_{ij}$ is the $N\times N$ matrix with 1 in the entry $(ij)$ and 0 otherwise.
	Hence, using also that $X$ is self-adjoint, we have
	\begin{align*}
	    \sum_{a,b=1}^{N}(X^{r-1})_{ab} (\partial_{i,j,k}X)_{ba}
	    & = e_k (X^{r-1})_{ij}^{*}
    	+(\1_{k=1}-\1_{k>1})e_k(X^{r-1})_{ij} \\
    	& = \1_{k=1} 2e_k \re (X^{r-1})_{ij}
    	+\1_{k>1} 2e_k \im_\F (X^{r-1})_{ij},
	\end{align*}
	where, for $z \in \F$, $\re(z)$ denotes its real part and $\im_\F(z) = z-\re(z)$ (note that if $z \in \C$, then $\im_\F(z) = \ii \im z$).
	Therefore, we get, for $i>j$,
	\[
	    \sum_{k=1}^\beta \abs{\partial_{i,j,k}f(X)}^2
	    \geq 4(rq\mrm{Tr}\left(X^{r}\right)^{q-1})^{2} \cdot  \abso{(X^{r-1})_{ij}}^2.
	\]
	In the same way,
	$\partial_{i,i,1}f(X)=rq\mrm{Tr}(X^r)^{q-1}(X^{r-1})_{ii}$, thus
	\begin{align*}
		\norme{\nabla f(X)}_{\mrm{HS}}^2
		& = \sum_{i=1}^{N} \abs{\partial_{i,i,1}f(X)}^2 
		+ \sum_{i>j} \sum_{k=1}^\beta \abs{\partial_{i,j,k}f(X)}^2 \\
		& \geq (rq\mrm{Tr}\left(X^{r}\right)^{q-1})^{2} \cdot
		\left( \sum_{i=1}^{N} \abso{(X^{r-1})_{ii}}^2
		+ 4 \sum_{i>j} \abso{(X^{r-1})_{ij}}^2 \right) \\
		& \geq (rq\mrm{Tr}\left(X^{r}\right)^{q-1})^{2} \cdot
		\sum_{i,j} \abso{(X^{r-1})_{ij}}^2 \\
		& = (rq\mrm{Tr}\left(X^{r}\right)^{q-1})^{2} \cdot \mrm{Tr}\left(X^{2r-2}\right).
	\end{align*}
	Taking the expectation and using Weyl's formula Lemma \ref{lem:straymond} with the homogeneous function $X\mapsto\mrm{Tr}\left(X^{2r-2}\right)\mrm{Tr}\left(X^{r}\right)^{2q-2}$ of degree $pv_2=2(rq-1)$ yields the desired lower bound.	
\end{proof}

The lower bound obtained in the previous lemma enables us to derive a manageable upper bound for the ratio we need to control to establish the KLS conjecture for our class of functions.

\begin{lemma}\label{lem:boundratio}
Let $r\geq 2$ be an even integer, $q\in [1,\infty)$ and $f(X)=\mrm{Tr}\left(X^{r}\right)^{q}$, then
	\begin{equation}\label{eq:ineqratio}
		\dfrac{\Var{f(X)}}{\Ec{\|\nabla f(X)\|_{\mrm{HS}}^2} \cdot\sup\limits_{\theta\in\mathbb{S}^{d_N-1}}\E\left[\braket{\theta,X}^2\right]}
		\leq \dfrac{d_N}{(rq)^2}
		\dfrac{\dfrac{G_{r,2q}}{\Gamma(u)\Gamma(u+2v_0)}-\dfrac{G_{r,q}^2}{\Gamma(u+v_0)^2}}
		{\dfrac{G_{2,1}}{\Gamma(u+v_1)} \cdot \dfrac{\E[\braket{\mu_N,x^r}^{2q-2}\braket{\mu_N,x^{2r-2}}]}{\Gamma(u+v_2)}},
	\end{equation}
	with $u=1+d_N/p$, $v_0=rq/p$, $v_1=2/p$ and $v_2=2(rq-1)/p$.
\end{lemma}

\begin{proof}
	We first notice that 
	\[
	\sup_{\theta\in\mathbb{S}^{d-1}}\E\left[\braket{\theta,X}^2\right]=\lambda_{\max}(\Sigma)\geq\dfrac{\mrm{Tr}\left(\Sigma\right)}{d_N}=\dfrac{\Ec{\| X\|_{\mrm{HS}}^2}}{d_N},
	\]
	where $\Sigma$ is the covariance matrix of $X$.
	Thus, by Lemma~\ref{lem:straymond2} with $r=2$ and $q=1$,
	\[
	\lambda_{\max}(\Sigma)\geq\dfrac{\Gamma(u)}{d_N\Gamma(u+v_1)}\left(\dfrac{\beta Nc_p}{2}\right)^{v_1}NG_{2,1}.
	\]
	Furthermore, applying Lemma~\ref{lem:straymond2} twice, we get
	\begin{equation*}
		\Var{f(X)}
		=\Ec{f(X)^2}-\Ec{f(X)}^2
		=\left(\dfrac{\beta Nc_p}{2}\right)^{2v_0}N^{2q} 
		\left( \dfrac{\Gamma(u)}{\Gamma(u+2v_0)}G_{r,2q}
		-\dfrac{\Gamma(u)^2}{\Gamma(u+v_0)^2}G_{r,q}^2 \right).
	\end{equation*}
	Finally, Lemma \ref{lem:gradiantf} provides a lower bound on $\Ec{\|\nabla f(X)\|_{\mrm{HS}}^2}$ and, combining the different blocks and using $v_1+v_2=2v_0$, yields the result.
\end{proof}

In the following lemma, we compute explicitly the moments of the equilibrium measure $\mu_V$. This will be needed to compute the limit of the quantities $G_{r,q}$.

\begin{lemma}\label{lem:momentsmuV}
	Let $k\geq 1$ be an integer, then
	\[
	\braket{\mu_V,x^{2k}}=2^{-2k}\binom{2k}{k}\cdot \dfrac{p}{p+2k}.
	\]
\end{lemma}

\begin{proof}
	As in \cite[Lemma 4.1]{van2006asymptotics}, $\braket{\mu_V,x^{2k}}=\E[A^{2k}]\cdot \E[B^{2k}]$ where $A$ and $B$ are independent with respectively arcsine distribution on $(-1,1)$ and Beta$(p,1)$. Thus, since
	$$\Ec{A^{2k}}=\dfrac{(2k)!}{2^{2k}(k!)^2}=4^{-k}\binom{2k}{k},\hspace{1cm}\Ec{B^{2k}}=\dfrac{p}{p+2k},$$
	we conclude.
\end{proof}

The asymptotic variance of the linear statistics will also be needed to analyze the ratio appearing in Lemma~\ref{lem:boundratio}.

\begin{lemma}\label{lem:variancemunxr}
    Let $r\geq 2$ be an even integer and $q\in[1,\infty)$, then
	\[
	\lim_{N\rightarrow\infty}d_N\Var \left( \braket{\mu_N,x^r}^q \right)
	\leq\dfrac{rq^2\braket{\mu_V,x^r}^{2q-2}}{4^{r-1}}\binom{2r-2}{r-1}.
	\]
\end{lemma}

\begin{proof}
    First note that, if $q>1$, then by \eqref{eq:g'_Taylor_order_2}, there exists $C > 0$ depending on $q$ such that for any $a>0$ and $h \geq -a$, we have
    \[
        \abs{(a+h)^q-a^q-qa^{q-1}h} 
        \leq C h^2 (\abs{a} \vee \abs{a+h})^{q-2}
        \leq C h^2 \abs{a}^{q-2} + C h^q \1_{q>2},
    \]
    bounding $\abs{a} \vee \abs{a+h} \geq \abs{a}$ when $q \leq 2$ and $\abs{a} \vee \abs{a+h} \leq \abs{a}+\abs{h}$ when $q > 2$.
    Note that the bound above is also true for $q=1$ because the LHS is simply zero.
    Therefore, if we decompose
    \[
    \braket{\mu_N,x^r}^q 
    = \left(\braket{\mu_V,x^r}+\dfrac{L_N(x^r)}{N}\right)^q
    \eqqcolon \braket{\mu_V,x^r}^q+q\braket{\mu_V,x^r}^{q-1} 
    \left( \frac{L_N(x^r)}{N} + \mathrm{Rem}^{(3)}_N \right),
    \]
    then the remainder satisfies, with a constant $C$ depending on $q$ and $a=\braket{\mu_V,x^r} \in (0,\infty)$,
    \[
    \abs{\mathrm{Rem}^{(3)}_N} \leq C \dfrac{L_N(x^r)^2}{N^2} + C \dfrac{L_N(x^r)^q}{N^q} \1_{q>2}.
    \]
    We obtain
    \[
        \Var\left(\braket{\mu_N,x^r}^q\right)
        = q^2\braket{\mu_V,x^r}^{2q-2} 
        \Var\left(\frac{L_N(x^r)}{N} + \mathrm{Rem}^{(3)}_N\right)
        \underset{N\to\infty}{\sim} \dfrac{\beta q^2\braket{\mu_V,x^r}^{2q-2}}{2d_N} 
        \sigma^2(x^r),
    \]
    using $N^2 \sim 2d_N/\beta$ and that $L_N(x^r)$ converges in terms of moments towards $\mc{N}(m_V(x^r),\sigma^2(x^r))$ by Theorem~\ref{thm:CLT}.
	Now if $A$ and $B$ are two independent arcsine (thus symmetric) distributions on $(-1,1)$, we then have from the definition of $\sigma^2$:
	$$\lim_{N\rightarrow\infty}d_N\Var{\braket{\mu_N,x^r}^q}=\dfrac{\beta q^2\braket{\mu_V,x^r}^{2q-2}}{2}\sigma^2(x^r)=\dfrac{q^2\braket{\mu_V,x^r}^{2q-2}}{2}\Ec{\left(\sum_{i=0}^{r-1}A^{r-1-i}B^{i}\right)^2(1-AB)},$$
	Furthermore, we have by the crude bound $1-AB\leq 2$, symmetry and independence:
	\begin{multline*}
	    \Ec{\left(\sum_{i=0}^{r-1}A^{r-1-i}B^{i}\right)^2(1-AB)}\leq 2\Ec{\sum_{i=0}^{r-1}\# \{ (k,\ell) \in \llbracket0,r-1 \rrbracket^2 : k+\ell = 2i \}A^{2r-2-2i}B^{2i}}
	    \\\leq2r\Ec{\sum_{i=0}^{r-1}A^{2r-2-2i}B^{2i}}.
	\end{multline*}
	Recalling that $\Ec{A^{2k}}=4^{-k}\binom{2k}{k}$, we obtain:
	\begin{multline*}
	    \lim_{N\rightarrow\infty}d_N\Var(\braket{\mu_N,x^r}^q)
	    \leq2r\dfrac{q^2\braket{\mu_V,x^r}^{2q-2}}{2}\sum_{i=0}^{r-1}\Ec{A^{2r-2-2i}}\Ec{B^{2i}}\\=2r\dfrac{q^2\braket{\mu_V,x^r}^{2q-2}}{2}\sum_{i=0}^{r-1}\dfrac{\binom{2r-2-2i}{r-1-i}}{4^{r-1-i}}\cdot \dfrac{\binom{2i}{i}}{4^{i}}=\dfrac{rq^2\braket{\mu_V,x^r}^{2q-2}}{4^{r-1}}\binom{2r-2}{r-1}
	\end{multline*}
	 by using the Vandermonde identity, which is the desired bound.
\end{proof}

We are now able to give the consistency check on the KLS conjecture.

\begin{proof}[Proof of Theorem \ref{thm:KLS}]
	We start from Lemma \ref{lem:boundratio} and bound first the numerator $\cN_N$ of the RHS of \eqref{eq:ineqratio}. We begin
	\begin{multline*}
		\dfrac{\mc{N}_N}{d_N}
		\coloneqq \dfrac{G_{r,2q}}{\Gamma(u)\Gamma(u+2v_0)}
		-\dfrac{G_{r,q}^2}{\Gamma(u+v_0)^2} \\
		=\dfrac{\Var{\braket{\mu_N,x^r}^{q}}}{\Gamma(u)\Gamma(u+2v_0)}
		+\dfrac{G_{r,q}^2}{\Gamma(u)^2} 
		\left(\dfrac{\Gamma(u)}{\Gamma(u+2v_0)}
		-\left(\dfrac{\Gamma(u)}{\Gamma(u+v_0)}\right)^{2}\right) 
		\eqqcolon (1)+(2),
	\end{multline*}
	recalling $u=1+d_N/p$ and $v_0=rq/p$. 
    Using Lemma \ref{lem:variancemunxr}, $\Gamma(u+2v_0) \sim u^{2v_0} \Gamma(u)$ and Stirling's formula $\Gamma(u) \sim u^u e^{-u} \sqrt{2\pi/u}$, we have as $N\rightarrow\infty$:
	\[
	(1)\leq \dfrac{rq^2\braket{\mu_V,x^r}^{2q-2}}{4^{r-1}}\binom{2r-2}{r-1} \frac{u^{-2u-2v_0} e^{2u}}{2\pi p}(1+o(1)).
	\]
	Using also
	\[
	\dfrac{\Gamma(u+v_0)}{\Gamma(u)}
	=u^{v_0}\left(1+\dfrac{v_0(v_0-1)}{2u}+O(u^{-2})\right),
	\]
	one obtains:
	\[
	(2)\sim-\braket{\mu_V,x^r}^{2q}\dfrac{v_0^{2}}{2\pi}u^{-2u-2v_0}e^{2u}<0.
	\]
	Thus, as $N\rightarrow\infty$,
	\begin{equation}\label{eq:equiNN}
	    \dfrac{\mc{N}_N}{d_N}\leq\dfrac{u^{-2u-2v_0} e^{2u}}{4^{r-1}} \cdot \dfrac{rq^2\braket{\mu_V,x^r}^{2q-2}}{2\pi p}\binom{2r-2}{r-1}(1+o(1)).
	\end{equation}
	Finally, we need to study the denominator $\mc{D}_N$ of the RHS of \eqref{eq:ineqratio}:
	\[
	\dfrac{\mc{D}_N}{(rq)^2}\coloneqq\dfrac{G_{2,1}\E[\braket{\mu_N,x^r}^{2q-2}\braket{\mu_N,x^{2r-2}}]}{\Gamma\left(u+v_1\right)\Gamma(u+v_2) },
	\]
	with $v_1=2/p$ and $v_2=2(rq-1)/p$. 
	First, by the law of large numbers for $\mu_N$ and then Lemma~\ref{lem:momentsmuV},
	\begin{multline*}
		G_{2,1}\Ec{\braket{\mu_N,x^r}^{2q-2}\braket{\mu_N,x^{2r-2}}} \tend{N\rightarrow\infty} \braket{\mu_V,x^2}\braket{\mu_V,x^{2r-2}}\braket{\mu_V,x^{r}}^{2q-2}
		\\=2\braket{\mu_V,x^{r}}^{2q-2} \cdot 4^{-r}\dfrac{p}{p+2}\dfrac{p}{p+2r-2}\binom{2r-2}{r-1}.
	\end{multline*}
	On the other hand, using previously mentioned properties of $\Gamma$ and noticing that $v_1+v_2=2v_0$, we obtain:
	$$\dfrac{1}{\Gamma\left(u+v_1\right)\Gamma(u+v_2) }\sim\dfrac{u^{-2u-2v_0+1}e^{2u}}{2\pi}.$$
	We thus deduce that as $N\rightarrow\infty$:
	\begin{equation}\label{eq:equiDN}
	    \dfrac{\mc{D}_N}{(rq)^2}\sim\dfrac{u^{-2u-2v_0+1}e^{2u}}{\pi}\braket{\mu_V,x^{r}}^{2q-2} \cdot 4^{-r}\dfrac{p}{p+2}\dfrac{p}{p+2r-2}\binom{2r-2}{r-1}
	\end{equation}
	and thus combining \eqref{eq:equiNN} and \eqref{eq:equiDN},
	$$\dfrac{\mc{N}_N}{\mc{D}_N}\leq\dfrac{2}{r}\dfrac{p+2}{p}\dfrac{p+2r-2}{p}\left (1+o(1)\right )\leq4+o(1),$$
	noting that the first bound obtained is decreasing in $p$, and so is maximized for $p=2$. This yields the desired claim.
\end{proof}

\subsection{Volume of Schatten balls}
\label{subsec:corballs}

In this subsection, we prove Corollary \ref{cor:volume_balls} which gives an asymptotic expansion for the volume of the unit Schatten balls $B_E(S_p^N)$.
Similar calculations have been performed in \cite{Sonn}.
\begin{proof}[Proof of Corollary \ref{cor:volume_balls}.]
	By \eqref{eq:ball_and_Z} and Corollary \ref{thm:subleadingZN}, we have as $N\rightarrow \infty$,
    \begin{multline}
    \label{eq:asymptotiqueBoule}
	\log\abso{B_E(S_p^N)}
    		= \frac{d_N}{p}\log\left(\frac{\beta N c_p}{2}\right) +\log c_N - \log\Gamma\left(1+\frac{d_N}{p}\right)  -N^2\frac{\beta}{2}\left(\log 2+\dfrac{3}{2p}\right)\\ +\frac{\beta}{2}N\log N
    		 + N\beta F^{\{-1\}}+o\left(N\right),
    \end{multline}
    where we recall that $d_N=\beta N(N-1)/2+N$ and we develop each term in terms of $N$. We find the asymptotic of $\log c_N$ as in \cite[Lemma 4.7]{DadFraGueZit2023} by using, by \cite[Proposition 4.1.14]{AndGuiZei10},
	$$ \abs{\mc{U}_N(\F)} = \dfrac{(2\pi)^{\frac{\beta N(N+1)}{4}}2^{N(1-\frac{\beta}{2})}}{\prod_{k=1}^N\Gamma(\beta k /2)}.$$
	Using Stirling's formula
	$$ \log\Gamma(x) =x\log x-x-\dfrac{1}{2}\log x+\dfrac{1}{2}\log(2\pi)+o_{x\to +\infty}(1),$$
	and the asymptotic expansions
	$$ \sum_{k=1}^N k\log k =\frac{1}{2}N^2\log N -\frac{1}{4}N^2 + \frac{1}{2}N\log N + o(N),\hspace{1cm} \sum_{k=1}^N \log k = N\log N-N+o(N),$$
	one gets from the definition of $c_N$, see Lemma \ref{lem:straymond},
	\begin{multline*}
	    \log c_N = -\frac{\beta}{4}N^2\log N + \frac{\beta}{4}\left(\frac{3}{2} + \log\frac{4\pi}{\beta}\right)N^2 -\left(\frac{\beta}{4}+\dfrac{1}{2}\right) N\log N +\\
	    \left(\log\Gamma(\beta/2) + \frac{1}{2} +\frac{\beta}{4}\left(1-\log(\pi\beta)\right)-\dfrac{1}{2}\log\frac{4\pi}{\beta}\right)N + o(N).
	\end{multline*}
	By Stirling's formula again,
	\begin{multline*}
	    \log\Gamma\left(1+\frac{d_N}{p}\right) = \frac{\beta}{p}N^2\log N +\frac{\beta}{2p}\left(\log\frac{\beta}{2p}-1\right)N^2 +\frac{2-\beta}{p}N\log N  
	    \\+\frac{2-\beta}{2p}
	    \log\left(\frac{\beta}{2p}\right)N +o(N).
	\end{multline*}
	The result then follows from collecting all the terms in \eqref{eq:asymptotiqueBoule}.
\end{proof}

\appendix{

\section{Proofs of technical estimates on functions}
\label{app:technical estimate function}

\subsection{Preliminaries on an integral operator}

In this subsection, we prove some basic facts concerning the regularity of an integral transformation of a function, which appears in various quantities, e.g.\@ in the definition of $r_\alpha$ (with $\mu$ being the arcsine distribution) or in the definition of the master operator (with $\mu = \mu_{V_\alpha}$).

We introduce the following notation for Taylor expansions. If $\theta$ is a function which is $n$ times differentiable at $x$, we write
\begin{equation} \label{eq:def_Taylor}
	T_n[\theta,x](t) \coloneqq \sum_{k=0}^n \frac{\theta^{(k)}(x)}{k!} (t-x)^k, 
	\qquad \text{for any } t\in\R.
\end{equation}

\begin{lemma} \label{lem:Theta}
	Let $\mu$ be a probability measure on $[-1,1]$. 
	Let $\ell \geq 1$ and $\theta \in \cC^\ell(\R)$. 
	For $x \in \R$, let
	\begin{equation} \label{eq:def_Theta}
		\Theta (x) \coloneqq
		\int_{-1}^1 \frac{\theta(t)-\theta(x)}{t-x} \diff \mu(t).
	\end{equation}
	Then, the following properties hold.
	\begin{enumerate}
		\item\label{it:lem_Theta_1} The function $\Theta$ is in $\cC^{\ell-1}(\R)$ with $\norme{\Theta}_{\cC^{\ell-1}} \leq \norme{\theta}_{\cC^\ell}$ and, for any $0 \leq k \leq \ell-1$ and $x \in \R$,
		\begin{equation} \label{eq:derivative_Theta}
			\Theta^{(k)}(x) = k!
			\int_{-1}^1 \frac{\theta(t)-T_k[\theta,x](t)}{(t-x)^{k+1}} \diff \mu(t).
		\end{equation}
		Moreover, $\Theta$ is $\cC^{\ell}$ on $\R\setminus [-1,1]$ and \eqref{eq:derivative_Theta} holds with $k=\ell$ and $x \notin [-1,1]$.
		\item\label{it:lem_Theta_2} If $\int_{-1}^1 \frac{1}{1-t} \diff \mu(t) < \infty$, then $\Theta$ is $\ell$ times differentiable at 1 and \eqref{eq:derivative_Theta} holds with $k=\ell$ and $x=1$.
		\item\label{it:lem_Theta_3} If $\mu$ has a density $\varrho$ such that, for some $M>0$, for any $x \in[-1,1]$, $\varrho(x) \leq M \sqrt{1-x}$, and if, for any $x \geq 1$, $\theta(x) = T_\ell[\theta,1](x)$, then
		\begin{equation} \label{eq:lem_Theta_3}
			\abs{\Theta^{(\ell)}(x)-\Theta^{(\ell)}(1)}
			\begin{cases}
				\leq 8 M (\ell+1) \left( \sup_{[-1,1]} \lvert \theta^{(\ell)} \rvert \right) (x-1)^{1/2}, & \text{for } x \in [1,2], \\
				= o((x-1)^{1/2}), & \text{as } x \downarrow 1.
			\end{cases}
		\end{equation}
	\end{enumerate}
\end{lemma}

\begin{proof}
	Part \ref{it:lem_Theta_1}. This is proved by induction using differentiation under the integral. For $0 \leq k \leq \ell-1$ the domination is given by 
	\[
	\abs{\frac{\theta(t)-T_k[\theta,x](t)}{(t-x)^{k+1}}} 
	\leq \lVert \theta^{(k+1)} \rVert_{\infty},
	\]
	which is an application of Taylor--Lagrange inequality. 
	For $k=\ell$ and on $\R \setminus [-1,1]$, the domination is done using that $1/(t-x)$ is bounded uniformly in $t\in[-1,1]$ and $x$ in a compact subset of $\R \setminus [-1,1]$.
	
	Part \ref{it:lem_Theta_2}. First assume $\ell \geq 2$. Using Part \ref{it:lem_Theta_1} and $\frac{1}{t-x} = \frac{1}{t-1} + \frac{x-1}{(t-x)(t-1)}$, we have, for any $x \in \R$,
	\begin{align}
		\Theta^{(\ell-1)}(x) 
		& = (\ell-1)! \left( \int_{-1}^1 \frac{\theta(t)-T_{\ell-1}[\theta,x](t)}{(t-x)^{\ell-1}} \frac{\diff \mu(t)}{t-1}
		+ (x-1) \int_{-1}^1 \frac{\theta(t)-T_{\ell-1}[\theta,x](t)}{(t-x)^{\ell}} \frac{\diff \mu(t)}{t-1} \right) \nonumber \\
		& = (\ell-1) \int_{-1}^1 \int_0^1 \theta^{(\ell-1)}(x+u(t-x)) (1-u)^{\ell-2} \diff u \frac{\diff \mu(t)}{t-1}
		- \theta^{(\ell-1)}(x) \int_{-1}^1 \frac{\diff \mu(t)}{t-1} \nonumber \\
		& \qquad {} + (x-1) (\ell-1)! \int_{-1}^1 \frac{\theta(t)-T_{\ell-1}[\theta,x](t)}{(t-x)^{\ell}} \frac{\diff \mu(t)}{t-1}, \label{eq:rewriting_Theta}
	\end{align}
	where, in the first integral, we decomposed $\theta(t)-T_{\ell-1}[\theta,x](t) = \theta(t)-T_{\ell-2}[\theta,x](t) - \frac{(t-x)^{\ell-1}}{(\ell-1)!} \theta^{(\ell-1)}(x)$ and applied Taylor integral formula.
	Note that the first term on the right-hand side is simply 0 when $\ell=1$.
	Therefore, we get
	\begin{align}
		& \frac{\Theta^{(\ell-1)}(x)-\Theta^{(\ell-1)}(1)}{x-1} \nonumber \\
		& = (\ell-1) \int_{-1}^1 \int_0^1 
		\frac{\theta^{(\ell-1)}(x+u(t-x))-\theta^{(\ell-1)}(1+u(t-1))}{x-1}
		(1-u)^{\ell-2} \diff u \frac{\diff \mu(t)}{t-1} \nonumber \\
		& \qquad {} - \frac{\theta^{(\ell-1)}(x)-\theta^{(\ell-1)}(1)}{x-1} \int_{-1}^1 \frac{\diff \mu(t)}{t-1} 
		+ (\ell-1)! \int_{-1}^1 \frac{\theta(t)-T_{\ell-1}[\theta,x](t)}{(t-x)^{\ell}} \frac{\diff \mu(t)}{t-1} \nonumber \\
		& \xrightarrow[x\to1]{} (\ell-1) \int_{-1}^1 \int_0^1 
		\theta^{(\ell)}(1+u(t-1)) (1-u)^{\ell-1} \diff u \frac{\diff \mu(t)}{t-1} \nonumber \\
		& \qquad {} - \theta^{(\ell)}(1) \int_{-1}^1 \frac{\diff \mu(t)}{t-1} 
		+ (\ell-1)! \int_{-1}^1 \frac{\theta(t)-T_{\ell-1}[\theta,1](t)}{(t-1)^{\ell}} \frac{\diff \mu(t)}{t-1}, \label{eq:derivative_at_1}
	\end{align}
	by dominated convergence for the first and the third term, using Taylor--Lagrange inequality and the local boundedness of $\theta^{(\ell)}$ for the domination. We also used here the assumption $\int_{-1}^{1} \frac{1}{1-t} \diff \mu(t) < \infty$.
	Using again Taylor integral formula shows that the right-hand side of \eqref{eq:derivative_at_1} equals the right-hand side of \eqref{eq:derivative_Theta} with $k=\ell$ and $x=1$.
	The case $\ell = 1$ is similar, but one should use directly the first equality in \eqref{eq:rewriting_Theta}, without rewriting further using Taylor integral formula.
	
	Part \ref{it:lem_Theta_3}. Using that $\theta(x) = T_\ell[\theta,1](x)$ for any $x \geq 1$, we also have $\theta^{(k)}(x)=T_{\ell-k}[\theta^{(k)},1](x)$ for any $0 \leq k \leq \ell$ and therefore $T_{\ell}[\theta,x](t) = T_{\ell}[\theta,1](t)$ for any $t \in \R$.
	Using this in \eqref{eq:derivative_Theta} and that, for $t<1<x$, $\lvert \frac{1}{(t-x)^{\ell+1}} - \frac{1}{(t-1)^{\ell+1}} \rvert \leq \frac{(\ell+1)(x-1)}{(x-t)(1-t)^{\ell+1}}$, we get
	\begin{equation} \label{eq:bound_diff_Theta_l}
		\abs{\Theta^{(\ell)}(x)-\Theta^{(\ell)}(1)}
		\leq \ell! \int_{-1}^{1} \abs{\theta(t)-T_\ell[\theta,1](t)}
		\frac{(\ell+1)(x-1)}{(x-t)(1-t)^{\ell+1}} \diff \mu(t).
	\end{equation}
	Note that, for any $t \in [-1,1]$, $\abs{\theta(t)-T_\ell[\theta,1](t)} \leq \frac{2}{\ell!} (1-t)^\ell \sup_{[-1,1]} \lvert \theta^{(\ell)} \rvert$ by Taylor--Lagrange inequality.
	Moreover, using the assumption on $\mu$, we have, for any $x \in [1,2]$, 
	\begin{align}
		\int_{-1}^{1} \frac{\diff \mu(t)}{(x-t)(1-t)}
		& \leq \int_{-1}^{1} \frac{M \diff t}{(x-t) \sqrt{1-t}} \nonumber \\
		& \leq \frac{M}{(x-1)} \int_{-1}^{1-(x-1)} \frac{\diff t}{\sqrt{1-t}}
		+ \int_{1-(x-1)}^1 \frac{M \diff t}{(1-t)^{3/2}} \nonumber \\
		& \leq \frac{4M}{\sqrt{x-1}}. \label{eq:bound_int_mu}
	\end{align}
	Coming back to \eqref{eq:bound_diff_Theta_l}, these bounds prove the first part of \eqref{eq:lem_Theta_3}.
	For the second part, let $\varepsilon>0$.
	Since $\theta$ is $\cC^\ell$, there exists $\delta > 0$ such that $\abs{\theta(t)-T_\ell[\theta,1](t)} \leq \varepsilon (1-t)^\ell$ for any $t \in [1-\delta,1]$.
	Therefore, we have
	\begin{align*}
		\int_{-1}^{1} \frac{\abs{\theta(t)-T_\ell[\theta,1](t)}}{(x-t)(1-t)^{\ell+1}} \diff \mu(t)
		& \leq \frac{2}{\ell!} 
		\left( \sup_{[-1,1]} \lvert \theta^{(\ell)} \rvert \right) 
		\int_{-1}^{1-\delta} \frac{\diff \mu(t)}{(x-t)(1-t)} 
		+ \varepsilon  \int_{1-\delta}^{1} \frac{\diff \mu(t)}{(x-t)(1-t)} \\
		& \leq \frac{2}{\ell!} 
		\left( \sup_{[-1,1]} \lvert \theta^{(\ell)} \rvert \right) 
		\cdot \frac{2}{\delta^2}
		+ \varepsilon \cdot \frac{4M}{\sqrt{x-1}},
	\end{align*}
	using \eqref{eq:bound_int_mu} to bound the second term.
	Bounding the second term in \eqref{eq:bound_diff_Theta_l} as before, this proves the second part of \eqref{eq:lem_Theta_3}.
\end{proof}

We also need this slight modification of Lemma \ref{lem:Theta} where the function $\theta$ can be irregular at 0.

\begin{lemma} \label{lem:Thetav2}
	Let $\mu$ be a probability measure on $[-1,1]$. 
	Let $\ell \geq 1$ and $\theta \in \cC^\ell(\R^*)$ such that $\int_{-1}^1 \abs{\theta} \diff \mu < \infty$. 
	Define $\Theta(x)$ for $x \in \R^*$ as in \eqref{eq:def_Theta}.
	Then, the following properties hold.
	\begin{enumerate}
		\item\label{it:lem_Thetav2_1} The function $\Theta$ is in $\cC^{\ell-1}(\R^*)$ and, for any $0 \leq k \leq \ell-1$ and $x \in \R^*$, \eqref{eq:derivative_Theta} holds.
		Moreover, $\Theta$ is $\cC^{\ell}$ on $\R \setminus [-1,1]$ and \eqref{eq:derivative_Theta} holds with $k=\ell$ and $x \notin [-1,1]$.
		\item\label{it:lem_Thetav2_2} If $\int \frac{1}{1-t} \diff \mu(t) < \infty$, then $\Theta$ is $\ell$ times differentiable at 1 and \eqref{eq:derivative_Theta} holds with $k=\ell$ and $x=1$.
		\item\label{it:lem_Thetav2_3} If $\mu$ has a density $\varrho$ such that, for some $M>0$, for any $x \in[-1,1]$, $\varrho(x) \leq M \sqrt{1-x}$, and if, for any $x \geq 1$, $\theta(x) = T_\ell[\theta,1](x)$, then
		\begin{equation} \label{eq:lem_Thetav2_3}
			\abs{\Theta^{(\ell)}(x)-\Theta^{(\ell)}(1)}
			\begin{cases}
				\leq 8 M (\ell+1) \left( \sup_{[1/2,1]} \lvert \theta^{(\ell)} \rvert \right) (x-1)^{1/2} & \\
				\qquad {} + 2^{\ell+2} (\ell+1)!  
				\left( \int \abs{\theta} \diff \mu + \sum_{k=0}^\ell \lvert \theta^{(k)}(1) \rvert \right) (x-1), & \text{for } x \in [1,2], \\
				= o((x-1)^{1/2}), & \text{as } x \downarrow 1.
			\end{cases}
		\end{equation}
	\end{enumerate}
\end{lemma}

\begin{proof}
	The proof is similar to the one of Lemma \ref{lem:Theta}, only the part of the integral close to 0 has to be treated differently. We give some details below.
	
	Part \ref{it:lem_Thetav2_1}. The part of the integral with $\abs{t} \leq \abs{x}/2$ has to be treated differently, by bounding $1/(t-x)^{k+1}$ and $T_k[\theta,x](t)$ uniformly and using that $\int \abs{\theta} \diff \mu < \infty$.
	
	Part \ref{it:lem_Thetav2_2}. In the formula \eqref{eq:derivative_Theta} for $\Theta^{(\ell-1)}(x)$, we cut the integral at $1/2$. The part $t \in [1/2,1]$ is treated as in Lemma \ref{lem:Theta} and the part $t \in [-1,1/2]$ is covered via bounds similar to those mentioned in Part \ref{it:lem_Thetav2_1}.
	
	Part \ref{it:lem_Thetav2_3}. We start from \eqref{eq:bound_diff_Theta_l} and cut again the integral at $1/2$. The part $t \in [1/2,1]$ is treated as in Lemma \ref{lem:Theta}. 
	The part $t \in [-1,1/2]$ is at most
	\begin{align*} 
		& (\ell+1)! \int_{-1}^{1/2} 
		\left( \abs{\theta(t)} + \sum_{k=0}^\ell (1-t)^k \lvert \theta^{(k)}(1) \rvert \right)
		\frac{(x-1)}{(x-t)(1-t)^{\ell+1}} \diff \mu(t) \\
		& \leq (x-1) 2^{\ell+2} (\ell+1)!  
		\left( \int_{-1}^1 \abs{\theta(t)} \diff \mu(t) + \sum_{k=0}^\ell \lvert \theta^{(k)}(1) \rvert \right)
	\end{align*}
	and this yields the desired result.
\end{proof}

\subsection{Bounds on the function \texorpdfstring{$r_\alpha$}{ralpha}}

\begin{lemma}\label{lem:bound on ralpha} 
    Let $p>2$.
	For any $\alpha\in[0,1]$, the function $r_{\alpha}$ is $\cC^1(\R)\cap\cC^\infty(\R^*)$.
	Moreover, there exists $C,c>0$ such that for any $\alpha\in[0,1]$ and $\lambda\in \R$, 
	\[
	c \left( 1+\alpha \abs{\lambda}^{p-2} \right) \leq r_{\alpha}(\lambda) \leq C \left( 1+\alpha \abs{\lambda}^{p-2} \right), 
	\qquad
	\abs{r_{\alpha}'(\lambda)} \leq \begin{cases}
	    C \alpha & \text{if } p \leq 3, \\
	    C \alpha (1+\abs{\lambda})^{p-3} & \text{if } p > 3, 
	\end{cases} 
	\]
	and, if $\lambda \neq 0$, 
	\[
	\abs{r_{\alpha}''(\lambda)} 
	\leq \begin{cases}
	    C \alpha \abs{\lambda}^{p-3} & \text{if } p < 3, \\
	    C \alpha (1+(\log1/\abs{\lambda})_+) & \text{if } p = 3, \\
	    C \alpha (1+\abs{\lambda})^{p-4} & \text{if } p >3.
	\end{cases}
	\]
\end{lemma}

Note that, when $p\leq3$, the second derivative is exploding at 0, whereas it stays bounded in the neighbourhood of 0 if $p > 3$ (but possibly explodes at infinity). The function $r_\alpha$ is actually in $\cC^{\lceil p-2 \rceil}(\R)$ but we do not try to optimize regularity statements in the case $p>3$.

\begin{proof}
	First note that $r_\alpha = \alpha r_1 + (1-\alpha) r_0$ and $r_0 = 2$, so it is enough to deal with the case $\alpha = 1$. We write $r = r_1$ for brevity. 
	
	\textbf{Lower bound on $r$.} 
	Recall from \eqref{def:r_generale} that
	\begin{equation} \label{def:r_1}
		\forall \lambda \in \R, \qquad 
		r(\lambda) = \frac{1}{2\pi} \int_{-1}^1\frac{V'(\lambda)-V'(t)}{\lambda-t}
		\frac{\diff t}{\sigma(t)}.
	\end{equation}
	Note that the integrand above is positive for any $t \neq \lambda$ by strict convexity of $V$, so $r(\lambda)>0$. Moreover, $V' \in \cC^1(\R)$ so, by Lemma \ref{lem:Theta}.\ref{it:lem_Theta_1}, $r$ is continuous on $\R$. 
	Finally, note that $r(\lambda) \sim V'(\lambda)/(2\lambda) \sim pc_p\abs{\lambda}^{p-2}/2$ as $\lambda \to \pm\infty$.
	Altogether, this proves there exists $c>0$ such that $r(\lambda) \geq c(1+\abs{\lambda}^{p-2})$.

	\textbf{Regularity.}  Since $V'$ is $\cC^\infty$ on $\R^*$, by \eqref{def:r_1} and Lemma~\ref{lem:Thetav2}.\ref{it:lem_Thetav2_1}, we get that $r$ is $\cC^\infty$ on $\R^*$. 
	On the other hand, recall \eqref{def:r}. 
	Observing that the entire series has radius of convergence 1, and that the term $x\mapsto \abs{x}^{p-1}(A_p-B_p\log\abs{x})\in\mathcal{C}^1(-1,1)$ since $p>2$, we deduce that $r$ is $\cC^1$ on $(-1,1)$ and therefore on $\R$ by what precedes.

	\textbf{Upper bounds.} From \eqref{def:r}, noting that $B_p$ is nonzero only if $p\in2\N+1$, we see that there exists $C>0$ such that, for any $\lambda \in [-1/2,1/2]$, $r(\lambda) \leq C$, $\abs{r'(\lambda)} \leq C$ and 
	\[
	\abs{r''(\lambda)} \leq \begin{cases}
	    C \abs{\lambda}^{p-3} & \text{if } p < 3, \\
        C (1+(\log1/\abs{\lambda})_+) & \text{if } p = 3, \\
	    C & \text{if } p > 3.
	\end{cases}
	\]
	Then, since $r$ is $\cC^\infty$ on $\R^*$, such inequalities are still true on the compact set $[-2,2] \setminus (-1/2,1/2)$ (up to changing $C$).
	Finally, we deal with $\lambda \notin [-2,2]$ and by parity we can assume $\lambda >2$.
	By Lemma~\ref{lem:Thetav2}.\ref{it:lem_Thetav2_1}, we have, for any $k \geq 0$, 
	\begin{align*}
		\abs{r^{(k)}(\lambda)}
		\leq \frac{k!}{2\pi}
		\int_{-1}^1 \left( \sum_{j=0}^k \frac{\abso{V^{(j+1)}(\lambda)}}{j!(\lambda-t)^{k+1-j}} 
		+ \frac{\abs{V'(t)}}{(\lambda-t)^{k+1}} \right)
		\frac{\diff t}{\sigma(t)} 
		\leq C(k,p) \lambda^{p-k-2},
	\end{align*}
	using $\lambda-t \geq \lambda/2$ and $V^{(j+1)}(\lambda) = C(j,p) \lambda^{p-j-1}$, where we write $C(j,p)$ for a constant that depends on $j$ and $p$ only. This concludes the proof.
\end{proof}

We often use the following direct consequence of the previous lemma controlling the function $1/r_\alpha$ and its first two derivatives. 
\begin{corollary} \label{cor:1/r}
    Let $p>2$. There exists $C>0$ such that, for any $\alpha \in [0,1]$, $\abs{1/r_\alpha} \leq C$, $\abs{(1/r_\alpha)'} \leq C$, and for any $\lambda \in \R^*$, 
    \[
	\abs{\left(\frac{1}{r_{\alpha}}\right)''(\lambda)} 
	\leq \begin{cases}
	    C \abs{\lambda}^{p-3} & \text{if } p < 3, \\
	    C (1+(\log1/\abs{\lambda})_+) & \text{if } p = 3, \\
	    C & \text{if } p>3.
	\end{cases}
	\]
\end{corollary}

\subsection{Inverse of the master operator}
\label{subsec:inverse_master_op}

The purpose of this subsection is to prove Lemma \ref{lem:inverse_master_op} concerning the existence and regularity of the preimage under the master operator $\Xi_\alpha$ of a function.

\begin{proof}[Proof of Lemma \ref{lem:inverse_master_op}]
	Part \ref{it:inverse_C2}. This part of the result would follow directly from \cite[Lemma~3.3]{BekLebSer2018} if $V$ was of class $\cC^4$, but here $V$ is only of class $\cC^2$. We therefore summarize their proof below and explain where modifications are needed.
	First, using singular integral equation theory \cite[Eq.\@ (89.15)-(89.17)]{Mus1972}, they observe that $\psi_\alpha$ as defined in \eqref{eq:def_psi} solves $\Xi_\alpha[\psi_\alpha] = f-a$ on $(-1,1)$ and therefore on $[-1,1]$ by continuity.
	Then, they note that the function
	\[
	\varphi \colon \lambda \in \R \longmapsto
	\int_{-1}^1 \frac{f(t)-f(\lambda)}{t-\lambda}
	\frac{\diff t}{\sigma(t)}
	\]
	is in $\cC^1(\R)$ and satisfies $\norme{\varphi}_{\cC^1} \leq \norme{f}_{\cC^2}$ by Taylor--Lagrange inequality (see Lemma \ref{lem:Theta}.\ref{it:lem_Theta_1}).
	Combining this with Corollary~\ref{cor:1/r}, we deduce that $\psi_\alpha \in \cC^1([-1,1])$ and satisfies $\norme{\psi_\alpha}_{\cC^1([-1,1])} \leq C \norme{f}_{\cC^2([-1,1])}$ (here they rely on \cite[Lemma 3.1]{BekLebSer2018} which would require $V$ of class $\cC^4$).
	Furthermore, the fact that $\Xi_\alpha[\psi_\alpha] = f-a$ holds on $\R \setminus [-1,1]$ follows from direct algebraic manipulations using the definition in \eqref{eq:def_psi}.
	Moreover, the function $\psi_\alpha$ has the same regularity as $f$ on $\R \setminus [-1,1]$, i.e.\@ $\cC^2$, because $V_\alpha$ is $\cC^\infty$ on $\R^*$.
	
	To conclude that $\psi_\alpha$ is $\cC^1$ on $\R$, it is enough to prove that $\psi_\alpha^{(k)}(\lambda) \to \psi_\alpha^{(k)}(1-)$ as $\lambda \downarrow 1$ for $k \in \{0,1\}$ and a similar result at $-1$. 
	We now detail this argument%
	\footnote{We follow ideas from \cite{BekLebSer2018}, but do not fully understand their argument, so we provide some details.}, 
	focusing on the behavior at~$1$, the argument at $-1$ being identical.
	Let $\theta_\alpha$ be equal to $\psi_\alpha$ on $[-1,1]$ and extended to $\R$ in a $\cC^1$ fashion by setting $\theta_\alpha(\lambda)=T_1[\psi_\alpha,1](\lambda)$ for $\lambda>1$ and $\theta_\alpha(\lambda)=T_1[\psi_\alpha,-1](\lambda)$ for $\lambda<-1$, where we recall the notation \eqref{eq:def_Taylor}.
	Then, for $\lambda \notin [-1,1]$, using $\theta_\alpha=\psi_\alpha$ on $[-1,1]$, we have
	\begin{equation} \label{eq:rewriting_psi}
		\psi_\alpha(\lambda) 
		= \frac{\int_{-1}^1 \frac{\theta_\alpha(t)}{\lambda-t} \diff \mu_{V_\alpha}(t)+f(\lambda)-a}
		{\int_{-1}^1 \frac{1}{\lambda-t} \diff \mu_{V_\alpha}(t) -\frac{1}{2} V_\alpha'(\lambda)}
		= \theta_\alpha(\lambda) 
		+ \frac{f(\lambda)-a - \Xi_\alpha[\theta_\alpha](\lambda)}
		{\int_{-1}^1 \frac{1}{\lambda-t} \diff \mu_{V_\alpha}(t) -\frac{1}{2} V_\alpha'(\lambda)}
		\eqqcolon \theta_\alpha(\lambda) 
		+ \frac{F(\lambda)}{G(\lambda)},
	\end{equation}
	where we introduced a shorthand for the numerator and the denominator.
	It is now enough to prove that $(F/G)^{(k)}(\lambda) \to 0$ as $\lambda \downarrow 1$ for $k \in \{0,1\}$.
	Recalling the definition of $\Xi_\alpha$ in \eqref{eq:def_master_op} and applying Lemma~\ref{lem:Theta} with $\mu = \mu_{V_\alpha}$, $\theta = \theta_\alpha$ and $\ell=1$, we deduce that the function $F$ is $\cC^1$ on $(1,\infty)$, differentiable at 1 and $F'(\lambda) = F'(1) + o((\lambda-1)^{1/2})$ as $\lambda \downarrow 1$ (for the term $\frac{1}{2} \theta_\alpha V_\alpha'$, note that $\theta_\alpha'' = 0$ on $(1,\infty)$ and that $V_\alpha$ is $\cC^\infty$ on $\R^*$).
	But the function $F$ is zero on $[-1,1]$, because $\Xi_\alpha[\theta_\alpha] = \Xi_\alpha[\psi_\alpha] = f+a$ on $[-1,1]$, so in particular $F(1)=F'(1)=0$.
	Combining these facts, we get
	\begin{equation} \label{eq:behavior_F}
		F(\lambda) = o((\lambda-1)^{3/2}) 
		\qquad \text{and} \qquad 
		F'(\lambda) = o((\lambda-1)^{1/2}),
		\qquad \text{as } \lambda \downarrow 1.
	\end{equation}
	On the other hand, for $\lambda >1$, $G(\lambda) $ can be rewritten in the following way, by definitions of $s_{V_\alpha}$ and $g_\alpha$ and then using \eqref{eq:s_formula},
	\begin{equation} \label{eq:rewriting_G}
		G(\lambda) 
		= \int_{-1}^1 \frac{1}{\lambda-t} \diff \mu_{V_\alpha}(t) -\frac{1}{2} V_\alpha'(\lambda)
		= - s_{V_\alpha}(\lambda) -\frac{g_\alpha(\lambda)}{2\lambda}
		= -r_\alpha(\lambda) \sqrt{\lambda^2-1}.
	\end{equation}
	But, by Corollary~\ref{cor:1/r}, $1/r_\alpha$ and its derivative are bounded on $\R$. So, combining this with \eqref{eq:behavior_F}, we deduce that $(F/G)^{(k)}(\lambda) \to 0$ as $\lambda \downarrow 1$ for $k \in \{0,1\}$ and this concludes the proof that $\psi_\alpha$ is $\cC^1$ on $\R$.
	
	It remains to show $\norme{\psi_\alpha}_{\cC^1} \leq C \norme{f}_{\cC^2}$. We have already seen that $\norme{\psi_\alpha}_{\cC^1([-1,1])} \leq C \norme{f}_{\cC^2([-1,1])}$.
	On $\R \setminus [-2,2]$, we use the expression of $\psi_\alpha(\lambda)$ in 
	\eqref{eq:def_psi}, where the denominator can be rewritten as in \eqref{eq:rewriting_G}. Then, using Corollary~\ref{cor:1/r} again, we get 
	\[
	\norme{\psi_\alpha}_{\cC^1(\R \setminus [-2,2])} \leq C (\norme{\psi_\alpha}_{\cC^0([-1,1])} + \norme{f}_{\cC^1} + \abs{a}) \leq C \norme{f}_{\cC^2},
	\]
	noting that $\abs{a} \leq \norme{f}_\infty$.
	Now we consider $\lambda \in [1,2]$, the case $\lambda \in [-2,-1]$ being identical.
	Here we use the expression of $\psi_\alpha(\lambda)$ given in \eqref{eq:rewriting_psi}.
	First, by definition of $\theta_\alpha$, we get $\norme{\theta_\alpha}_{\cC^1([1,2])} \leq C (\psi_\alpha(1)+\psi_\alpha'(1)) \leq C \norme{f}_{\cC^2}$.
	Then, using Lemma~\ref{lem:Theta} again, we have 
	\begin{equation*}
		\abs{F'(\lambda)-F'(1)} 
		\leq (\lambda-1) \cdot 
		\sup_{(1,2]} \abs{ \left(f+\frac{1}{2} \theta_\alpha V_\alpha' \right)''} 
		+ C (\lambda-1)^{1/2} \cdot \sup_{[-1,1]} \abs{\psi_\alpha'}
		\leq C \norme{f}_{\cC^2} (\lambda-1)^{1/2},
	\end{equation*} 
	and, together with $F(1)=F'(1)=0$, this yields
	\begin{equation} \label{eq:behavior_F_2}
		\abs{F(\lambda)} \leq C \norme{f}_{\cC^2} (\lambda-1)^{3/2}
		\qquad \text{and} \qquad 
		\abs{F'(\lambda)} \leq C \norme{f}_{\cC^2} (\lambda-1)^{1/2},
		\qquad \text{for any } \lambda \in [1,2].
	\end{equation}
	Combining this with the expression of $G(\lambda)$ in \eqref{eq:rewriting_G} and Corollary~\ref{cor:1/r}, we get $\norme{F/G}_{\cC^1([1,2])} \leq C \norme{f}_{\cC^2}$, which concludes the proof of $\norme{\psi_\alpha}_{\cC^1} \leq C \norme{f}_{\cC^2}$.
	
	Part \ref{it:inverse_C3}. The argument is similar to Part \ref{it:inverse_C2}. 
	We now have $\varphi \in \cC^2(\R)$ and $\norme{\varphi}_{\cC^2} \leq \norme{f}_{\cC^3}$, which, together with Corollary~\ref{cor:1/r}, shows that $\psi_\alpha \in \cC^2([-1,1]\setminus\{0\})$ and 
	\begin{equation} \label{eq:bound_psi''}
		\forall \lambda \in (-1,1) \setminus \{0\}, \qquad \abs{\psi_\alpha''(\lambda)} \leq C \norme{f}_{\cC^3} 
		\begin{cases}
	    \abs{\lambda}^{p-3} & \text{if } p < 3, \\
	    (1+(\log1/\abs{\lambda})_+) & \text{if } p = 3, \\
	    1 & \text{if } p>3.
	\end{cases}
	\end{equation}
	On $\R \setminus [-1,1]$, the function $\psi_\alpha$ still has the same regularity as $f$, that is $\cC^3$.	
	The regularity of $\psi_\alpha$ at $\pm 1$ follows from the same lines, but we apply Lemma~\ref{lem:Thetav2} instead of Lemma~\ref{lem:Theta}, because $\psi_\alpha$, and hence $\theta_\alpha$, is not $\cC^2$ at 0.
	This shows that $F''(\lambda) = F''(1) + o((\lambda-1)^{1/2})$ as $\lambda \downarrow 1$ and it follows similarly that $(F/G)^{(k)}(\lambda) \to 0$ as $\lambda \downarrow 1$ for $k \in \{0,1,2\}$, which proves $\psi_\alpha$ is $\cC^2$ at $1$.
	Moreover, Lemma~\ref{lem:Thetav2} also yields, for $\lambda \in [1,2]$,
	\begin{align*}
		\abs{F''(\lambda)-F''(1)} 
		& \leq (\lambda-1) \cdot 
		\sup_{(1,2]} \abs{ \left(f+\frac{1}{2} \theta_\alpha V_\alpha' \right)'''} 
		+ C (\lambda-1)^{1/2} \cdot \sup_{[1/2,1]} \abs{\psi_\alpha''} \\
		& \qquad {}
		+ C (\lambda-1) \left( \int \abs{\psi_\alpha} \diff \mu_{V_\alpha} + \sum_{k=0}^2 \lvert \psi_\alpha^{(k)}(1) \rvert \right) \\
		& \leq C \norme{f}_{\cC^3} (\lambda-1)^{1/2},
	\end{align*} 
	using \eqref{eq:bound_psi''} and that $\norme{\psi_\alpha}_{\cC^1([-1,1])} \leq C \norme{f}_{\cC^2([-1,1])}$ by Part \ref{it:inverse_C2}.
	Together with Corollary~\ref{cor:1/r}, it follows that $\abs{\psi_\alpha''(\lambda)} \leq C \norme{f}_{\cC^3}$ for $\lambda \in [1,2]$, and so for $\lambda \in [-2,-1]$ as well. 
	The same bound but for $\lambda \in \R \setminus [-2,2]$ is done as in Part~\ref{it:inverse_C2}.
\end{proof}

\subsection{Bounds on \texorpdfstring{$f_{\alpha,z}$}{falpha,z} and its preimage via the master operator}
\label{subsec:fz_and_psiz}

This section contains the proofs of Lemmas \ref{lem:f_z} and \ref{lem:psi_z}.
Recall that $g(t) = tV'(t) = pc_p\abs{t}^p$ for $t \in \R$.
We first establish a preliminary result containing bounds on the function $g$, which are improved versions of the Taylor--Lagrange inequality (note that $g'''$ is unbounded when $p<3$).

\begin{lemma}\label{lem:borneABC} Let $p>2$. There exists $C>0$ such that, for any $t,\lambda \in \R$, 
	\begin{align}
		\abs{g(t) - g(\lambda) - (t-\lambda)g'(\lambda) - \frac{(t-\lambda)^2}{2} g''(\lambda)} 
		& \leq C \abs{t-\lambda}^3 (\abs{\lambda} \vee \abs{t})^{p-3},
		\label{eq:g_Taylor_order_3} \\
		\abs{g'(t) - g'(\lambda) - (t-\lambda) g''(\lambda)} 
		& \leq C \abs{t-\lambda}^2 (\abs{\lambda} \vee \abs{t})^{p-3},
		\label{eq:g'_Taylor_order_2} \\
		\abs{g''(t) - g''(\lambda)} 
		& \leq C \abs{t-\lambda} (\abs{\lambda} \vee \abs{t})^{p-3}.
		\label{eq:g''_Taylor_order_1}
	\end{align}
\end{lemma}

\begin{proof}
	Recall our notation for Taylor expansions in \eqref{eq:def_Taylor}.
	For $k \in \{0,1,2\}$, by Taylor integral formula, we have
	\begin{equation*}
		\abs{g^{(k)}(t) - T_{2-k}[g^{(k)},\lambda](t)}
		\leq \frac{\abs{t-\lambda}^{2-k}}{(2-k)!} \int_{\lambda\wedge t}^{\lambda\vee t} \abs{g'''(u)} \diff u.
	\end{equation*}
	Hence, it is enough to prove that there exists $C>0$ such that, for any $t,\lambda \in \R$, 
	\begin{equation} \label{eq:int_u_gamma-3}
		\int_{\lambda\wedge t}^{\lambda\vee t} \abs{u}^{p-3} \diff u
		\leq C \abs{t-\lambda} (\abs{\lambda} \vee \abs{t})^{p-3}.
	\end{equation}
	If $p \geq 3$, this obtained directly by bounding $\abs{u}^{p-3} \leq (\abs{\lambda} \vee \abs{t})^{p-3}$, so we now focus on the case $p \in (2,3)$.
	By symmetries, we can assume that $\lambda < t$ and that $t>0$. We distinguish cases:
	\begin{itemize}
		\item If $\lambda<0$, then
		$\int_{\lambda}^{t} \abs{u}^{p-3} \diff u 
		= \frac{1}{p-2} (\abs{\lambda}^{p-2} + t^{p-2})
		\leq C (\abs{\lambda} \vee \abs{t})^{p-2}
		\leq C (t-\lambda) (\abs{\lambda} \vee \abs{t})^{p-3}$,
		because $t-\lambda \geq \abs{\lambda} \vee \abs{t}$ in that case.
		\item If $\lambda > 0$ and $t \in (\lambda, 2\lambda]$, then $\int_{\lambda}^{t} \abs{u}^{p-3} \diff u 
		\leq (t-\lambda) \lambda^{p-3} 
		\leq C (t-\lambda) (\abs{\lambda} \vee \abs{t})^{p-3}$,
		using $\lambda^{p-3} \leq C t^{p-3}$ because  $t \leq 2\lambda$.
		\item If $\lambda > 0$ and $t > 2\lambda$, then
		$\int_{\lambda}^{t} \abs{u}^{p-3} \diff u 
		\leq C t^{p-2} 
		\leq C (t-\lambda) t^{p-3}
		= C (t-\lambda) (\abs{\lambda} \vee \abs{t})^{p-3}$,
		where in the second inequality we used that $t-\lambda \geq t/2$ in that case.
	\end{itemize}
	Therefore, we proved \eqref{eq:int_u_gamma-3} and the result follows.
\end{proof}

\begin{proof}[Proof of Lemma \ref{lem:f_z}] 
	Recall the definition of $f_{\alpha,z}$ in \eqref{def:fz}.
	Note that $f_{\alpha,z} = \alpha f_{1,z} + (1-\alpha) f_{0,z}$, so it is enough to deal with the case $\alpha=0$ and $\alpha = 1$. 
	But the case $\alpha = 0$ is immediate, because $f_{0,z}(\lambda) = 4(\lambda+z)$, so we now focus on the case $\alpha=1$.
	The regularity of $f_{1,z}$ directly follows from the one of $g$.
	Notice that for $k \in \{0,1,2\}$, $z \in \C\setminus\R$ and $\lambda \in \R$, we have
	\[
	f_{1,z}^{(k)}(\lambda)=(-1)^{k}\frac{k!}{(z-\lambda)^{k+1}}\left(g(z)-T_k[g,\lambda](z)\right),
	\]
	where the notation $T_k[g,\lambda](z)$ is introduced in \eqref{eq:def_Taylor}. The same formula holds for $k=3$ and $\lambda\neq 0$. 
	We now consider $z \in K\setminus\R$ and $\lambda \in K\cap\R$ and bound $\lvert f_{1,z}^{(k)}(\lambda)\rvert$ successively for $k \in \{0,1,2,3\}$.
	
	\textbf{Case $k=0$.} 
	Injecting the definition of $g(z)$ for complex $z$ in \eqref{eq:def_g_C},
	\begin{align*}
		\abs{f_{1,z}(\lambda)}
		&=\frac{\abs{g(x)-g(\lambda)+\ii yg'(x)-y^2g''(x)/2 -\ii y^3 g'''(x) \chi(y/x)/6}}{\abs{z-\lambda}}\\
		&\leq \frac{1}{\abs{z-\lambda}}
		\left( (\abs{x-\lambda} + \abs{y}) \norme{g'}_{[-M,M],\infty} 
		+ \frac{y^2}{2} \norme{g''}_{[-M,M],\infty}
		+ C \abs{y}^{p\vee 3}
		\right),
	\end{align*} 
	with $M$ such that $\re K \subset [-M,M]$, where we used Taylor--Lagrange inequality for the term involving $\abs{g(x)-g(\lambda)}$ and \eqref{eq:last_term_g} for the one involving $g'''(x)$. This proves $\abs{f_{1,z}(\lambda)} \leq C$ because $\abs{x-\lambda}\vee\abs{y} \leq \abs{z-\lambda}$.
	
	\textbf{Case $k=1$.} Similarly, we get
	\begin{align*}
		\abs{f_{1,z}'(\lambda)}
		&= \frac{1}{\abs{z-\lambda}^2}
		\abs{ \left( g(x)-g(\lambda)-(x-\lambda)g'(\lambda) \right)
			+ \ii y (g'(x)-g'(\lambda)) - \frac{y^2}{2} g''(x)
			- \ii \frac{y^3}{6} g'''(x) \chi \left( \frac{y}{x} \right)}\\
		& \leq \frac{1}{\abs{z-\lambda}^2}
		\left( \left(\frac{(x-\lambda)^2}{2} + \abs{y}|x-\lambda| + \frac{y^2}{2} \right) 
		\norme{g''}_{[-M,M],\infty}
		+ C \abs{y}^{p\vee 3} \right),
	\end{align*}
	which yields $\abso{f_{1,z}'(\lambda)} \leq C$ as before.
	
	\textbf{Case $k=2$.} Here, we decompose
	\begin{align*}
		f_{1,z}''(\lambda) 
		&= \frac{2}{(z-\lambda)^3} \left(g(z)-g(\lambda)-(z-\lambda)g'(\lambda) - \frac{(z-\lambda)^2}{2}g''(\lambda) \right)
		= \frac{2}{(z-\lambda)^3}(\mathrm{I}+\mathrm{II}+\mathrm{III}+\mathrm{IV}),
	\end{align*}
	with
	\begin{align*}
		\mathrm{I} & \coloneqq g(x)-g(\lambda)-(x-\lambda)g'(\lambda)-(x-\lambda)^2g''(\lambda)/2, \\
		\mathrm{II} & \coloneqq
		\ii y\big( g'(x)-g'(\lambda) - (x-\lambda)g''(\lambda) \big), \\
		\mathrm{III} & \coloneqq 
		- y^2 \left( g''(x)-g''(\lambda) \right) /2, \\
		\mathrm{IV} & \coloneqq - \ii y^3 g'''(x) \chi(y/x)/6.
	\end{align*}
	By Lemma \ref{lem:borneABC}, the first three terms are bounded in absolute value by
	\begin{equation} \label{eq:bound_k=2}
		C\abs{z-\lambda}^3 (\abs{x} \vee \abs{\lambda})^{p-3}
		\leq C\abs{z-\lambda}^3 (\abs{x} \vee \abs{\lambda})^{-(3-p)_+},
	\end{equation}
	bounding $(\abs{x} \vee \abs{\lambda})^{p-3} \leq C$ with $C$ depending on $K$ when $p \geq 3$.
	It remains to bound $\mathrm{IV}$. 
	When $p\geq 3$, by \eqref{eq:last_term_g}, we have $\abs{\mathrm{IV}} \leq Cy^3 \leq C \abs{z-\lambda}^3$. 
	Now assume $p<3$. We distinguish two cases. If $\abs{\lambda} \leq 2 \abs{x}$, then we bound $\chi \leq 1$ and get $\abs{\mathrm{IV}} \leq C\abs{y}^3 \abs{x}^{p-3}$, which is itself bounded by \eqref{eq:bound_k=2} in that case.
	On the other hand, if $\abs{\lambda} > 2 \abs{x}$, then $\abs{z-\lambda} \geq \abs{x-\lambda} \geq \abs{\lambda}/2$, so using again \eqref{eq:last_term_g}, we have $\abs{\mathrm{IV}} \leq Cy^p \leq C \abs{z-\lambda}^3 \abs{\lambda}^{p-3}$, which is also bounded by \eqref{eq:bound_k=2} in that case.
	Combining these yields $\abso{f_{1,z}''(\lambda)} \leq C
	(\abs{x} \vee \abs{\lambda})^{-(3-p)_+}$ as desired.
	
	\textbf{Case $k=3$.} We assume here that $\lambda\neq 0$. Using the same bound as for $k=2$ for the first terms and $\abs{g'''(\lambda)} \leq C \abs{\lambda}^{-(3-p)_+}$ for the last term, we get
	\begin{align*}
		\abs{f_{1,z}'''(\lambda)}
		&= \frac{6}{\abs{z-\lambda}^4} 
		\abs{ \left( g(z) -  g(\lambda)-(z-\lambda)g'(\lambda) - \frac{(z-\lambda)^2}{2} g''(\lambda) \right) 
			- \frac{(z-\lambda)^3}{6} g'''(\lambda) } \\
		& \leq C \frac{\abs{\lambda}^{-(3-p)_+}}{\abs{z-\lambda}}.
	\end{align*}
	If $\abso{x-\lambda} \geq \abs{x}/2$, then $\abs{z-\lambda} \geq \abs{z}/2$, so we get $\abso{f_{1,z}'''(\lambda)} \leq C\abs{\lambda}^{p-3}/\abs{z}$.
	If $\abs{x} \leq 2 \abs{y}$, then $\abs{z-\lambda} \geq \abs{y} \geq \abs{z}/\sqrt{5}$ and we also get the desired result.
	It remains to cover the case where $\abs{\lambda-x} < \abs{x}/2$ and $2\abs{y} < \abs{x}$. Here we decompose differently:
	\begin{align*}
		f_{1,z}'''(\lambda) 
		= \frac{6}{(z-\lambda)^4}
		(\mathrm{I'}+\mathrm{II'}+\mathrm{III'}+\mathrm{IV'}),
	\end{align*}
	with
	\begin{align*}
		\mathrm{I'} & \coloneqq g(x) - g(\lambda) - (x-\lambda)g'(\lambda) - (x-\lambda)^2 g''(\lambda)/2 - (x-\lambda)^3 g'''(\lambda)/6, \\
		\mathrm{II'} & \coloneqq
		\ii y \left( g'(x) - g'(\lambda) - (x-\lambda) g''(\lambda) - (x-\lambda)^2 g'''(\lambda)/2 \right), \\
		\mathrm{III'} & \coloneqq 
		- y^2/2 \left( g''(x)-g''(\lambda) - (x-\lambda) g'''(\lambda) \right), \\
		\mathrm{IV'} & \coloneqq 
		-\ii y^3 (g'''(x)-g'''(\lambda))/6,
	\end{align*}
	where we used that $\chi(y/x) = 1$ when $2\abs{y} < \abs{x}$.
	The condition $\abs{\lambda-x} < \abs{x}/2$ yields that, on the interval between $\lambda$ and $x$, $\lvert g^{(4)} \rvert \leq C \abs{x}^{p-4}$ and, using also that $2 \abs{y} < \abs{x}$ implies $\abs{x} \geq \abs{z}/2$, we get $\lvert g^{(4)} \rvert  \leq C \abs{\lambda}^{p-3}/\abs{z}$ on this interval.
	Therefore, applying Taylor--Lagrange inequality to the four terms, we get
	\begin{align*}
		\abs{f_{1,z}'''(\lambda)}
		= \frac{C}{\abs{z-\lambda}^4} 
		\left( (x-\lambda)^4 
		+ \abs{y} \abs{x-\lambda}^3 
		+ y^2 (x-\lambda)^2 + \abs{y}^3 \abs{x-\lambda} \right) \frac{\abs{\lambda}^{p-3}}{\abs{z}}
		\leq C \frac{\abs{\lambda}^{(3-p)_+}}{\abs{z}},
	\end{align*}
	which concludes the proof.
\end{proof}

We now deal with the proof of Lemma \ref{lem:psi_z}.
We first study the following function 
\begin{equation} \label{eq:def_varphi_z}
	\varphi_{\alpha,z}(\lambda)
	\coloneqq \int_{-1}^1 \frac{f_{\alpha,z}(t)-f_{\alpha,z}(\lambda)}{t-\lambda} \frac{\diff t}{\sigma(t)},
	\qquad \lambda \in \R.
\end{equation}

\begin{lemma} \label{lem:varphi_z}
    Let $p>2$.
	For any $\alpha \in [0,1]$ and $z = x +\ii y \in \C \setminus\R$, $\varphi_{\alpha,z} \in \cC^2(\R)$.
	Moreover, for any compact set $K \subset \C$, there exists $C>0$ such that, for any $\alpha \in [0,1]$, $z \in K\setminus\R$ and $\lambda \in [-1,1]$,
	\[
	\abs{\varphi_{\alpha,z}(\lambda)} \leq C, \qquad
	\abs{\varphi_{\alpha,z}'(\lambda)} \leq C, \qquad \abs{\varphi_{\alpha,z}''(\lambda)} \leq \frac{C}{\abs{z}}.
	\]
\end{lemma}

\begin{proof}
	From Lemma~\ref{lem:f_z}, we know that $f_{\alpha,z}\in\cC^2(\R)$ and that $\abs{f_{\alpha,z}} \leq C$ and $\abso{f_{\alpha,z}'} \leq C$ on $[-1,1]$. 
	Hence, by Lemma~\ref{lem:Theta}.\ref{it:lem_Theta_1}, we deduce that $\varphi_{\alpha,z} \in \cC^1(\R)$, with $\abs{\varphi_{\alpha,z}} \leq C$ on $[-1,1]$ and, for $\lambda \in [-1,1]$,
	\begin{equation*}
		\varphi_{\alpha,z}'(\lambda)
		= \int_{-1}^1\frac{f_{\alpha,z}(t)- T_1[f_{\alpha,z},\lambda](t)}{(t-\lambda)^2} \frac{\diff t}{\sigma(t)}
		= \int_{-1}^1 \frac{1}{(t-\lambda)^2}
		\left( \int_\lambda^t (t-u) f_{\alpha,z}''(u) \diff u \right)
		\frac{\diff t}{\sigma(t)},
	\end{equation*}
	using Taylor integral formula.
	Bounding $\abso{f_{\alpha,z}''(u)} \leq C \abs{u}^{-(3-p)_+}$ by Lemma~\ref{lem:f_z} and $\abs{t-u} \leq \abs{t-\lambda}$, we get
	\begin{equation} \label{eq:before}
		\abs{\varphi_{\alpha,z}'(\lambda)} 
		\leq C \int_{-1}^1 \frac{1}{\abs{t-\lambda}} 
		\left( \int_{\lambda\wedge t}^{\lambda\vee t} \abs{u}^{-(3-p)_+} \diff u \right) 
		\frac{\diff t}{\sigma(t)}
		\leq C \int_{-1}^1 \abs{t}^{-(3-p)_+}
		\frac{\diff t}{\sigma(t)},
	\end{equation}
	by applying \eqref{eq:int_u_gamma-3} if $p<3$ and computing explicitly the integral w.r.t.\@ $u$ if $p \geq 3$. This proves 
	$\abso{\varphi_{\alpha,z}'(\lambda)} 
	\leq C$ because $p>2$. 
	
	We now turn to the second derivative. Firstly, since $f_{\alpha,z}\in\cC^3(\R^*)$ by Lemma~\ref{lem:f_z}, we have $\varphi_{\alpha,z} \in \cC^2(\R^*)$ by Lemma~\ref{lem:Thetav2}.\ref{it:lem_Thetav2_1} and, for any $\lambda \in \R^*$,
	\begin{align*}
		\varphi_{\alpha,z}''(\lambda) &= \int_{-1}^1 \frac{f_{\alpha,z}(t)- T_2[f_{\alpha,z},\lambda](t)}{(t-\lambda)^3}
		\frac{\diff t}{\sigma(t)} 
		= \int_{-1}^1 \frac{1}{(t-\lambda)^3}
		\left( \int_{\lambda}^t \frac{(t-u)^2}{2}f_{\alpha,z}'''(u) \diff u \right) 
		\frac{\diff t}{\sigma(t)},
	\end{align*}
	by Taylor integral formula again.
	Then, for $\lambda \in [-1,1] \setminus \{0\}$, bounding $\abso{f_{\alpha,z}'''(u)} \leq C\abs{u}^{-(3-p)_+}/\abs{z}$ by Lemma \ref{lem:f_z} and $(t-u)^2 \leq (t-\lambda)^2$, and then proceeding as before in \eqref{eq:before}, we get $\abso{\varphi_{\alpha,z}''(\lambda)} \leq C/\abs{z}$ as desired.
	The same inequalities provides the domination needed to prove convergence of $\varphi_{\alpha,z}''(\lambda)$ as $\lambda \to 0$ from the dominated convergence theorem, proving $\varphi_{\alpha,z} \in \cC^2(\R)$ and concluding the proof.
\end{proof}

\begin{proof}[Proof of Lemma \ref{lem:psi_z}]
	The proof follows from the same lines as Lemma~\ref{lem:inverse_master_op}.\ref{it:inverse_C3}, even if here the function $f_{\alpha,z}$ is not in $\cC^3(\R)$, but only in $\cC^2(\R)\cap\cC^3(\R^*)$. This regularity at 0 is only used in the proof of Lemma~\ref{lem:inverse_master_op} when controlling $\normeo{\varphi_{\alpha,z}''}_{\cC^2}$, but this has been done in Lemma~\ref{lem:varphi_z} above.
	The rest of the proof is identical replacing occurrences of $\norme{f}_{\cC^3}$ by $C/\abs{z}$.
\end{proof}

\section{Proofs of results concerning the equilibrium Stieltjes transform}
\label{app:eq_stieltjes}

This section contains 
the proof of Lemmas \ref{lem:stieltjes} and \ref{lem:stability}, as well as some intermediary results.

\subsection{Identities for the equilibrium Stieltjes transform}

\begin{proof}[Proof of Lemma~\ref{lem:stieltjes}]
	\textbf{Part \ref{lemsta:point1}.} 
	By differentiating \eqref{eq:characterization muValpha}, we have for $x\in (-1,1)$
	\begin{equation} \label{eq:equilibrium}
		\frac{V_\alpha'(x)}{2}=\fint_{-1}^1 \frac{\diff \mu_{V_\alpha}(t)}{x-t},
	\end{equation}
	and, for $z\in \C\setminus\R$, we find by integrating \eqref{eq:equilibrium} against $\frac{1}{x-z} \diff \mu_{V_\alpha}(x)$
	\[
	\frac{1}{2}\int_{-1}^1 \frac{V_\alpha'(x)}{x-z}\diff \mu_{V_\alpha}(x) = \int_{-1}^1 \fint_{-1}^1 \frac{\diff \mu_{V_\alpha}(t)}{x-t} \frac{\diff \mu_{V_\alpha}(x)}{x-z} = -\frac{1}{2}\iint_{[-1,1]^2}\frac{\diff \mu_{V_\alpha}(t)\diff \mu_{V_\alpha}(x)}{(t-z)(x-z)} = -\dfrac{1}{2}s_{V_\alpha}(z)^2,
	\]
	where the second equality is obtained by symmetrizing the roles of $x$ and $t$.
	Therefore,
	\[
	\int_{-1}^1 \frac{V_\alpha'(x)}{x-z}\diff \mu_{V_\alpha}(x) + s_{V_\alpha}(z)^2 = 0.
	\]
	Now adding to this equation the equality
	\[
	\frac{1}{z}\int_{-1}^1 V_\alpha'(x) \diff \mu_{V_\alpha}(x) = 0,
	\]
	obtained by integrating \eqref{eq:equilibrium} against $\diff \mu_{V_\alpha}(x)$, we find that
	\begin{equation*}
		0=\dfrac{1}{z}\int_{-1}^1 \frac{xV_\alpha'(x)}{x-z}\diff \mu_{V_\alpha}(x) + s_{V_\alpha}(z)^2 = \dfrac{1}{z}\int_{-1}^1 \frac{xV_\alpha'(x)-g_\alpha(z)}{x-z}\diff \mu_{V_\alpha}(x) +\dfrac{g_\alpha(z)}{z}s_{V_\alpha}(z) + s_{V_\alpha}(z)^2.
	\end{equation*}
	Finally, recalling that $g_\alpha(x)=xV'_\alpha(x)$ and the definition of $h_\alpha$ in Definition~\ref{def:g,g analytic, f_z, Palpha, halpha} proves Part \ref{lemsta:point1}.
	
	\textbf{Part \ref{lemsta:point2}.}
	Let $x\in \R^*$. Using the definition of $r_\alpha$ in \eqref{def:r_generale} and that $\frac{1}{x-t} = \frac{t}{x(x-t)}+\frac{1}{x}$, we get 
	\begin{align*}
		r_\alpha(x)=\frac{1}{2\pi}\int_{-1}^1\frac{V_\alpha'(x)-V_\alpha'(t)}{x-t}\frac{\diff t}{\sigma(t)} &=\frac{1}{2\pi}\int_{-1}^1 \left(\frac{V_\alpha'(x)}{x-t}-\frac{tV_\alpha'(t)}{x(x-t)}-\frac{V_\alpha'(t)}{x}\right)\frac{\diff t}{\sigma(t)} \\
		&=\frac{1}{2\pi x}\int_{-1}^1 \frac{g_\alpha(x)-g_\alpha(t)}{x-t}\frac{\diff t}{\sigma(t)},
	\end{align*}
	where in the last equality we used that $xV_\alpha'(x)=g_\alpha(x)$ for real $x$, and that $V_\alpha'/\sigma$ is odd so that $\int_{-1}^1 V_\alpha'(t)\frac{\diff t}{\sigma(t)}=0$. This shows that \eqref{eq:r'} is indeed an extension of the definition of $r_\alpha$ in~\eqref{def:r_generale}.

	\textbf{Part \ref{lemsta:point3}.}
	Since $z\in \C\setminus[-1,1]$, we can use \eqref{eq:r'} and separate $r_\alpha$ into
	\begin{equation} \label{eq:r_alpha_split}
		r_\alpha(z) 
		= \frac{1}{2\pi z}\int_{-1}^1 \frac{g_\alpha(t)}{t-z}\frac{\diff t}{\sigma(t)} - \frac{g_\alpha(z)}{2\pi z}\int_{-1}^1 \frac{1}{t-z}\frac{\diff t}{\sigma(t)}.
	\end{equation}
	Using $g_\alpha(t) = t V_\alpha'(t)$, we rewrite the first term as follows
	\[
	\frac{1}{2\pi z} \int_{-1}^1 \frac{g_\alpha(t)}{t-z}\frac{\diff t}{\sigma(t)}
	= \frac{1}{2\pi} \int_{-1}^1 \left( \frac{V_\alpha'(t)}{t-z} + \frac{V_\alpha'(t)}{z} \right)\frac{\diff t}{\sigma(t)}
	= \frac{1}{2\pi} \int_{-1}^1 \frac{V_\alpha'(t)}{t-z} \frac{\diff t}{\sigma(t)},
	\]
	using that $V_\alpha'/\sigma$ is an odd function.
	The second term on the right-hand side of \eqref{eq:r_alpha_split} is proportional to the Stieltjes transform of the arcsine law:
	\begin{equation} \label{eq:arcsine_stieltjes}
		\frac{1}{\pi}\int_{-1}^1 \frac{1}{t-z}\frac{\diff t}{\sigma(t)} 
		= \frac{-1}{\sqrt{z-1}\sqrt{z+1}}.
	\end{equation}
	Therefore, coming back to \eqref{eq:r_alpha_split} and recalling the notation $b(z) = \sqrt{z-1}\sqrt{z+1}$, we obtain
	\[
	r_\alpha(z) b(z)
	= \frac{b(z)}{2\pi}\int_{-1}^1 \frac{V_\alpha'(t)}{t-z}\frac{\diff t}{\sigma(t)} + \frac{g_\alpha(z)}{2z}.
	\]
	Hence, it is now sufficient to show that
	\begin{equation} \label{eq:point3_new_goal}
		s_{V_\alpha}(z) = \frac{b(z)}{2\pi}\int_{-1}^1 \frac{V_\alpha'(t)}{t-z}\frac{\diff t}{\sigma(t)}.
	\end{equation}
	Recalling \eqref{eq:mu=rsigma}, we have 
	\begin{align*}
		s_{V_\alpha}(z)
		&=\frac{1}{\pi}\int_{-1}^1 \frac{\sigma(t)r_\alpha(t)}{t-z}\diff t
		=\frac{1}{2\pi^2}\int_{-1}^1 \frac{\sigma(t)}{t-z}
		\left(\fint_{-1}^1 \frac{V_\alpha'(u)}{u-t}
		\frac{\diff u}{\sigma(u)}\right)\diff t,
	\end{align*}
	where we used the definition of $r_\alpha(t)$ in \eqref{def:r_generale}, together with the fact that $\fint_{-1}^1 \frac{1}{t-u}\frac{\diff u}{\sigma(u)}=0$ for $t\in (-1,1)$.
	Therefore, switching the order of the integrals (see \cite[Eq.\@ (23.18)]{Mus1972}), we get
	\[
	s_{V_\alpha}(z)
	= \frac{1}{2\pi^2}\int_{-1}^1 \frac{V_\alpha'(u)}{\sigma(u)}\left(\fint_{-1}^1 \frac{\sigma(t)}{(t-z)(u-t)}\diff t\right) \diff u.
	\]
	Moreover, by direct computation,
	\[
	\fint_{-1}^1 \frac{\sigma(t)}{(t-z)t(u-t)}\diff t 
	= \frac{1}{u-z} \left( 
	\int_{-1}^1 \frac{\sigma(t)\diff t}{t-z}
	+ \fint_{-1}^1 \frac{\sigma(t)\diff t}{u-t} \right)
	= \frac{1}{u-z} \left( 
	\pi\left(b(z)-z\right) + \pi u \right)
	\]
	and \eqref{eq:point3_new_goal} follows using again that $\int_{-1}^1 V_\alpha'(u) \frac{\diff u}{\sigma(u)} = 0$.
	
	\textbf{Part \ref{lemsta:point4}.} This point comes from the general fact that if $\lambda = a-b/2$ solves the quadratic equation $X^2+bX+c$, where $a,b,c \in \C$, then the other solution is given by $\mu=-a-b/2$.
\end{proof}

\subsection{Proof of the stability lemma}
\label{subsec:stability}

The main goal of this section is to prove the stability lemma (Lemma~\ref{lem:stability}). The difference with the proof of the stability lemma in \cite{BouModPai2022} is the lack of analyticity of $V_\alpha$ and hence of $r_\alpha$.
We start with the following regularity bounds on $r_\alpha$, which would be obvious if $r_\alpha$ analytic.

\begin{lemma}\label{lem:estimees_r} Let $p>2$.
	\begin{enumerate}
		\item\label{it:r_alpha_1} For every compact $K\subset\C$, there exists $C>0$ such that, for any $\alpha \in [0,1]$ and $z \in K$, 
		\[
		\abs{\im r_\alpha(z)} \leq C \abs{\im z}.
		\]
		\item\label{it:r_alpha_2} The function $r_\alpha$ is continuous on $\C$ uniformly in $\alpha\in [0,1]$, \emph{i.e.\@} for all $\varepsilon>0$ and $z\in\C$, there exists $\delta>0$ such that for all $\alpha\in [0,1]$, $\forall z'\in\C$, $|z-z'|<\delta\Rightarrow|r_\alpha(z)-r_\alpha(z')|<\varepsilon$.
		\item\label{it:r_alpha_3} For any compact $K\subset \C$, there exist $c, \delta>0$ such that for any $z\in K$ with $\abs{\im{z}} \leq \delta$, and for any $\alpha\in [0,1]$ one has $\re r_\alpha(z)\geq c$.
	\end{enumerate}
\end{lemma}

\begin{proof}
	First note that $r_\alpha = \alpha r_1 + 2(1-\alpha)$, so it is enough to prove the result for $\alpha = 1$. For brevity, we denote $r = r_1$.
	
	\textbf{Part \ref{it:r_alpha_1}.}
	Write $z = x + \ii y$. By symmetry, we can assume that $x,y\geq 0$. Moreover, if $y=0$, then $\im r(z) = 0$ so we can assume that $y > 0$.
	Recall the definition of $g(z)$ in \eqref{eq:def_g_C}.
	In order to treat the last term in $g(z)$ separately, we define
	\[
	\forall  z = x+\ii y \in \C, \quad
	\widetilde{g}(z) 
	\coloneqq g(x) + \ii y g'(x) - \frac{y^2}{2} g''(x) 
	= g(z) + \ii \frac{y^3}{6} g'''(x) \chi \left( \frac{y}{x} \right).
	\]
	Recalling the definition of $r$ on $\C\setminus \R$ in \eqref{eq:r'}, we first cut $r$ into three parts that will be treated differently: $r=r^{(1)}+r^{(2)}+r^{(3)}$, with
	\begin{align}
		\label{eq:def_r_1}
		r^{(1)}(z) 
		& \coloneqq  \frac{1}{2\pi z} \int_{[-1,1] \cap [-2x,2x]} \frac{g(t)-\widetilde{g}(z)}{t-z} \frac{\diff t}{\sigma(t)}, \\
		\nonumber r^{(2)}(z) 
		& \coloneqq  \frac{1}{2\pi z} \int_{[-1,1] \setminus [-2x,2x]} \frac{g(t)-\widetilde{g}(z)}{t-z} \frac{\diff t}{\sigma(t)}, \\
		\nonumber r^{(3)}(z) 
		& \coloneqq  \frac{1}{2\pi z} \int_{[-1,1]} \dfrac{y^3}{6}\frac{\ii  g'''(x) }{t-z}\chi \left( \frac{y}{x} \right) \frac{\diff t}{\sigma(t)}.
	\end{align}
	
	We start with $r^{(1)}(z)$. 
	We write $I = [-1,1] \cap [-2x,2x]$ for the domain of integration. 
	By definition of $\widetilde{g}(z)$, we have
	\begin{align}\nonumber
		\frac{g(t)-\widetilde{g}(z)}{t-z}
		& = \frac{1}{t-z} \left( g(t) - g(x) - \ii y g'(x) + \frac{y^2}{2} g''(x) \right) \\\label{eq: app decomp r1}
		& = \frac{1}{t-z} \left( g(t) - g(x) - (t-x) g'(x) - \frac{(t-x)^2}{2} g''(x) \right) 
		+ g'(x) + \frac{t-\overline{z}}{2} g''(x), 
	\end{align}
	where we used $((t-x)^2+y^2)/(t-z) = t-\overline{z}$.
	Using \eqref{eq:g_Taylor_order_3} together with $\abs{t} \leq 2x$ for the first term on the right-hand side, we get
	\begin{align*}
		\abs{\im r^{(1)}(z)}
		& \leq C \int_{I} \left(
		\abs{t-x}^3 x^{p-3} \abs{\im \frac{1}{z(t-z)}}
		+ x^{p-1} \abs{\im \frac{1}{z}}
		+ x^{p-2} \abs{\im \frac{t-\overline{z}}{z}}
		\right)\frac{\diff t}{\sigma(t)} \\
		& = \frac{Cy}{\abs{z}^2} \int_{I} \left(
		\abs{t-x}^3 x^{p-3} \frac{\abs{2x-t}}{\abs{t-z}^2}
		+ x^{p-1} 
		+ x^{p-2} \abs{2x-t}
		\right)\frac{\diff t}{\sigma(t)},
	\end{align*}
	by computing explicitly the imaginary parts. Recalling $I \subset[-2x,2x]$, the last integral is then bounded by $C \int_{I} x^{p-1} \frac{\diff t}{\sigma(t)} \leq Cx^p$.
	Using $p\geq2$ and $x \leq \abs{z}$ yields $\abso{\im r^{(1)}(z)} \leq C y$.
	
	We now deal with $r^{(2)}(z)$. 
	We write $J = [-1,1] \setminus [-2x,2x]$ for the domain of integration. 
	Note that this interval is empty if $x > 1/2$, so we can further assume that $x \in [0,1/2]$.
	By parity, $\int_J \frac{g(t)}{t} \frac{\diff t}{\sigma(t)} = 0$ so we can write
	\begin{equation} \label{eq:rewriting_r2}
		r^{(2)}(z) 
		= \frac{1}{2\pi z} \int_{J} 
		\left( \frac{g(t)-\widetilde{g}(z)}{t-z} - \frac{g(t)}{t} \right) 
		\frac{\diff t}{\sigma(t)}
		=  \int_{J} \frac{1}{t-z}
		\left( \frac{g(t)}{t} - \frac{\widetilde{g}(z)}{z} \right) 
		\frac{\diff t}{\sigma(t)}.
	\end{equation}
	Then, using the definition of $\widetilde{g}(z)$ and then that $g(t)/t = V'(t)$, we decompose
	\begin{align*}
		\frac{g(t)}{t} - \frac{\widetilde{g}(z)}{z}
		& = \frac{g(t)}{t} - 
		\frac{g(x) + \ii y g'(x)-\frac{y^2}{2} g''(x)}{z}  \\
		& = \left( V'(t) - V'(x) - \ii yV''(x) \right)
		+ \left( V'(x) - \frac{g(x)}{z} + \ii y \left( V''(x) - \frac{g'(x)}{z} \right) -\frac{y^2}{2} \frac{g''(x)}{z} \right),
	\end{align*}
	and treat the two terms on the right-hand side differently.
	For the first term, note that 
	\begin{equation*}
		\abs{V'(t) - V'(x) - \ii yV''(x) - (t-z)V''(x)}
		= \abs{V'(t) - V'(x) - (t-x)V''(x)}
		\leq C \abs{t-x}^2 \abs{t}^{p-3},
	\end{equation*}
	using \eqref{eq:g'_Taylor_order_2} together with $x \leq \abs{t}$ for $t \in J$.
	Therefore, the contribution of this first term to $\im r^{(2)}(z)$ is
	\begin{align}
		\abs{\im \int_{J} \frac{V'(t) - V'(x) - \ii yV''(x)}{t-z}
			\frac{\diff t}{\sigma(t)} }
		& \leq \abs{\im \int_{J} V''(x) \frac{\diff t}{\sigma(t)} } 
		+ C \int_{J} \abs{t-x}^2 \abs{t}^{p-3} \abs{\im \frac{1}{t-z}} \frac{\diff t}{\sigma(t)}\nonumber \\
		& \leq 0+C y \int_{J} \frac{\abs{t-x}^2}{\abs{t-z}^2} \abs{t}^{p-3} \frac{\diff t}{\sigma(t)}
		\leq Cy, \label{eq:r_2_1st_part}
	\end{align}
	using $p>2$.
	For the second term, using that $V'(x) - \frac{g(x)}{z} = g(x) (\frac{1}{x} - \frac{1}{z}) = p c_p x^{p-1} \frac{\ii y}{z}$ and bounding the imaginary part by the absolute value, we get
	\begin{align}
		& \abs{\im \int_{J} \frac{1}{t-z} 
			\left( V'(x) - \frac{g(x)}{z} + \ii y \left( V''(x) - \frac{g'(x)}{z} \right) -\frac{y^2}{2} \frac{g''(x)}{z} \right) 
			\frac{\diff t}{\sigma(t)} } \nonumber \\
		& \leq C 
		\left( \frac{x^{p-1} y}{\abs{z}} + y \left( x^{p-2} + \frac{x^{p-1}}{\abs{z}} \right) + \frac{y^2}{2} \frac{x^{p-2}}{\abs{z}}\right)
		\int_{J} \frac{1}{\abs{t-z}} 
		\frac{\diff t}{\sigma(t)}. \label{eq:r_2_2nd_part}
	\end{align}
	Then, note that 
	\begin{equation} \label{eq:log_on_J}
		\int_{J} \frac{1}{\abs{t-z}} 
		\frac{\diff t}{\sigma(t)} \leq 2 \int_{2x}^1 \frac{1}{\abs{t-x}} 
		\frac{\diff t}{\sigma(t)} \leq C \log 1/x,
	\end{equation}
	using here that $x \in [0,1/2]$ so that the pole at $x$ and the pole at 1 due to $\sigma(t)$ are bounded away from each other.
	So the right-hand side of \eqref{eq:r_2_2nd_part} is at most $C y x^{p-2} \log 1/x \leq Cy$, because $p>2$.
	Combining this with \eqref{eq:r_2_1st_part} shows $\abso{\im r^{(2)}(z)} \leq C y$. 
	
	Finally, we bound $r^{(3)}(z)$.
	Using the formula \eqref{eq:arcsine_stieltjes} to compute explicitly the integral and applying \eqref{eq:last_term_g}, we get
	\begin{equation} \label{eq:r3}
		\abs{r^{(3)}(z)}
		\leq \frac{C y^{p\wedge 3}}{\abs{z} \abs{z^2-1}^{1/2}}
		\leq C y^{(p-1)\wedge 2},
	\end{equation}
	where we noted that $\abs{z} \abs{z^2-1}^{1/2} \geq c y$ by distinguishing whether $z$ is closer to $-1$, $0$ or $1$.
	Since $p \geq 2$, this yields $\abso{\im r^{(3)}(z)} \leq C y$ and concludes the proof of Part \ref{it:r_alpha_1}.
	
	\textbf{Part \ref{it:r_alpha_2}.}
	First recall that $g$ is continuous on $\C$ (see Remark \ref{rem:g}). Continuity of $r$ at a point $z_0 \in \C^*$ then follows by dominated convergence theorem, using the following domination: for all $t \in [-1,1]$ and $\abs{z-z_0} \leq 1$,
	\[
	\abs{\frac{g(t)-g(z)}{t-z}} = \abs{f_{1,z}(t)} 
	\leq C,
	\]
	by Lemma \ref{lem:f_z}.
	
	Continuity at 0 is more subtle due to the $1/z$ prefactor.
	We reuse the decomposition of the proof of Part~\ref{it:r_alpha_1}, but control the pieces in terms of modulus instead of imaginary part. We need less precision in the bounds since we only look for continuity.
	Again, by symmetry we can assume $x,y\geq 0$.
	Taking back $r^{(1)}(z)$ from \eqref{eq:def_r_1}, and using the same decomposition \eqref{eq: app decomp r1} as for Part~\ref{it:r_alpha_1}, we have for $t \in [-2x,2x]$
	\begin{equation*} 
		\frac{1}{\abs{z}}\abs{\frac{g(t)-\widetilde{g}(z)}{t-z}}
		\leq C\left(\frac{1}{\abs{z}\abs{t-z}}\abs{t-x}^3 x^{p-3} + \frac{x^{p-1}}{\abs{z}} 
		+ \frac{\abs{t-\overline{z}}}{2\abs{z}} x^{p-2}\right)
		\leq C x^{p-2},
	\end{equation*}
	where in the second inequality we used that
	$\abs{t-x}^3 \leq \abs{t-z} \cdot 3x \cdot 3 \abs{z}$
	for the first term, and
	$\abs{t-\overline{z}}\leq 2x + \abs{z} \leq 3 \abs{z}$ 
	for the third term. Integrating this bound, we deduce that $r^{(1)}(z) \to 0$ as $z \to 0$. 
	For $r^{(2)}(z)$, we can assume that $x\leq 1/2$, and writing again $J=[-1,1]\setminus [-2x,2x]$, we get from \eqref{eq:rewriting_r2} that
	\[
	r^{(2)}(z) = \frac{1}{2\pi}\int_{J}\frac{g(t)}{t(t-z)}\frac{\diff t}{\sigma(t)} 
	- \frac{\widetilde{g}(z)}{2\pi z}\int_{J}\frac{\diff t}{(t-z)\sigma(t)}.
	\]
	Note that the second term vanishes as $z \to 0$ by \eqref{eq:log_on_J} combined with the fact that $\abs{\widetilde{g}(z)/z} = \abs{g(x)+iyg'(x)-y^2g''(x)/2}/\abs{z}\leq c(x^{p-1} + yx^{p-2})\leq C x^{p-2}$. 
	For the first term, by dominated convergence theorem, we have
	\[
	\frac{1}{2\pi} \int_{J}\frac{g(t)}{t(t-z)}\frac{\diff t}{\sigma(t)} 
	= \frac{1}{2\pi} \int_{-1}^1 \1_{\abs{t}>2x} \frac{g(t)}{t(t-z)} \frac{\diff t}{\sigma(t)}
	\xrightarrow[z \to 0]{} 
	\frac{1}{2\pi} \int_{-1}^1 \frac{g(t)}{t^2} \frac{\diff t}{\sigma(t)}
	= r(0),
	\]
	where the domination follows from the following bound, for $\abs{t}\geq 2x$,
	\begin{align*}
		\abs{\frac{g(t)}{t(t-z)\sigma(t)}} \leq C\frac{\abs{t}^{p-1}}{\abs{t-x}\sigma(t)} \leq C\frac{\abs{t}^{p-2}}{\sigma(t)},
	\end{align*}
	using that $\abs{t-x}\geq \abs{t}-\abs{x}\geq \abs{t}/2$. 
	This proves $r^{(1)}(z) \to r(0)$ as $z \to 0$. 
	Finally, the fact that $r^{(3)}(z) \to 0$ as $z \to 0$ has already been proved in \eqref{eq:r3}.
	
	\textbf{Part \ref{it:r_alpha_3}.} 
	By Lemma \ref{lem:bound on ralpha}, there exists $c_0>0$ such that $\re r = r \geq c_0$ on $\R$. 
	Moreover, by Part~\ref{it:r_alpha_2}, the function $r$ is continuous on the compact set $K \cup \{ x \in \R : \exists y \in \R, x+\ii y \in K \}$, and hence uniformly continuous. Combining these two facts yields the result.
\end{proof}

The proof of the stability lemma in \cite{BouModPai2022} relies on two lemmas. The first one \cite[Lemma B.2]{BouModPai2022} contains bounds on the function $b(z)=\sqrt{z+1}\sqrt{z-1}$, which we can directly use in our context. 
The second one is \cite[Lemma B.3]{BouModPai2022} and it can be stated as follows in our context, the main differences being that $V'(z)$ in \cite{BouModPai2022} is replaced by $g_\alpha(z)/z$ here.
Note that \eqref{eq:Stab4} and \eqref{eq:Stab5} are not used to prove the stability lemma, but are useful for the proof of the local law.

\begin{lemma}\label{lem:prelim_stability}
    Let $p>2$.
	There exist $\delta_0,c,C>0$ such that, for any $\alpha\in[0,1]$ and $z=x+\ii y$ with $\abs{x} \leq 1+\delta_0$ and $y\in(0,\delta_0]$,
	\begin{align}
		\label{eq:lb_r}
		& c \leq \abs{r_\alpha(z)} \leq C
		\\
		\label{eq:lb_Im(rb)}
		& \im(r_\alpha(z) b(z)) \geq \abs{\im\left(\frac{g_\alpha(z)}{z}\right)}
	\end{align}
	Moreover, if $x\in [-1-y,1+y]$, 
	\begin{equation}
		\im(r_\alpha(z) b(z)) \geq c \abs{r_\alpha(z) b(z)}
		\label{eq:lb_Im(rb)_trapez}
	\end{equation}
	and, if $x\notin[-1-y,1+y]$, 
	\begin{align}
		\label{eq:Stab4}
		& \abs{\im s_{V_\alpha}(z)} \vee \abs{\im \widetilde{s}_{V_\alpha}(z)}
		\leq \frac{Cy}{\abs{b(z)}}
		\\
		\label{eq:Stab5}
		& \re (s_{V_\alpha}(z_0)-\widetilde{s}_{V_\alpha}(z))
		\geq c \abs{r_\alpha(z)b(z)}
	\end{align}
	where $z_0\coloneqq 1+y$ if $x>1+y$, and $z_0\coloneqq -1-y$ if $x<-1-y$.
\end{lemma}

\begin{proof}
	We highlight below the modifications needed in the proof of \cite[Lemma B.3]{BouModPai2022}.
	
	Firstly, by Lemma \ref{lem:estimees_r}, there exist $\delta_0,c,C>0$ such that, for any $\alpha\in[0,1]$ and $z=x+\ii y$ with $\abs{x} \leq 1+\delta_0$ and $y\in(0,\delta_0]$,
	\begin{equation} \label{eq:input_r}
		\abs{\im r_\alpha(z)} \leq Cy, \qquad 
		\re r_\alpha(z) \geq c, \qquad
		\abs{r_\alpha(z)} \leq C.
	\end{equation}
	In particular, this implies \eqref{eq:lb_r}.
	Secondly, we prove that there exists $C>0$ such that for any $z=x+iy$ with $y> 0$, 
	\begin{equation} \label{eq:estimee1}
		\abs{\im\left(\frac{g_\alpha(z)}{z}\right)} \leq Cy.
	\end{equation}
	In view of $g_\alpha(z) = \alpha g(z) + 4(1-\alpha)z^2$, it suffices to establish \eqref{eq:estimee1} for $\alpha=1$.
	Recalling the definition of $g(z)$ in \eqref{eq:def_g_C} and bounding the last term in $g(z)$ as in \eqref{eq:last_term_g}, we find
	\begin{align*}
		\abs{\im\left(\frac{g(z)}{z}\right)}
		= \frac{1}{\abs{z}^2} \abs{x \im g(z) - y \re g(z)}
		\leq \frac{C}{\abs{z}^2} \left( y \abs{x}^p + \abs{x} y^{p\wedge 3} + y \abs{x}^p + y^3 \abs{x}^{p-2} \right)
		\leq Cy, 
	\end{align*}
	using $p\geq 2$ and that $x$ and $y$ are in a compact set, proving \eqref{eq:estimee1}. 
	The inequalities \eqref{eq:input_r} and \eqref{eq:estimee1} are exactly the inputs which are used in \cite{BouModPai2022} to prove \eqref{eq:lb_Im(rb)}, \eqref{eq:lb_Im(rb)_trapez} and \eqref{eq:Stab4}, so the rest of the proof is identical (recall that our function $b$ is the same as theirs).
	
	Lastly, the proof of \eqref{eq:Stab5} requires slight modifications compared to the one of \cite[Eq.\@ (B.11)]{BouModPai2022}.
	As argued there, if $x \notin [-1-y,1+y]$, we have
	$\re(r_\alpha(z)b(z)) \geq c\abs{r_\alpha(z)b(z)}$, and therefore, by \eqref{eq:s_formula} and \eqref{def:s_tilde},
	\begin{align}
		\re (s_{V_\alpha}(z_0)-\widetilde{s}_{V_\alpha}(z))
		& = \re\left( \frac{g_\alpha(z)}{2z}-\frac{g_\alpha(z_0)}{2z_0} 
		+ r_\alpha(z)b(z)+r_\alpha(z_0)b(z_0) \right) \nonumber \\
		& \geq c\abs{r_\alpha(z)b(z)}
		-\frac{1}{2} 
		\abs{\frac{g_\alpha(z)}{z}-\frac{g_\alpha(z_0)}{z_0}}.
		\label{eq:eqfinale_stabilite}
	\end{align}
	Since $z\mapsto g_\alpha(z)/z$ is $\cC^1$ on $\C^*$ (as a function of two variables, not as a function of a complex variable), we have, for $z$ in the region considered here,
	\[
	\abs{\frac{g_\alpha(z)}{z}-\frac{g_\alpha(z_0)}{z_0}} \leq C \abs{z-z_0} \leq C \kappa,
	\]
	where $\kappa = \abs{z+1}\vee\abs{z-1}$. 
	On the other hand, we have $\abs{r_\alpha(z)b(z)} \geq c \sqrt{\kappa}$ near $\pm 1$. Therefore, one can choose $\delta_0$ small enough so that, on the right-hand side of \eqref{eq:eqfinale_stabilite}, the first term is at least twice larger as the second one, from which we conclude on \eqref{eq:Stab5}.
\end{proof}

\begin{proof}[Proof of Lemma \ref{lem:stability}]
	The proof is identical to the one of \cite[Lemma~B.1]{BouModPai2022}, with \cite[Lemma~B.3]{BouModPai2022} replaced by Lemma~\ref{lem:prelim_stability} here.
\end{proof}

\bibliographystyle{alpha}
\bibliography{biblio}

\end{document}